\documentclass{article}

\usepackage[a4paper]{geometry}

\usepackage{bm}

\usepackage{amsmath}
\usepackage{amssymb}
\usepackage{amsthm}

\usepackage{tikz}
\usetikzlibrary{matrix,arrows}

\usepackage[all,2cell]{xy}
\UseAllTwocells

\usepackage{pb-diagram}
\usepackage{pb-xy}

\usepackage{bm}

\usepackage{hyperref}
\hypersetup
{
    colorlinks,	%
    citecolor=black,%
    filecolor=black,%
    linkcolor=black,%
    urlcolor=gray
}

\usepackage{color}

\theoremstyle{plain}
  \newtheorem{theorem}{Theorem}[section]
  \newtheorem{corollary}[theorem]{Corollary}
  \newtheorem{lemma}[theorem]{Lemma}
  
  \newtheorem{proposition}[theorem]{Proposition}

\theoremstyle{definition}
  \newtheorem{definition}[theorem]{Definition}
  
  \newtheorem{ex}[theorem]{Example}
  
  \newenvironment{example}{\begin{ex}}{\qed\end{ex}}

\theoremstyle{remark}
  \newtheorem{remark}[theorem]{Remark}

\numberwithin{equation}{section}
  
  \newcommand{\set}[2]{%
  \left\{#1 \mathrel{}\middle|\mathrel{}  #2\right\}
  }

  \newcommand{\quotient}[2]{%
  \left(#1\right)
  \hspace{-4pt}\raisebox{-5pt}{$\bigg/$}\hspace{-2pt}\raisebox{-12pt}{$#2$}%
  }

  \newcommand{\rarrow}[1]{\buildrel #1 \over \longrightarrow}
  \newcommand{\larrow}[1]{\buildrel #1 \over \longleftarrow}

  \newcommand{\R}{\mathbb{R}}

  \newcommand{\N}{\mathbb{N}}
  \newcommand{\Z}{\mathbb{Z}}
  
  \newcommand{\category}[1]{\operatorname{\mathbf{#1}}}
  
  \newcommand{\Cats}{\category{Cat}}
  \newcommand{\Posets}{\category{Poset}}
  \newcommand{\Sets}{\category{Set}}
  \newcommand{\Spaces}{\category{Top}}

  \newcommand{\op}{{\operatorname{\mathrm{op}}}} 
  
  \newcommand{\pr}{\operatorname{\mathrm{pr}}}

  \newcommand{\Sd}{\operatorname{\mathrm{Sd}}}
  
  \newcommand{\inj}{\mathrm{inj}}
  \newcommand{\Conv}{\mathrm{Conv}}

  \newcommand{\Cr}{\operatorname{\mathrm{Cr}}}

  \newcommand{\FP}{\mathrm{FP}}

\title{\textbf{Discrete Morse Theory and Classifying Spaces}} 
\author{Vidit Nanda, Dai Tamaki, and Kohei Tanaka}
\date{}

\begin{document}

\maketitle


\begin{abstract}
 The aim of this paper is to develop a refinement of Forman's discrete
 Morse theory. To an acyclic partial matching $\mu$ on a finite regular
 CW complex $X$, Forman introduced a discrete analogue of gradient
 flows. Although Forman's gradient flow has been proved to be useful in
 practical computations of homology groups, it is not sufficient to
 recover the homotopy type of $X$. Forman also proved the existence of a
 CW complex which is homotopy equivalent to $X$ and whose cells are in
 one-to-one correspondence with the critical cells of $\mu$, but the
 construction is ad hoc and does not have a combinatorial description.  
 By relaxing the definition of Forman's gradient flows, we introduce the
 notion of flow paths, which contains enough information to reconstruct
 the homotopy type of $X$, while retaining a combinatorial description. 
 The critical difference from Forman's gradient flows is the existence
 of a partial order on the set of flow paths, from which a $2$-category
 $C(\mu)$ is constructed. It is shown that the classifying space of
 $C(\mu)$ is homotopy equivalent to $X$ by using homotopy theory of
 $2$-categories. This result can be also regarded as a discrete analogue
 of the unpublished work of Cohen, Jones, and Segal on Morse theory in
 early 90's. 
\end{abstract} 

\tableofcontents

\section{Introduction}
\label{discrete_CJS_intro}

The goal of this paper is to describe a new combinatorial theory of
gradient flow on cell complexes which provides direct insight into
discrete Morse homotopy, and hence to extend Forman's discrete Morse
theory \cite{Forman95,Forman98-2}. 

\subsection{Discrete Morse Theory}

The central objects of study in Forman's adaptation of Morse theory to
CW complexes are discrete Morse functions, which assign (real) values to
cells. 
Every discrete Morse function $f$ on a regular CW complex $X$ imposes a 
partial pairing on its constituent cells, i.e.\ a bijection
$\mu_{f} : D(\mu_{f})\to U(\mu_{f})$ between disjoint subsets
$D(\mu_{f}),U(\mu_{f})$ of the face poset $F(X)$.
The unpaired cells are analogous to critical points from smooth Morse
theory, while the paired cells generate combinatorial gradient paths
between critical ones.
Such gradient paths have been proved to be quite useful in computing
homology and cohomology in various contexts such as topological
combinatorics \cite{math/9705219,Shareshian01}, hyperplane arrangements
\cite{0705.2874,0705.3107,0711.1517}, cohomology of algebraic structures
\cite{Batzies-Welker02,cs.DM/0504090,Skoeldberg06,Joellenbeck-Welker2009,math.GT/0603177}, 
and topological data analysis \cite{Mischaikow-Nanda2013}.

Compared to these nice applications to homology, it is disappointing to
see the failure of the use of Forman's gradient paths to the study of
homotopy types. 
Recall that a gradient path, in the sense of Forman, from a critical
cell $c$ to another $c'$ is an alternating sequence of cells
\[
 c\succ d_1 \prec u_1 \succ \cdots \prec u_{k-1}
 \succ d_{k} \prec u_{k} 
 \succ \cdots \prec u_{n} \succ c',
\]
where $d_{i}\in D(\mu_{f})$ and $u_{i}\in U(\mu_{f})$ with
$\mu_{f}(d_i)=u_i$, and all the face 
relations $\prec$ and $\succ$ appearing in this sequence are of
codimension $1$, which imposes the dimension constraint
$\dim c=\dim c'+1$.
This failure is best illustrated by the following simple example.
Consider the ``height function'' $h$ on the boundary of a $3$-simplex
$\Delta^3=[v_{0},v_{1},v_{2},v_{3}]$, whose associated partial matching
$\mu_{h}$ is 
indicated by arrows in Figure \ref{height_function_on_tetrahedron}. 
\begin{figure}[ht]
 \begin{center}
  \begin{tikzpicture}
   \draw [dotted] (0.5,2.5) -- (0,0);
   \draw (0,0) -- (1,1.5) -- (-1,2) -- (0,0);
   \draw (-1,2) -- (0.5,2.5) -- (1,1.5);
   \draw [fill] (0,0) circle (2pt);
   \draw [fill] (0.5,2.5) circle (2pt);
   \draw [fill] (-1,2) circle (2pt);
   \draw [fill] (1,1.5) circle (2pt);

   
   \draw (0,-0.4) node {$v_{0}$};
   \draw (1.4,1.4) node {$v_{1}$};
   \draw (0.6,2.9) node {$v_{2}$};
   \draw (-1.4,2.1) node {$v_{3}$};


   \draw [->] (1,1.3) -- (0.8,1);
   \draw [->] (0.48,2.3) -- (0.38,1.8);
   \draw [->] (-1,1.8) -- (-0.8,1.4);
   \draw [->] (0,1.75) -- (0,1.4);
   \draw [->] (0.75,2) -- (0.65,1.6);
   \draw [->] (-0.25,2.25) -- (-0.15,1.75);

  \end{tikzpicture}
 \end{center}
 \caption{A partial matching on $\partial\Delta^3$}
 \label{height_function_on_tetrahedron}
\end{figure}
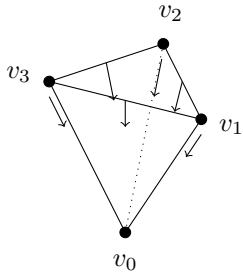
For example, the vertex $[v_{1}]$ is matched with the edge
$[v_{0},v_{1}]$ and the edge $[v_{1},v_{2}]$ is matched with the face
$[v_{0},v_{1},v_{2}]$. The remaining two cells, the bottom vertex
$[v_{0}]$ and the top face $[v_{1},v_{2},v_{3}]$, are critical. 
There is no gradient path between these two critical cells because of
the dimension gap, although there is an apparent deformation indicated
by the arrows, which results in the minimal cell decomposition of a
sphere $S^2=e^0\cup e^2$.

In fact, Forman generalized the deformation in the previous example and
showed that matched pairs of cells can be collapsed together 
without changing the homotopy type. Thus we obtain a new cell complex
$X_{f}$, whose cells are indexed by the set $\Cr(f)$ of critical cells
of $f$. 
However, there is no explicit description of the resulting cell complex
in terms of gradient paths.
In smooth Morse theory, on the other hand, there is a well-known
construction of a cell decomposition of a smooth manifold $M$ from 
gradient flows generated by a Morse-Smale function $f:M\to\R$, whose
cells are indexed by critical points
\cite{Kalmbach1975,Franks79}. 


\subsection{Main Result}

In this paper, we introduce \emph{flow paths} of acyclic partial
matchings. These generalize Forman's gradient paths, and they may be
used to explicitly recover the homotopy type of the original cell
complex.

\begin{definition}
 A \emph{flow path} from a critical cell $c$ to another $c'$ is a
 sequence of cells of the following form
 \begin{multline*}
  c \succ u_1\succ \cdots \succ u_{i_1-1} \succ d_{i_1} 
  \prec \mu_{f}(d_{i_1})=u_{i_1} \succ \cdots  \\
  \succ u_{i_2-1} \succ d_{i_2} \prec \mu_{f}(d_{i_2})=u_{i_2} \succ
  \cdots \succ u_n \succ c'
 \end{multline*}
 or
 \begin{multline*}
  c \succ d_1 \prec \mu_{f}(d_1)=u_1 \succ \cdots \succ u_{i_1-1}
  \succ d_{i_1} \prec \mu_{f}(d_{i_1})=u_{i_1} \succ \cdots \\
  \succ u_{i_2-1} \succ d_{i_2} \prec \mu_{f}(d_{i_2})=u_{i_2} \succ
  \cdots \succ u_n \succ c',   
 \end{multline*}
 where $u_{i}\in U(\mu_{f})$ and $d_{i}\in D(\mu_{f})$.
\end{definition}

There are two differences from Forman's gradient paths. We allow
descending sequences of cells in $U(\mu_{f})$ to appear in
a flow path and the face relations $u_{i_{k}-1}\succ d_{i_{k}}$ have no
codimension restrictions.

This simple extension of gradient paths turns out to contain enough
information to reconstruct the homotopy type. The crucial property is
the existence of a partial order on the set $\FP(\mu_{f})$ of all flow
paths, which induces a partial order on the set $C(\mu_{f})(c,c')$ of
flow paths from $c'$ to $c$.

Recall that there is a standard way of constructing a simplicial
complex $BP$ from a poset $P$, called the \emph{order
complex} \cite{Bjorner95}. Regarding posets as a special class of small
categories, the order complex construction has been extended to small
categories as the \emph{classifying space} construction \cite{SegalBG}. 
By regarding the poset $C(\mu_{f})(c,c')$ as a small category, our flow
category $C(\mu_{f})$ becomes a category enriched over the category of
small categories, i.e.\ a (strict) $2$-category \cite{math.CT/0702535}.
The classifying space construction has been also extended to 
$2$-categories, with which our main theorem can be stated as follows.

\begin{theorem}
 \label{main1}
 Let $X$ be a finite regular CW complex and $f$ a discrete Morse function
 on the face poset $F(X)$.
 Then the classifying space of the flow category $C(\mu_{f})$ is
 homotopy equivalent to $X$.
\end{theorem}

Note that there are several ways to take classifying spaces of
$2$-categories. One of the most popular models is the classifying
space of the topological category obtained by taking the classifying
space of each morphism category. This construction is denoted by $B^2$
in this paper. We also make use of the ``normal colax version''
$B^{ncl}$ in the proof of the theorem. 
It is known that these two, including several other constructions, are
homotopy equivalent to each other. 
This is proved by Carrasco, Cegarra, and Garz{\'o}n in
\cite{0903.5058}. 

The classifying space $B^2C(\mu_{f})$ is described
by the combinatorial data associated to a discrete Morse function or an
acyclic partial matching.
Thus it provides a systematic way of reconstructing the homotopy type
of the original cell complex purely in terms of Morse data.

For example, this description as the classifying space of a $2$-category
or a topological category can serve as an alternative systematic
approach to computing homology via discrete Morse theory. 
Given a multiplicative homology theory $h_*(-)$ satisfying the strong
form of the K{\"u}nneth isomorphism, the homotopy equivalence in Theorem 
\ref{main1} gives rise to a spectral sequence
\[
 E^2 \cong H_*(h_*(BC(\mu_{f}))) \Longrightarrow
 h_*(B^2C(\mu_{f})) 
 \cong h_*(X) 
\]
by \cite{SegalBG}, where $h_*(BC(\mu_{f}))$ is the category
enriched over graded $h_*$-modules whose set of objects is
$\Cr(\mu_{f})$ and whose module
of morphisms from $c$ to $c'$ is $h_*(BC(\mu_{f})(c,c'))$. And
$H_*(-)$ is the homology of linear categories, studied, for example, in
\cite{Watts1966}. 

As a version of Morse theory, it would be desirable to find a cell
decomposition of $B^2C(\mu_{f})$ or $B^{ncl}C(\mu_{f})$
indexed by $\Cr(\mu_{f})$. Our categorical construction is also useful
to study this problem. We discuss this problem in a separate paper.

One of the sources of inspiration for this work is an unpublished
manuscript of Ralph Cohen, John Jones, and Graeme Segal
\cite{Cohen-Jones-SegalMorse} that appeared in early 90's, in which they 
proposed a way to reconstruct the homotopy type (and even a
homeomorphism type) of a smooth manifold from a topological category
consisting of critical points and moduli spaces of gradient flows
associated to a Morse-Smale function. Our main theorem can be regarded
as a realization of their proposal in a discrete setting.
Another possible use of our result is, therefore, to give an alternative
proof of the homotopy-equivalence part of Cohen-Jones-Segal Morse theory
and its extensions. 
As is done by Gallais \cite{0803.2616}, it is possible to use
discrete Morse theory to approximate smooth Morse theory.
The original motivation of the paper
\cite{Cohen-Jones-SegalMorse} seems to develop a toy model for Floer
homotopy theory \cite{Cohen-Jones-Segal94,Cohen-Jones-Segal95}.
It would be interesting if the theory of flow paths developed in this
paper provides an alternative approach to Floer homotopy theory.

\subsection{Organization of the Paper}

The paper is organized as follows.

\begin{itemize}
 \item We recall basics of discrete Morse theory in
       \S\ref{discrete_morse_function} to fix 
       notation and terminology. The notion of flow paths is introduced
       in \S\ref{flow_on_poset} together with a partial order on the set
       $\FP(\mu)$ of all flow paths with respect to a partial matching
       $\mu$. We also introduce reduced flow paths 
       and compare the set $\overline{\FP}(\mu)$ of reduced
       flow paths with $\FP(\mu)$.
       In \S\ref{continuous_path}, we introduce geometric flows
       associated to flow paths to define a subdivision
       $\Sd_{\mu}(X)$ of $X$. 
       
 \item \S\ref{flow_category} is the main body of this paper.
       After reviewing notation and terminology for small categories in 
       \S\ref{small_category}, the flow category $C(\mu)$ and its
       reduced version  
       $\overline{C}(\mu)$ are introduced in 
       \S\ref{category_of_flows}.

       Theorem \ref{main1} is proved as follows
       \begin{enumerate}
	\item In the first half of \S\ref{collapsing_functor}, we
	      construct a normal colax functor 
	      $\tau:\FP(\mu)\to C(\mu)$, called the collapsing functor, by
	      assigning the terminal cell to 
	      each flow path.  
	      It is shown that the collapsing functor can be restricted
	      to a functor
	      $\tau:\overline{\FP}(\mu) \rarrow{}\overline{C}(\mu)$.

	\item In the second half of \S\ref{collapsing_functor}, we show
	      that $\tau$ induces a homotopy equivalence between  
	      ``normal colax'' classifying spaces
	      \[
	      B^{ncl}\tau:B\overline{\FP}(\mu) = B^{ncl}\overline{\FP}(\mu)
	      \rarrow{\simeq} B^{ncl}\overline{C}(\mu)
	      \]
	      by a $2$-categorical version of Quillen's Theorem A.

	\item In \S\ref{stable_subdivision}, we show that the face poset
	      of the subdivision $\Sd_{\mu}(X)$ constructed in
	      \S\ref{continuous_path} is isomorphic to
	      $\overline{\FP}(\mu)$. Thus we obtain  
	      a chain of homotopy equivalences 
	      \[
	      X \cong BF(\Sd_{\mu}(X)) \cong B\overline{\FP}(\mu) \simeq
	      B^{ncl}\overline{C}(\mu) \simeq B^2\overline{C}(\mu)
	      \simeq B^2C(\mu).
	      \]
       \end{enumerate}

 \item An appendix on one of our main tools, i.e.\ homotopy
       theory of small categories, is attached at the end of this paper
       with the hope of making this article self-contained. 

       Although homotopy theory of small
       categories has been an indispensable tool in topology and
       combinatorics since Segal \cite{SegalBG} and Quillen
       \cite{Quillen73}, there seems to be no standard 
       reference for novices.
\end{itemize}

\subsection{Acknowledgments}

This work was initiated by VN when he met DT during the workshop in
Hakata, Japan on applied topology organized by Yasu Hiraoka in 2011. We
would like to thank Yasu for organizing the meeting.
VN and DT would also like to thank Rob Ghrist for his encouragement
during the workshop and other occasions, without which this work would
not have been completed.
Finally DT appreciates the kindness of Martin Guest who handed him a
copy of Cohen-Jones-Segal paper \cite{Cohen-Jones-SegalMorse} when DT was
a graduate student at the University of Rochester in early 90's.

DT is supported by JSPS KAKENHI Grant Number JP23540082 and JP15K04870.
KT is supported by JSPS KAKENHI Grant Number JP15K17535.

\section{Discrete Morse Theory}
\label{morse_theory}

The main aim of this section is to introduce the notion of flow paths
for discrete Morse functions.

\subsection{Discrete Morse Functions}
\label{discrete_morse_function}

Let us briefly recall Forman's discrete Morse theory.
Throughout the rest of this section, we fix a 
regular CW complex $X$, whose face poset is denoted by $F(X)$. 
We denote the partial order in $F(X)$ by $e\preceq e'$, which means that
$e$ is a face of $e'$.
Moreover, we use $\prec_1$ to denote the \emph{cover}
relation, i.e.\ $e\prec_1 e'$ if and only if $e \preceq e'$ and
$\dim e = \dim e'-1$.

\begin{definition}
 For a function $f : F(X)\to \R$ and a cell $e\in F(X)$, define
 \begin{eqnarray*}
  N_f^{+}(e) & = & \set{e'\in F(X)}{e\prec_1 e',
   f(e)\ge f(e')} \\
  N_f^{-}(e) & = & \set{e'\in F(X)}{e'\prec_1 e,
   f(e')\ge f(e)}.
 \end{eqnarray*}
 $f$ is called a \emph{discrete Morse function}
 if $|N_f^{+}(e)|\le 1$ and $|N_f^{-}(e)|\le 1$ for all $e\in F(X)$.
 A cell $e\in F(X)$ is said to be \emph{critical} if
 $N_f^{+}(e)=N_f^{-}(e)=\emptyset$. The set of critical cells is denoted
 by $\Cr(f)$. 
\end{definition}

\begin{remark}
 We sometimes regard $f$ as a locally constant function on $X$ under the
 composition $X\rarrow{\pi_{X}} F(X)\rarrow{f} \R$, where $\pi_{X}$ is
 the map which defines the cell decomposition or stratification of $X$,
 i.e.\ $\pi_{X}(x)=e$ if $x\in e$.
\end{remark}

Forman observed that, for a noncritical cell $e$, either
$|N_f^{+}(e)|=1$ or $|N_f^{-}(e)|=1$ happens. This allows us to define a
partial matching on the face poset $F(X)$. 

\begin{definition}
\label{matching}
 A \emph{partial matching} on the face poset $F(X)$ is a bijection
 $\mu:D(\mu)\to U(\mu)$ 
 between disjoint subsets $D(\mu)$, $U(\mu)$ of $F(X)$ such 
 that $d \prec_1 \mu(d)$ for each $d \in D(\mu)$.  
 Elements of
 \[
  \Cr(\mu) = F(X)\setminus D(\mu)\cup U(\mu)
 \]
 are called \emph{critical}.
\end{definition}

\begin{definition}
 Given a discrete Morse function $f$ on $F(X)$, define a partial
 matching $\mu_f$ as follows. The domain and the range of $\mu_f$ are
 given by 
 \begin{eqnarray*}
  D(\mu_{f}) & = & \set{e\in F(X)}{N_f^{+}(e)\neq\emptyset} \\
  U(\mu_{f}) & = & \set{e\in F(X)}{N_f^{-}(e)\neq\emptyset}.
 \end{eqnarray*}
 For $e\in D(\mu_{f})$, $\mu_f(e)$ is the unique cell in $N_{f}^{+}(e)$.
\end{definition}

Note that $\Cr(f)=\Cr(\mu_f)$ for a discrete Morse function $f$. 
We identify two discrete Morse functions when their partial matchings
agree. 

\begin{definition}
 Two discrete Morse functions $f$ and $g$ on a regular CW complex
 $X$ are said to be \emph{equivalent} if, for every pair $e\prec_1 e'$ of
 cells in $X$, we have
 \[
  f(e)<f(e') \Longleftrightarrow g(e)<g(e').
 \]
\end{definition}

\begin{definition}
 \label{Forman_path}
 A \emph{Forman path} $\rho$ with respect to a partial matching
 $\mu$ is a sequence of distinct noncritical cells  
 \[
 \rho : d_1 \prec \mu(d_1) \succ d_2 \prec \mu(d_2) \succ \cdots 
 \succ d_n \prec \mu(d_n),  
 \]
 with $\dim d_{i+1}=\dim \mu(d_i)-1$. Equivalently it is 
 a sequence $\rho=(d_1,\ldots,d_n)$ of distinct cells of the same
 dimension in $D(\mu)$ with $\mu(d_i)\succ d_{i+1}$ for all
 $i=1,\ldots,n-1$.  
 $\rho$ is called a \emph{gradient path} if either $n = 1$ or
 $d_1 \not\prec \mu(d_n)$.  
\end{definition}


Forman observed that partial matchings associated to discrete Morse
functions have the following significant property.

\begin{definition}
\label{acyclic_matching}
 A partial matching is called \emph{acyclic} if all of its Forman paths
 are gradient. 
\end{definition}

Conversely, it is known that any acyclic partial matching comes from a
discrete Morse function.

\begin{proposition}
 \label{discrete_Morse_function_for_acyclic_partial_matching}
 For any acyclic partial matching $\mu$ on a finite regular CW complex
 $X$, there exists a discrete Morse function $f$ on $X$ with
 $\mu=\mu_{f}$. 
\end{proposition}

\begin{proof}
 See Theorem 9.3 in \cite{Forman95}.
\end{proof}

This implies that discrete Morse theory can be developed
entirely in terms of acyclic partial matchings. It is, however, useful 
to have good discrete Morse functions at hand.

Forman already observed that
any discrete Morse function on 
a finite CW complex is equivalent to an injective one\footnote{See the
first paragraph of the proof of Theorem 3.3 in \cite{Forman98-2}}, which
can be perturbed further into a faithful Morse function.

\begin{definition}
 A discrete Morse function $f$ is said to be \emph{faithful} if it
 satisfies the following conditions:
 \begin{enumerate}
  \item $f$ is injective.
  \item If $e\prec e'$ and $e'\neq \mu_{f}(e)$, then
	$f(e)<f(e')$. 
 \end{enumerate}
\end{definition}

\begin{proposition}
 \label{faithful_discrete_Morse_function}
 For any discrete Morse function $f$ on a finite regular CW complex $X$,
 there exists a $\Z$-valued faithful discrete Morse function $\tilde{f}$
 equivalent to $f$.
\end{proposition}

Our proof of this fact is closely related to the notion of flow paths,
which will be introduced in \S\ref{flow_on_poset}. A proof of this
proposition is given there, after Remark \ref{description_of_flow_path}.

By Proposition \ref{faithful_discrete_Morse_function}, the discrete
Morse function in Proposition
\ref{discrete_Morse_function_for_acyclic_partial_matching} can be chosen
to be faithful. In the rest of this paper, we fix an acyclic
partial matching $\mu$ on a regular CW complex $X$. When $X$ is finite
we also choose a faithful discrete Morse function $f$ with $\mu=\mu_{f}$.
In particular, all critical values of $f$ are distinct.


\subsection{Combinatorial Flows on Face Posets}
\label{flow_on_poset}

As we have recalled in Definition \ref{Forman_path}, gradient paths in
the sense of Forman can only connect pairs of cells of 
dimension difference $1$. Such flows are enough for computations of
homology. We need more general flows, called flow paths, to describe
relations among cells of arbitrary dimension differences. We also need 
to define a partial order on the set of all flow paths to describe ``the
topology of the space of flow paths'', which can be used to recover the
homotopy type of the original CW complex.

\begin{definition}
\label{flow_path}
 A \emph{flow path} with respect to an acyclic partial matching 
 $\mu$ is a sequence $\gamma=(e_1,u_1,\ldots,e_n,u_n;c)$ of cells 
 satisfying the following conditions:
 \begin{enumerate}
  \item $u_i \in U(\mu)$ for $1\le i\le n$;
  \item either $e_i=u_i$ or $e_i=\mu^{-1}(u_i)$ for $1\le i\le n$;
  \item the last cell $c$ is critical; 
  \item \label{face_condition_for_flow_path}
       $u_i\succ e_{i+1}$ for $1\le i\le n$, where $e_{n+1}=c$;
 \end{enumerate}
 The number $n$ is called the \emph{length} of $\gamma$ and denoted by
 $\ell(\gamma)$. We allow $n$ to be $0$, in which case $\gamma=(c)$.

 The critical cell $c$ is called the \emph{target} of $\gamma$ and is 
 denoted by $\tau(\gamma)$. The cell $e_1$ is called the
 \emph{initial cell} of $\gamma$ and is denoted by $\iota(\gamma)$. When 
 $\ell(\gamma)=0$, $\iota(\gamma)$ is defined to be $c$.
 We also use the notation $e_{n+1}=c$ when $n=\ell(\gamma)$.
 The set of flow paths with respect to $\mu$ is denoted by
 $\FP(\mu)$. When $\mu$ comes from a discrete Morse function $f$, it is
 also denoted by $\FP(f)$.
\end{definition}

\begin{remark}
 \label{description_of_flow_path}
 Recall that a Forman path is a sequence of cells of the following form
 \[
 d_1\prec \mu(d_1) \succ d_2 \prec \mu(d_2) \succ \cdots \succ
 d_n\prec \mu(d_n) \succ c
 \]
 with $d_i\in D(\mu)$ or
 \[
  \mu^{-1}(u_1) \prec u_1 \succ \mu^{-1}(u_2) \prec u_2 \succ \cdots
 \succ \mu^{-1}(u_n) \prec u_n \succ c
 \]
 with $u_i\in U(\mu)$. And all the face relations are of codimension
 $1$.
 
 A flow path, on the other hand, is a sequence of the following form
 \[
 e_1\preceq u_1 \succ e_2 \preceq u_{2} \succ \cdots 
 \succ e_{n} \preceq u_{n} \succ c  
 \]
 in which either $e_i=u_i$ or $\mu^{-1}(u_{i})$. And face relations
 $u_{i}\succ e_{i+1}$ are arbitrary.
 
 Thus a flow path can be written as
 \[
 u_1\succ \cdots \succ u_{i_1-1} \succ \mu^{-1}(u_{i_1}) 
 \prec u_{i_1} \succ \cdots \succ u_{i_2-1} \succ \mu^{-1}(u_{i_2})
 \prec u_{i_2} \succ \cdots \succ u_n \succ c
 \]
 or
 \[
  \mu^{-1}(u_1) \prec u_1 \succ \cdots \succ u_{i_1-1}
  \succ \mu^{-1}(u_{i_1}) \prec u_{i_1} \succ 
  \cdots \succ u_{i_2-1}
  \succ \mu^{-1}(u_{i_2}) \prec u_{i_2} \succ
  \cdots \succ u_n \succ c   
 \]
 depending on $e_1=u_1$ or $\mu^{-1}(u_1)$.
 Note that both of these sequences are considered to be of length $n$. 

 If $\mu=\mu_{f}$ for a faithful discrete Morse function $f$, a flow
 path $\gamma=(e_1,u_1,\ldots,e_n,u_n;c)$ 
 gives rise to a decreasing sequence of real numbers
 \[
 f(e_{1}) \ge f(u_{1}) > \cdots > f(e_{n}) \ge
 f(u_{n})>f(\tau(\gamma)).  
 \]
 This is why we regard flow paths as a discrete analogue of gradient
 flows on smooth manifolds. 
\end{remark}

There is another interpretation of flow paths.

\begin{definition}
 Define a relation $\rhd_1$ on $F(X)$ as follows:
 $e\rhd_{1} e'$ if either $e\succeq e'$ with $\mu(e')\neq e$, or
 $e\in D(\mu)$ and $\mu(e)=e'$.
 Write $\unrhd$ for the transitive closure of $\rhd_{1}$.
\end{definition}

\begin{remark}
 \label{unrhd_as_flow}
 If $e\unrhd e'$, there exists a sequence of cells
 \[
  e=e_1\rhd_{1} e_2\rhd_{1} \cdots \rhd_{1} e_{n-1}\rhd_1 e_{n}=e'.
 \]
 Eliminating superfluous equalities, we see that such a sequence is
 either of the following forms 
 \[
 e = e_1\succ \cdots \succ e_{i_1-1} \succ \mu^{-1}(e_{i_1}) 
 \prec e_{i_1} \succ \cdots \succ e_{i_2-1} \succ \mu^{-1}(e_{i_2})
 \prec e_{i_2} \succ \cdots \succ e_n = e'
 \]
 or
 \[
  e= \mu^{-1}(e_1) \prec e_2 \succ \cdots \succ e_{i_1-1}
  \succ \mu^{-1}(e_{i_1}) \prec e_{i_1} \succ 
  \cdots \succ e_{i_2-1}
  \succ \mu^{-1}(e_{i_2}) \prec e_{i_2} \succ
  \cdots \succ e_n = e',   
 \]
 since there can be no successive sequence of two or more matched pairs
 $\mu^{-1}(e_i)\prec e_i$.
 In particular, $e\unrhd e'$ and $\dim e=\dim e'$ imply that the
 sequence is a strictly alternating sequence of $\prec_1$ and $\succ_1$
 in which $\prec_1$ are matched pairs.
 In particular it is a Forman path in the
 sense of Definition \ref{Forman_path}.

 It should be also noted that flow paths are special kind of such
 sequences, in which we require the 
 last cell is critical and that cells $e_i$ belong to $U(\mu)$.
\end{remark}

\begin{lemma}
 \label{unrhd}
 The relation $\unrhd$ is a partial order on $F(X)$.
\end{lemma}

\begin{proof}
 Since reflexivity and transitivity of $\unrhd$ follow
 immediately by definition, we establish antisymmetry.
 Assume, for contradiction, that $e\unrhd e'\unrhd e$ holds for distinct 
 cells $e$ and $e'$ of $X$. Then, there must exit two chains 
 $e\rhd_1 e_1\rhd_1 \cdots e_{m}\rhd_1 e'$ and 
 $e'\rhd_1 e_{m+1}\rhd_1\cdots \rhd_1 e_{m+n}\rhd e$.
 By concatenating these flow paths, we obtain a sequence
 \[
  e\rhd_1 e_1\rhd_1 \cdots \rhd_{1} e_{m}\rhd_1 e'\rhd_1
 e_{m+1}\rhd_1\cdots \rhd_1 e_{m+n}\rhd_1 e. 
 \]
 By Remark \ref{unrhd_as_flow}, we obtain a Forman path from $e$ to
 itself, which contradicts the acyclicity of $\mu$. 
\end{proof}

We are now ready to prove Proposition
\ref{faithful_discrete_Morse_function}.

\begin{proof}[Proof of Proposition \ref{faithful_discrete_Morse_function}]
 By Lemma \ref{unrhd}, the acyclic partial matching $\mu_{f}$ associated
 to $f$ defines a partial order $\unrhd$ on $F(X)$. Take a linear
 extension of this partial order and $g: F(X)\to\N$ be an injective
 enumeration of cells so that if $e\unrhd e'$ then $g(e)\ge g(e')$.
 To see that $g$ is a discrete Morse function, suppose $e\prec_1 e'$ and
 $g(e)\ge g(e')$. By remark \ref{unrhd_as_flow}, this happens only when
 $e\in D(\mu_{f})$ and $\mu_{f}(e)=e'$.
 Thus, $g$ is a discrete Morse function with $\mu_{g}=\mu_{f}$. Since both 
 $f$ and $g$ are injective, their equivalence follows immediately from the
 fact that their partial matchings coincide. 
\end{proof}

\begin{remark}
 We may define a category whose objects are cells in $X$ and whose
 morphisms from $e$ to $e'$ are sequences of the form
 $e\rhd_1 e_1\rhd_1 \cdots \rhd_{1} e_{m}\rhd_1 e'$. The above argument
 implies that this is an acyclic category and the partial order $\unrhd$
 is the partial order associated with this acyclic category.
\end{remark}

A flow path does not contain a cycle as in the case of gradient paths,  
but it can go back to the boundary of one of previous cells.
We restrict our attention to the following reduced paths in which such
moves are not allowed. 

\begin{definition}
\label{reduced_flow_path}
 A flow path $\gamma=(e_1,u_1,\ldots,e_n,u_n;c)$
 is called \emph{reduced} if $e_{i+1}\not\prec \mu^{-1}(u_i)$ for all
 $i$. When there is a pair $e_{i+1}\prec \mu^{-1}(u_i)$, we say that
 $\gamma$ is \emph{reducible at $i$}.
 The set of reduced flow paths is denoted by $\overline{\FP}(\mu)$ or
 $\overline{\FP}(f)$ if $\mu=\mu_{f}$. 

 For a flow path $\gamma=(e_1,u_1,\ldots,e_n,u_n;c)$, if $\gamma$ is not 
 reducible at $b-1$ and $e_{b}\prec \mu^{-1}(u_{b-1})$,
 $e_{b}\prec \mu^{-1}(u_{b-2})$,
 $\ldots$, $e_{b}\prec \mu^{-1}(u_{a})$, but
 $e_{b}\not\prec \mu^{-1}(u_{a-1})$, the
 set $\{a,a+1,\ldots,b\}$ is called a \emph{reducible interval} for
 $\gamma$. 
\end{definition}

\begin{lemma}
 \label{reduction_is_flow_path}
 When a flow path $\gamma=(e_1,u_1,\ldots,e_n,u_n;c)$ is reducible at
 $i$, define
 \[
  r_i(\gamma) =
 (e_1,u_1,\ldots,e_{i-1},u_{i-1},e_{i+1},u_{i+1},\ldots,e_n,u_n;c).  
 \]
 Then $r_i(\gamma)$ is a flow path.
\end{lemma}

\begin{proof}
 When $\gamma$ is reducible at $i$, we have
 $u_{i-1} \succ e_{i} \succeq \mu^{-1}(u_{i}) \succ e_{i+1}$.   
 Hence the condition \ref{face_condition_for_flow_path} in the
 definition of flow path is satisfied. 
\end{proof}

\begin{definition}
 For $\gamma\in \FP(\mu)$ and a reducible interval
 $I=\{a,a+1,\ldots,b\}$,
 define  
 \[
  r_{I}(\gamma) = r_{a}(r_{a+1}(\cdots r_{b}(\gamma)\cdots)).
 \]
 Let $I_1,\ldots,I_{\ell}$ be the collection of
 all reducible intervals with $i<j$ for all $i\in I_{p}$ and
 $j\in I_{p+1}$ and all $p=1,\ldots,\ell-1$. Define 
 \[
  r(\gamma) = r_{I_1}(r_{I_2}(\cdots r_{I_{\ell}}(\gamma)\cdots)).
 \]

 By Lemma \ref{reduction_is_flow_path}, $r$ defines a map
 \[
 r : \FP(\mu) \rarrow{} \overline{\FP}(\mu).
 \]
 This is called the \emph{reduction map}.
\end{definition}

We define a relation $\preceq$ on $\FP(\mu)$ 
as follows. 

\begin{definition}
\label{subpath}
 For flow paths $\gamma=(e_1,u_1,\ldots,e_m,u_m;c)$ and 
 $\gamma'=(e'_1,u'_1,\ldots,e'_n,u'_n;c')$,  
 define $\gamma \preceq \gamma'$
 if and only if there exists a strictly increasing function
 \[
  \varphi : \{0,1,\ldots,k\} \rarrow{} \{0,1,\ldots,n+1\}
 \]
 for some $1\le k\le m+1$ satisfying the following conditions:
 \begin{enumerate}
  \item \label{subpath_initial} $\varphi(0)=0$,
  \item \label{subpath_constant_cells} $u_{j}=u'_{\varphi(j)}$ for each  
	$1\le j< k$, 
  \item \label{subpath_terminal} $\varphi(k)=n+1$, and
  \item \label{subpath_face_relation}
	for each $1\le j\le k$, $e_{j} \preceq e_{p}'$ for all
	$\varphi(j-1) < p \le \varphi(j)$.
 \end{enumerate} 
 When $\gamma\preceq \gamma'$, $\gamma$ is called a \emph{subpath} of
 $\gamma'$. The function $\varphi$ is called the \emph{embedding
 function for $\gamma\preceq\gamma'$}.
\end{definition}

\begin{remark}
 \label{e_k_is_face_of_target}
 Note that we denote $e'_{n+1}=c'$ for
 $\gamma'=(e'_1,u'_1,\ldots,e'_n,u'_n,c')$.
 Thus the conditions \ref{subpath_terminal} and
 \ref{subpath_face_relation} imply that 
 $e_k\preceq c'=\tau(\gamma')$.
\end{remark}

In particular, the map
$\varphi:\{0,1\}\to\{0,1,\ldots,n+1\}$ defined by $\varphi(1)=n+1$ is
an embedding function for $(\tau(\gamma))\preceq \gamma$ if
$\tau(\gamma)\preceq e_1$, where
$(\tau(\gamma))$ is the flow path consisting of a single critical cell
$\tau(\gamma)$. Another typical example is the
following.

\begin{example}
 \label{upgrading_flow_path}
 For $\gamma=(e_1,u_1,e_2,u_2,\ldots,e_n,u_n;c)$,
 define
 \[
 u(\gamma)=(u_1,u_1,e_2,u_2,\ldots,e_n,u_n;c).
 \]
 The identity map $\{0,\ldots,n+1\}\to \{0,\ldots,n+1\}$ is an embedding
 function and we have $\gamma\preceq u(\gamma)$. 
\end{example}

In order to understand the meaning of the relation $\preceq$ on $\FP(\mu)$,
let us take a look at a more practical example.

\begin{example}
 \label{subpath_example}
 Consider the $2$-simplex in Figure
 \ref{subpath_example_figure1} with the acyclic partial matching
 $\mu$ indicated by the arrows. 
 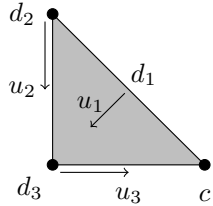
\begin{figure}[ht]
  \begin{center}
   \begin{tikzpicture}
    \draw [fill,lightgray] (2,0) -- (0,2) -- (0,0) -- (2,0);
    \draw (2,0) -- (0,2) -- (0,0) -- (2,0);

    \draw [->] (0.95,0.95) -- (0.5,0.5);
    \draw [->] (-0.1,1.9) -- (-0.1,1);
    \draw [->] (0.1,-0.1) -- (1,-0.1);

    \draw [fill] (2,0) circle (2pt);
    \draw [fill] (0,2) circle (2pt);
    \draw [fill] (0,0) circle (2pt);

    \draw (1.2,1.2) node {$d_1$};
    \draw (0.5,0.8) node {$u_1$};
    \draw (-0.4,2) node {$d_2$};
    \draw (-0.4,1) node {$u_2$};
    \draw (-0.3,-0.3) node {$d_3$};
    \draw (1,-0.4) node {$u_3$};
    \draw (2,-0.4) node {$c$};    
   \end{tikzpicture}  
  \end{center}
  \caption{A partial matching on $2$-simplex}
  \label{subpath_example_figure1}
 \end{figure}
 The only critical cell is $c$ with $D(\mu)=\{d_1,d_2,d_3\}$ and
 $U(\mu)=\{u_1,u_2,u_3\}$. 

 The sequences
 \begin{eqnarray*}
  \delta & = & (d_1,u_1,u_2,u_2,d_3,u_3;c) \\
  \gamma & = & (d_1,u_1,d_3,u_3;c)
 \end{eqnarray*}
 are reduced flow paths terminating at $c$.
 Let us renumber cells in $\gamma$ as
 $\gamma=(d'_1,u'_1,d'_2,u'_2;c)$.
 
 Define $\varphi:\{0,1,2,3\}\to\{0,1,2,3,4\}$ by
 $\varphi(0)=0$, $\varphi(1)=1$, $\varphi(2)=3$, $\varphi(3)=4$. Then
 this is an embedding function for $\gamma\prec \delta$. In fact,
 $u_{\varphi(1)}=u_1=u'_1$, $u_{\varphi(2)}=u_{3}=u'_{2}$ and
 the condition 2 is satisfied. For the condition 4,
 $d'_1=d_1$ is a face of $d_1$, $d'_2=d_3$ is a face of $u_2$
 and $d_3$. Thus $\gamma\prec \delta$.
\end{example}

In order to prove that $\preceq$ is a partial order on $\FP(\mu)$, we
need to prepare a couple of Lemmas.

\begin{lemma}
\label{uniqueness_of_embedding_function}
 For $\gamma \preceq \gamma'$ in $\FP(\mu)$, embedding function 
 $\varphi$ in Definition \ref{subpath} is unique.
\end{lemma}

\begin{proof}
 Let $\gamma=(e_1,u_1,\ldots,e_m,u_m;c)$ and 
 $\gamma'=(e'_1,u'_1,\ldots,e'_n,u'_n;c')$ be flow paths. Suppose
 $\gamma\preceq\gamma'$ and that there
 exist two embedding functions
 \begin{eqnarray*}
  \varphi & : & \{0,1,\ldots,k\} \rarrow{} \{0,1,\ldots,n+1\} \\
  \varphi' & : & \{0,1,\ldots,k'\} \rarrow{} \{0,1,\ldots,n+1\}
 \end{eqnarray*}
 for this relation.

 By the condition $u_j=u'_{\varphi'(j)}=u'_{\varphi(j)}$, it suffices to
 show that $k=k'$. Suppose $k<k'$. By Remark
 \ref{e_k_is_face_of_target}, we have $e_k\preceq c'$ by using the
 conditions for $\varphi$.
 Let $f$ be a faithful discrete Morse function with $\mu=\mu_{f}$.
 Since $c'$ is critical, we have
 $f(e_{k})\le f(c')$. On the other hand, we have
 $u_{k}=u'_{\varphi'(k)}$ by the conditions for $\varphi'$.
 By Remark \ref{description_of_flow_path}, we have decreasing
 sequences 
 \begin{eqnarray*}
 f(e_{1}) \ge f(u_{1}) > & \cdots & > f(e_{m}) \ge
 f(u_{m})>f(c) \\  
 f(e'_{1}) \ge f(u'_{1}) > & \cdots & > f(e'_{n}) \ge
 f(u'_{n})>f(c').  
 \end{eqnarray*}
 In particular, we have
 \[
  f(e_{k}) \ge f(u_{k}) = f(u'_{\varphi'(k)}) > f(c'), 
 \]
 which contradicts to $f(e_k)\le f(c')$.
\end{proof}

An analogous argument implies the following. The proof is omitted.

\begin{lemma}
 \label{ft_is_poset_map}
 Let $f$ be a faithful discrete Morse function with $\mu=\mu_{f}$.
 For $\gamma,\gamma'\in \FP(\mu)$,
 $\gamma\preceq\gamma'$ implies $f(\tau(\gamma))\le f(\tau((\gamma'))$. 
\end{lemma}

%

Recall that $\ell(\gamma)$ is the length of a flow path $\gamma$.

\begin{lemma}
 \label{ell_is_poset_map}
 If $\gamma\preceq\gamma'$ and $\tau(\gamma)=\tau(\gamma')$,
 then we have $\ell(\gamma)\le \ell(\gamma')$. 
\end{lemma}

\begin{proof}
 Suppose  $\gamma=(e_1,u_1,\ldots,e_m,u_m;c)$ and
 $\gamma'=(e'_1,u'_1,\ldots,e'_n,u'_n;c)$.  
 Let
 \[
  \varphi : \{0,\ldots,k\}\longrightarrow \{0,\ldots,n+1\}
 \]
 be the embedding function for $\gamma \preceq \gamma'$. Note that we
 have $e_k\prec \tau(\gamma')=\tau(\gamma)$ by Remark
 \ref{e_k_is_face_of_target}.

 Let $f$ be a faithful discrete Morse function with $\mu=\mu_{f}$.
 By Lemma \ref{ft_is_poset_map}, we have
 $f(\tau(\gamma)) \le f(e_{k}) \le f(\tau(\gamma'))$.
 Since $\tau(\gamma)=\tau(\gamma')$, we obtain
 $f(e_{k})=f(\tau(\gamma))$
 and the injectivity of $f$ implies that $e_{k}=\tau(\gamma)$. Thus
 we have $m=k-1$ by the definition of length of $\gamma$. 
 On the other hand, the injectivity of $\varphi$ implies
 $m+1=k \le n+1$. And we have $\ell(\gamma) \le \ell(\gamma')$.
\end{proof}

\begin{remark}
 \label{k=m+1}
 The above proof implies that, when
 $\tau(\gamma)=\tau(\gamma')$ and $\gamma\preceq\gamma'$, its
 embedding function $\varphi$ is of the form
 \[
 \varphi : \{0,\ldots,\ell(\gamma)+1\} \longrightarrow
 \{0,\ldots,\ell(\gamma')+1\}.   
 \] 
 Thus when $\tau(\gamma)=\tau(\gamma')$ and $\ell(\gamma)=\ell(\gamma')$,
 the embedding for $\gamma\preceq\gamma'$ must be the identity map. 
\end{remark}

\begin{proposition}
 \label{partial_order_on_FP}
 The relation $\preceq$ is a partial order on the set $\FP(\mu)$ of flow
 paths. 
\end{proposition}

\begin{proof}
 For $\gamma\in \FP(\mu)$, the relation $\gamma\preceq \gamma$ is given by
 the identity embedding function.

 Suppose $\gamma\preceq\gamma'$ and $\gamma'\preceq\gamma''$. 
 Let
 \begin{eqnarray*}
  \varphi & : & \{0,\ldots,k\} \longrightarrow
   \{0,\ldots,\ell(\gamma')+1\} \\
  \varphi' & : & \{0,\ldots,k'\} \longrightarrow
   \{0,\ldots,\ell(\gamma'')+1\} 
 \end{eqnarray*}
 be embedding functions for
 $\gamma\preceq\gamma'$ and $\gamma'\preceq\gamma''$, respectively.

 Since $k'\le \ell(\gamma')+1$, there exists $r$ such that 
 $\varphi(r-1)<k'\le \varphi(r)$.
 Define a function
 $\varphi'':\{0,\ldots,r\}\to\{0,\ldots,\ell(\gamma'')+1\}$ by 
 \[
  \varphi''(i) = 
 \begin{cases}
  (\varphi'\circ\varphi)(i), & i<r \\
  \varphi'(k')=\ell(\gamma'')+1, & i=r.
 \end{cases}
 \]
 This is an embedding function for $\gamma \preceq \gamma''$. 

 Finally suppose that $\gamma\preceq\gamma'$ and
 $\gamma'\preceq\gamma$ for $\gamma=(e_1,u_1,\ldots,e_m,u_m;c)$ and
 $\gamma'=(e'_1,u'_1,\ldots,e'_n,u'_n;c')$. 
 By Lemma \ref{ft_is_poset_map} we have $f(\tau(\gamma))=f(\tau(\gamma'))$.
 Since our Morse function $f$ is assumed to be injective, we have
 $\tau(\gamma)=\tau(\gamma')$. By Lemma \ref{ell_is_poset_map}, we have
 $\ell(\gamma)=\ell(\gamma')$. 

 The relation $\gamma\preceq\gamma'$ implies that, if $e_j=u_j$
 in $\gamma$, 
 then
 \[
  u_j=e_j \preceq e'_{\varphi(j)}\preceq u'_{\varphi(j)} = u_j
 \]
 and we have
 $e'_{\varphi(j)}=u'_{\varphi(j)}$. Thus we have  
 \[
 \set{u_{i}}{e_i=u_{i}} \subset \set{u_{j}'}{e'_j=u_{j}'}. 
 \] 
 The assumption $\gamma'\preceq\gamma$ then implies that
 \[
 \set{u_{i}}{e_i=u_{i}} = \set{u_{j}'}{e'_j=u_{j}'}. 
 \] 
 Furthermore, by Remark \ref{k=m+1}, the embedding function
 $\varphi$ for the relation $\gamma\preceq\gamma'$ is the identity
 map. Thus the relations 
 $\gamma\preceq \gamma'$ and $\gamma'\preceq\gamma$ show that the
 corresponding cells in $\gamma$ and $\gamma'$ are identical and we have
 $\gamma=\gamma'$.  
\end{proof}

We regard the set of reduced flow paths $\overline{\FP}(\mu)$ 
as a subposet of $\FP(\mu)$. On the other hand, we have the reduction map
$r:\FP(f)\to\overline{\FP}(\mu)$.  

\begin{lemma}
 \label{r_is_poset_retraction}
 The reduction $r : \FP(\mu) \to \overline{\FP}(\mu)$ is a poset map.
 It is also a retraction. 
\end{lemma}

\begin{proof}
 Suppose $\gamma \preceq \delta$ in $\FP(\mu)$ for
 $\gamma=(e_{1},u_{1},\ldots,e_{m},u_{m};c)$
 and
 $\delta=(e_{1}',u_{1}',\ldots, e_{n}',u_{n}';c')$ and let
 $\varphi:\{0,\ldots,k\}\to\{0,\ldots,n+1\}$ be the embedding function
 for $\gamma\preceq\delta$.

 Let $I_1,\ldots, I_{\ell}$ and
 $J_1,\ldots,J_{\ell'}$ be the reducible intervals for $\gamma$ and
 $\delta$, respectively.  
 Let us first show that the embedding function $\varphi$ can be
 restricted to 
 \[
 \varphi : \{0,\ldots,k\}\setminus\bigcup_{i=1}^{\ell}I_i \rarrow{}
 \{0,\ldots,n+1\}\setminus\bigcup_{j=1}^{\ell'}J_j.
 \]
 In other words, we want to show that, if $\delta$ is reducible at
 $\varphi(i)$, then $\gamma$ is reducible at $i$. When $\delta$ is
 reducible at $\varphi(i)$,
 $e'_{\varphi(i)+1}\prec \mu^{-1}(u'_{\varphi(i)})$.
 By the conditions for embedding functions, on the other hand, we have
 $e_{i+1}\prec e'_{\varphi(i)+1}$ and $u'_{\varphi(i)} = u_{i}$. Thus 
 \[
  e_{i+1}\prec e'_{\varphi(i)+1}\prec \mu^{-1}(u'_{\varphi(i)}) =
 \mu^{-1}(u_{i}), 
 \]
 which implies that $\gamma$ is reducible at $i$.

 Let $k'$ be the cardinality of
 $\{1,\ldots,k\}\setminus\bigcup_{i=1}^{\ell}I_i$ and
 $n'=n-|J_1|-\ldots-|J_{\ell'}|$. 
 The order preserving bijections are denoted by
 \begin{eqnarray*}
  \psi & : & \{0,\ldots,k'\}\rarrow{}
   \{0,\ldots,k\}\setminus\bigcup_{i=1}^{\ell}I_i \\
  \theta & : & \{0,\ldots,n'+1\} \rarrow{}
   \{0,\ldots,n+1\}\setminus\bigcup_{j=1}^{\ell'}J_{j}.
 \end{eqnarray*}
 Then we have a strictly increasing function $\varphi'$ by the
 composition 
 \[
 \varphi' : \{0,\ldots,k'\} \rarrow{\psi}
 \{0,\ldots,k\}\setminus\bigcup_{i=1}^{\ell}I_i \rarrow{\varphi}
 \{0,\ldots,n+1\}\setminus\bigcup_{j=1}^{\ell'}J_j \rarrow{\theta^{-1}}
 \{0,\ldots,n'+1\}. 
 \]
 The fact that $\varphi'$ is the embedding function for
 $r(\gamma)\preceq r(\delta)$ follows immediately from the fact that
 $\varphi$ is the embedding function for $\gamma\preceq\delta$.
\end{proof}

Recall that, when we regard posets as small categories, order preserving
maps correspond to functors.
The following is an important property of the reduction map from the
view point of category theory.

\begin{proposition}
 \label{adjoint_poset}
 The composition
 $\FP(\mu)\rarrow{r}\overline{\FP}(\mu)\hookrightarrow\FP(\mu)$ is a
 descending closure operator\footnote{See Definition
 \ref{closure_operator_definition}}.  
\end{proposition}

\begin{proof}
 Let us show that $r(\gamma) \preceq \gamma$ for any
 $\gamma \in \FP(\mu)$. 
 Let $I_1,\ldots,I_{\ell}$ be the reducible intervals in $\gamma$ and 
 \[
  \psi : \{0,1,\ldots,\ell(r(\gamma))+1\} \rarrow{}
 \{0,\ldots,n+1\}\setminus\bigcup_{i=1}^{\ell}I_i 
 \]
 the order preserving bijection. Then the composition
 \[
  r(\varphi) : \{0,1,\ldots,\ell(r(\gamma))+1\} \rarrow{\psi}
 \{0,\ldots,n+1\}\setminus\bigcup_{i=1}^{\ell}I_i \hookrightarrow
 \{0,1,\ldots,n+1\} 
 \]
 is the embedding function for $r(\gamma)\preceq\gamma$.


%
\end{proof}

\begin{remark}
 The reduction $r(\gamma)$ is the maximal reduced flow path contained
 in $\gamma$ by Proposition \ref{adjoint_poset}. 
\end{remark}

By Corollary \ref{DR_by_closure_operator}, we obtain the following
important fact.

\begin{corollary}
 \label{reduction_is_deformation_retraction}
 $B\overline{\FP}(\mu)$ is
 a strong deformation retract of $B\FP(\mu)$. 
\end{corollary}

\subsection{From Combinatorial Flows to Stable Subdivision}
\label{continuous_path}

In order to relate the combinatorial definition of flow paths to the
homotopy type of $X$, we construct a continuous flow starting from
each $x\in X$ by using a faithful discrete Morse function.
We also construct a subdivision of $X$ by using those continuous flows. 

The first step is to replace characteristic maps for cells as follows. 

\begin{proposition}
 \label{cell_structure_for_matched_pair}
 Given an acyclic partial matching $\mu$ on a finite regular CW complex
 $X$, we may choose 
 characteristic maps for cells in $X$ in such a way that they satisfy the
 following conditions:
 For each matched pair $d\prec \mu(d)=u$ with $\dim u = n$, the
 characteristic map $\varphi$ for $u$ is a
 homeomorphism of triples    
 \[
 \varphi : \left(D^{n}, S^{n-1}_{+}, S^{n-1}_{-}\right) 
 \rarrow{\cong} \left(\overline{u}, \overline{d}, \partial u\setminus
 d\right),   
 \]
 where $S^{n-1}_{+}$ and $S^{n-1}_{-}$ are the
 northern and southern hemispheres of $\partial D^{n}=S^{n-1}$,
 respectively.  
\end{proposition}

We need a version of generalized Sch{\"o}nflies theorem to prove this.
We first need to recall the notion of locally flat embeddings.

\begin{definition}
 Let $M$ and $N$ be topological manifolds without boundaries of
 dimension $m$ and $n$, respectively.
 An embedding of $M$ into $N$ is called \emph{locally flat} at $x\in M$,
 if there exists a neighborhood $U$of $x$ in $N$ such that $(U,U\cap M)$
 is homeomorphic to $(\R^n,\R^m)$.
\end{definition}

The following theorem is due to M.~Brown \cite{M.Brown60}.

\begin{theorem}
 Let $f:S^{n-1}\hookrightarrow S^n$ be a locally flat embedding.
 Then there exists a homeomorphism $\varphi:S^n\to S^n$ with
 \[
  \varphi(f(S^{n-1}))=S^{n-1}=\set{(x_0,\ldots,x_n)\in S^n}{x_n=0}.
 \]
\end{theorem}

 Any piecewise-linear (PL) embedding of a PL $m$-manifold into a PL
 $n$-manifold is known to be locally flat, if $n-m\neq 2$. See Theorem
 1.7.2 in Rushing's book \cite{Rushing73}, for example. 
 By the Jordan-Brouwer separation theorem, for any
 embedding $S^{n-1}\hookrightarrow S^n$, the complement of $S^{n-1}$ has
 two connected components. In particular, we obtain the following fact. 

\begin{corollary}
 \label{generalized_Schoenflies}
 Let $\Sigma^{n-1}$ be a PL $(n-1)$-sphere embedded in a PL $n$-sphere
 $\Sigma^n$. 
 Let us denote the closures of the connected components of
 $\Sigma^n\setminus\Sigma^{n-1}$ by $\Delta^n_{+}$ and $\Delta^n_{-}$. 
 Then both $\Delta^n_{+}$ and $\Delta^n_{-}$ are homeomorphic to an
 $n$-disk $D^n$.
\end{corollary}

\begin{remark}
 Note that this Corollary is not the PL Sch{\"o}nflies theorem, which
 implies that $\Delta^n_+$ and $\Delta^n_{-}$ are PL disks.
\end{remark}

\begin{proof}[Proof of Proposition \ref{cell_structure_for_matched_pair}]
 We modify characteristic maps by induction on dimensions of
 cells. Suppose we have succeeded in modifying characteristic maps for
 cells of dimension less than $n$.
 Let $u$ be an $n$-cell and $\varphi:D^n\to \overline{u}$ the
 original characteristic map for $u$.
 Since $X$ is assumed to be regular, the restriction
 \[
  \varphi|_{S^{n-1}} : S^{n-1} \rarrow{} \partial u
 \]
 is a homeomorphism.
 The regularity of $X$ also implies that, $\overline{u}$ has a structure
 of a PL $n$-disk containing $\partial u$ as a PL $(n-1)$-sphere. This can be
 done by taking the barycentric subdivision of $X$ twice, if necessary.

 The boundary $\partial u$ contains a PL $(n-2)$-sphere $\partial d$.
 By Corollary \ref{generalized_Schoenflies}, there exist homeomorphisms
 \begin{eqnarray*}
  \varphi_{+} & : & D^{n-1} \rarrow{} \overline{d} \\
  \varphi_{-} & : & D^{n-1} \rarrow{} \partial u\setminus d.
 \end{eqnarray*}
 The composition
 \[
  S^{n-2}=\partial D^{n-1} \rarrow{\varphi_{+}} \partial d
 \rarrow{\varphi_{-}^{-1}} S^{n-2}
 \]
 can be extended radially to a homeomorphism $\psi : D^{n-1}\to D^{n-1}$.
 The composition
 \[
 \tilde{\varphi}_{-} : D^{n-1} \rarrow{\psi} D^{n-1}
 \rarrow{\varphi_{-}} \partial u\setminus d
 \]
 is a homeomorphism which agrees with $\varphi_{+}$ on $S^{n-2}$. Thus
 we obtain a homeomorphism
 $\varphi_{\partial} : S^{n-1} \to \partial u$ by gluing $\varphi_{+}$ and
 $\tilde{\varphi}_{-}$ along $S^{n-2}$. 

 The composition 
 \[
 S^{n-1} \rarrow{\varphi_{\partial}} \partial e \rarrow{\varphi^{-1}}
 S^{n-1} 
 \]
 is a homeomorphism. Extend this to a homeomorphism
 $\tilde{\varphi} : D^n \to D^n$ radially. Then the composition 
 \[
 \varphi' : D^n \rarrow{\tilde{\varphi}} D^n \rarrow{\varphi}
 \overline{u} 
 \]
 is a characteristic map for $u$ which satisfies the required condition.
\end{proof}

\begin{definition}
 \label{retraction}
 Let $d\prec\mu(d)=u$ and $\varphi$ be as above. The complement 
 $\partial u\setminus d$ of $d$ in $\partial u\cong S^{n-1}$ is denoted 
 by $d^{c}$. 
 
 The retraction $r : D^{n} \to S^{n-1}_{-}$
 given by $(\bm{x},x_n) \mapsto \left(\bm{x},-\sqrt{1-\|\bm{x}\|^2}\right)$
 induces a retraction of $\overline{u}$ onto $d^{c}$ 
 through $\varphi$, which is denoted by $R_{u} : \overline{u}\to d^c$.
 The composition
 $\overline{d}\hookrightarrow \overline{u}\rarrow{R_{u}} d^{c}$ is
 denoted by $R_d$.
\end{definition}

\begin{remark}
 \label{retraction_on_disk}
 We have the following commutative diagram
 \[
 \begin{diagram}
  \node{D^{n-1}\times[-1,1]} \arrow{e,t}{\psi} \arrow{s,l}{\pr_1}
  \node{D^n} \arrow{s,r}{r} \arrow{e,t}{\varphi}
  \node{\overline{u}} \arrow{s,r}{R_{u}} \\ 
  \node{D^{n-1}} \arrow{e,b}{\psi_{-}} \node{S^{n-1}_{-}}
  \arrow{e,b}{\varphi|_{S^{n-1}_{-}}} \node{d^{c},} 
 \end{diagram}
 \]
 where $\psi$ and $\psi_{-}$ are maps defined by
 \begin{eqnarray*}
  \psi(\bm{x},t) & = & \left(\bm{x}, t\sqrt{1-\|\bm{x}\|^2}\right) \\
  \psi_{-}(\bm{x}) & = & \left(\bm{x},-\sqrt{1-\|\bm{x}\|^2}\right).
 \end{eqnarray*}
 
 Note that
 $\psi|_{D^{n-1}\times\partial[0,1]\cup Int D^{n-1}\times[0,1]}$
 is a homeomorphism onto $D^{n}\setminus S^{n-2}$.
 It should be also noted that $R_{d}(A)=A$ if and only if
 $R_{d}^{-1}(A)=A$ if and only if $A\subset\partial d$.
 These facts will be used in \S\ref{stable_subdivision} to construct a
 subdivision $\Sd_{\mu}(X)$ of $X$.
\end{remark}

When $d$ is not a face of any other cell, the deformation retraction
$R_d$ can be extended to $X\to X\setminus (d\cup\mu(d))$. Such a
deformation retraction is called an \emph{elementary collapse} in simple
homotopy theory \cite{M.CohenSimpleHomotopy}. Thus matched pairs can be
regarded as generalizations of elementary collapses.

In the rest of this paper, we fix a deformation
retraction $R_{u}: \overline{u}\to d^c=\partial u\setminus d$ for each
matched pair $d\prec_1 u$.
Now we are ready to construct continuous flows on $X$.
Let
\[
 L:[-1,1]\times [0,1] \rarrow{} [-1,1]
\]
be the linear flow on $[-1,1]$ which carries $x\in [-1,1]$ to $-1$,
i.e.\ $L(x,t)=(1-t)x-t$. By extending 
this flow, we obtain a flow on $D^{n-1}\times[-1,1]$ and hence a flow
$L_{d,u}:\overline{u}\times[0,1]\to \overline{u}$ which makes the following
diagram commutative
\[
 \begin{diagram}
  \node{D^{n-1}\times[-1,1]\times[0,1]} \arrow{e,t}{\psi\times 1}
  \arrow{s,l}{1\times L} 
  \node{D^n\times[0,1]} \arrow{s} \arrow{e,t}{\varphi\times 1}
  \node{\overline{u}\times[0,1]} \arrow{s,r}{L_{d,u}} \\ 
  \node{D^{n-1}\times[-1,1]} \arrow{e,b}{\psi} \node{D^{n}}
  \arrow{e,b}{\varphi} \node{\overline{u}.} 
 \end{diagram}
\]

\begin{definition}
 \label{continuous_flow_definition}
 Suppose $X$ is finite.
 For each $x\in X$, we assign a nonnegative number $h_{x}$ and define a 
 continuous path 
 $L_{x}:[0,h_{x}]\to X$ with the following
 properties:
\begin{enumerate}
 \item $L_{x}(0)=x$.
 \item When $u\cap L_{x}([0,h_{x}])\neq\emptyset$ for $u\in U(\mu)$,
       the restriction of $L_{x}$ to $u$ coincides with a
       restriction of $L_{\mu^{-1}(u),u}$ with a certain parameter
       shift. 
\end{enumerate}

 Choose a faithful discrete Morse function $f$ for $\mu$. 
 We proceed by induction on $f(x)$. (Recall that we regard $f$ as an
 integer-valued locally constant function on $X$.)

 When $f(x)$ is the minimum value of $f$, the path
 $L_{x}$ is defined to be the constant path at $x$ with length $0$.

 Suppose we have defined paths $L_{y}$ for all points $y$ with
 $f(y)<f(x)$.
 Let $e$ be the unique cell containing $x$ in its interior.
 If $e\in \Cr(\mu)$, we also define $L_{x}$ to be the constant path
 at $x$ with length $0$. 
 When $e\in U(\mu)$, there exist $y\in \mu^{-1}(e)$ and $s\in [0,1]$
 with $x=L_{\mu^{-1}(e),e}(y,s)$. Note that $y$ and $s$ are uniquely
 determined by $x$.
 Let $e'$ be the unique cell containing
 $x'=L_{\mu^{-1}(e),e}(y,1)$ in its interior. Then $e'\prec e$ but $e'$ is
 not matched with $e$, since $e'\neq \mu^{-1}(e)$. Thus $f(e')<f(e)$,
 since $f$ is faithful.
 By the inductive hypothesis, there exist a number $h_{x'}$ and a
 continuous path 
 $L_{x'}: [0,h_{x'}]\to X$ satisfying the required conditions. Define
 $h_{x}=h_{x'}+1-s$ and 
 \[
 L_{x}(t) =
 \begin{cases}
  L_{x'}(t-1+s), & \text{ if } t\in [1-s,h_{x}] \\
  L_{\mu^{-1}(e),e}(y,t+s), & \text{ if } [0,1-s].
 \end{cases}
 \]
 When $e\in D(\mu)$, let $e'$ be the unique cell containing
 $x'=L_{e,\mu(e)}(x,1)$ in its interior. Then $f(e')<f(e)$ and the
 inductive hypothesis applies to $x'$. Now define $h_{x}=h_{x'}+1$ and 
 \[
 L_{x}(t)=
 \begin{cases}
  L_{x'}(t-1), & \text{ if } t\in [1,h_{x}] \\
  L_{e,\mu(e)}(x,t), & \text{ if } [0,1].
 \end{cases}
 \]
 And we obtain a continuous path
 \[
  L_{x} : [0,h_{x}] \rarrow{} X
 \]
 satisfying the desired conditions.
 The number $h_{x}$ defined by the above procedure is called the
 \emph{height of $x$ with respect to $f$}. And the path $L_{x}$ is
 called the \emph{flow associated to $f$ with initial point $x$}.
\end{definition}

\begin{remark}
 The finiteness of $X$ is used to choose a faithful discrete Morse
 function $f$ for an acyclic partial matching. If we start from a
 faithful discrete Morse function, the finiteness assumption is not
 necessary. 
\end{remark}


The continuous path constructed above is closely
related to flow paths.

\begin{lemma}
 \label{flow_path_from_x}
 For each $x\in X$, there exists a unique reduced flow path
 $\gamma_{x}$ which contains $x$ in its initial cell $\iota(\gamma_{x})$
 and contains the image of $L_{x}$ in the union of cells in
 $\gamma_{x}$.

 Conversely, for each reduced flow path
 $\gamma=(e_1,u_1,\ldots,e_n,u_n;c)$, there exists $x\in e_1$ such that
 $\gamma_{x}=\gamma$. 
\end{lemma}

\begin{proof}
 Let $e_1,\ldots,e_n$ be the sequence of cells appeared in the
 construction of $L_{x}$. Depending on $e_i\in D(\mu)$ or
 $e_i\in U(\mu)$, define $u_i=\mu(e_i)$ or $u_i=e_i$.
 By the construction, $L_x$ lands in a critical cell, say $c$. 
 Then the sequence  $\gamma_{x}=(e_1,u_1,\ldots,e_n,u_n;c)$ is a reduced
 flow path. The uniqueness follows from the uniqueness of the choices of
 cells $e_1,\ldots,e_n$.

 Conversely, let $\gamma=(e_1,u_1,\ldots,e_n,u_n;c)$ be a reduced flow
 path. Choose a point $z\in c$. Starting with $z$, let us reverse the
 construction of flows in Definition \ref{continuous_flow_definition} to
 obtain $x\in X$ with $\gamma_{x}=\gamma$ and $L_{x}(h_x)=z$.
 We proceed by induction on the length $n$ of $\gamma$.
 When $n=0$, $\gamma=(c)$ and there is nothing to prove.

 Suppose we have found such points for reduced flow paths of lengths
 $\le n-1$. Let $\gamma'=(e_2,u_2,\ldots,e_n,u_n;c)$ and apply the
 inductive hypothesis. Then there exists $x'\in e_2$ such that the image 
 of $L_{x'}$ is contained in the union of cells in $\gamma'$.
 Since $\gamma$ is reduced, 
 $e_{2}\subset \partial u_1\setminus \overline{\mu^{-1}(u_1)}$. When
 $e_1=\mu^{-1}(u_1)$, there exists $x\in e_1$ such that
 $L_{e_1,u_1}(x,1)=x'$. When $e_1=u_1$, there exist $x\in e_1$
 and $s\in (0,1)$ such that $x'=L_{\mu^{-1}(u_1),u_1}(x,s)$.
 By construction, cells containing the image of $L_{x}$ form $\gamma$ and
 thus $\gamma=\gamma_{x}$.
\end{proof}

This lemma allows us to define a surjective map
\[
 \pi_{\mu} : X \rarrow{} \overline{\FP}(\mu)
\]
by $\pi_{\mu}(x)=\gamma_{x}$. Hence we have a stratification on $X$
\begin{equation}
 X= \bigcup_{\gamma \in \overline{\FP}(\mu)} e_{\gamma},
  \label{decomposition_by_flows}
\end{equation}
where $e_{\gamma}=\pi_{\mu}^{-1}(\gamma)$, for a reduced flow path
$\gamma$. 

Note that we have a commutative diagram
\begin{equation}
 \begin{diagram}
  \node{} \node{X} \arrow{sw,t}{\pi_{\mu}} \arrow{se,t}{\pi_{X}} \node{} \\
  \node{\overline{\FP}(\mu)} \arrow[2]{e,b}{\iota} \node{} \node{F(X),}
 \end{diagram}
 \label{pi_f_is_subdivision}
\end{equation}
where $\pi_{X}$ is the defining map for the cell decomposition of $X$,
i.e.\ $x\in X$ is mapped to the unique cell containing $x$.
Thus 
\[
 e_{\lambda} =
 \bigcup_{\iota(\gamma)=e_{\lambda}} e_{\gamma} 
\]
for $e_{\lambda}\in F(X)$.

The following description of $e_{\gamma}$ is useful.

\begin{lemma}
 \label{e_gamma_by_L}
 For a reduced flow path $\gamma=(e_1,u_1,\ldots,e_n,u_n;c)$, we have
 \[
  e_{\gamma} = \set{x\in e_1}{L_{x}(h_x)\in c}.
 \]
 \end{lemma}

\begin{proof}
 By definition
 \[
  e_{\gamma} = \set{x\in e_1}{\gamma_{x}=\gamma}.
 \]
 The condition $\gamma_{x}=\gamma$ is equivalent to saying that $\gamma$
 consists of cells which intersect with the image of $L_{x}$. By the
 construction of $L_{x}$, this is equivalent to $L_{x}(h_x)\in c$.
\end{proof}

\begin{corollary}
 \label{inductive_e_gamma}
 For a reduced flow path $\gamma=(e_1,u_1,\ldots,e_n,u_n;c)$, let
 $\gamma'=(e_2,u_2,\ldots,e_n,u_n;c)$. Then
 \[
  e_{\gamma} = R_{u_1}^{-1}(e_{\gamma'})\cap e_1.
 \]
 When $e_1=\mu^{-1}(u_1)$, we also have
 \[
  e_{\gamma} = R_{e_1}^{-1}(e_{\gamma'})\cap e_1.
 \]
\end{corollary}

\begin{proof}
 Suppose $e_1=\mu^{-1}(u_1)$. By the construction of $L_{x}$, we have
 \[
  e_{\gamma} = \set{x\in e_1}{L_{e_1,u_1}(x,1)\in e_{\gamma'}}
 \]
 and the map $L_{e_1,u_1}(-,1) : \overline{u_1}\to e_{1}^{c}$ coincides
 with $R_{u_1}$. The case $e_1=u_1$ is analogous and is omitted.
\end{proof}

The rest of this section is devoted to the proof of the following fact. 

\begin{proposition}
 \label{stable_subdivision_is_cellular}
 The subdivision (\ref{decomposition_by_flows}) of $X$
 is a regular cell decomposition.
\end{proposition}

In order to prove this, let us consider the following more general
situation. Suppose we have a regular cell complex $X$. Suppose further 
that each cell $e_{\lambda}$ in $X$ is equipped with a decomposition
 \[
 e_{\lambda} = \bigcup_{\alpha\in A_{\lambda}}
 e_{\lambda,\alpha}. 
 \]
We would like to know when the decomposition of $X$ 
\begin{equation}
  X = \bigcup_{\alpha\in A_{\lambda}}\bigcup_{\lambda\in\Lambda}
   e_{\lambda,\alpha} 
   \label{subdivision_by_subdivision_of_cells}
\end{equation}
is a regular cell decomposition.

It is easy to see that the closure of $e_{\lambda}$ has the following
decomposition 
\begin{equation}
 \overline{e_{\lambda}} =
  \bigcup_{\mu\in\Lambda} \bigcup_{\beta\in A_{\mu}}
  \bigcup_{e_{\mu,\beta}\cap\overline{e_{\lambda}}\neq\emptyset}
  e_{\mu,\beta}. 
\label{closure_of_subdivision}
\end{equation}

\begin{lemma}
 \label{subdivision_of_regular_cell_complex}
 If (\ref{closure_of_subdivision}) is a regular cell decomposition of
 $\overline{e_{\lambda}}$ for all $\lambda$, then
 (\ref{subdivision_by_subdivision_of_cells}) 
 is a regular cell decomposition of $X$, which is a subdivision of the
 original cell decomposition on $X$.
\end{lemma}

\begin{proof}
 By assumption, each $e_{\lambda,\alpha}$ is equipped with a
 characteristic map
 \[
  \varphi_{\lambda,\alpha} : D^{\dim e_{\lambda,\alpha}} \rarrow{}
 \overline{e_{\lambda}}. 
 \]
 Composed with the inclusion $\overline{e_{\lambda}}\hookrightarrow X$,
 we obtain a characteristic map for $e_{\alpha,\lambda}$ in $X$, since
 the closure of $e_{\lambda,\mu}$ in 
 $\overline{e_{\lambda}}$ coincides with the closure in $X$.
 
 The condition that
 $\partial e_{\lambda,\alpha}=\overline{e_{\lambda,\alpha}}\setminus e_{\lambda,\alpha}$
 is covered with cells of dimension $< \dim e_{\lambda,\alpha}$ 
 follows from the assumption.
\end{proof}

The following observation is essential in the proof of Proposition
\ref{stable_subdivision_is_cellular}.

\begin{lemma}
 \label{inverse_image_of_collapse}
 Let $K$ be a triangulation of a convex polytope $P$ of dimension $n$. 
 Define a relation $\sim$ on $P\times[0,1]=|K|\times[0,1]$ by
 $(\bm{x},s) \sim (\bm{x},t)$ 
 for $\bm{x}\in\partial P$ and $s,t\in [0,1]$. The equivalence relation
 generated by $\sim$ is also denoted by $\sim$.
 Denote the quotient space $P\times[0,1]/_{\sim}$ by $E$ and the
 canonical projection onto $P$ by
 \[
  p : E \rarrow{} P.
 \]
 Then for any simplex $\sigma\in K$, $p^{-1}(\sigma)$ is homeomorphic to
 $D^{\dim \sigma+1}$.

 More generally for any subcomplex $L\subset K$ whose geometric
 realization is homeomorphic to a disk
 of dimension $n$, $p^{-1}(|L|)$ is homeomorphic to a disk of dimension
 $n+1$. 
\end{lemma}

\begin{proof}
 When $\sigma\cap \partial P=\emptyset$, $p^{-1}(\sigma)$
 can be identified with $\sigma\times[0,1]$ and is homeomorphic
 to $D^{\dim \sigma+1}$. Suppose 
 $\sigma\cap \partial P\neq\emptyset$. Then
 $\tau=\sigma\cap \partial P$ is a face of $\sigma$.
 We have
 \[
  p^{-1}(\sigma) = \sigma\times [0,1]/_{\sim_{\sigma}},
 \]
 where the equivalence relation $\sim_{\sigma}$ is the restriction of
 $\sim$ and thus  
 $(\bm{x},s)\sim_{\sigma} (\bm{x},t)$
 for $\bm{x}\in \sigma\cap \partial P$ and $s,t\in [0,1]$.

 Let $\{v_0,v_1,\ldots,v_k\}$ be the set of vertices of $\sigma$. We may
 assume that the vertices belonging to $\tau$ are the first $\ell+1$
 vertices $v_0,\ldots, v_{\ell}$. Define a convex polytope $Q$
 contained in $\sigma\times[0,1]$ by
 \[
  Q= \Conv((v_0,0), \ldots, (v_{\ell},0),(v_{\ell+1},0), \ldots,
 (v_{k},0), (v_{\ell+1},1), \ldots, (v_{k},1)).
 \]
 We claim that $p^{-1}(\sigma)$ is homeomorphic to $Q$. A homeomorphism
 $\pi: p^{-1}(\sigma) \to Q$ is, for example, defined by
 \[
  \pi\left(\left[\sum_{i=0}^k t_iv_i,s\right]\right) = \sum_{i=0}^{\ell}
 t_i(v_i,0) + 
 \sum_{i=\ell+1}^k (1-s)t_i(v_i,0) + \sum_{i=\ell+1}^{k} st_i(v_i,1).
 \]
 Since elements in the face $\tau$ correspond to
 $\sum_{i=0}^k t_iv_i$ with $t_{\ell+1}=\cdots=t_k=0$, this map is
 well defined and continuous. And this is easily seen to be bijective.

 Let $L$ be a subcomplex of $K$ with $|L|\cong D^n$. We prove that
 $p^{-1}(|L|)$ is homeomorphic to an $(n+1)$-disk by induction on the
 number of $n$-simplices and $n$.
 Suppose we have proved the statement for $n\le m-1$. Furthermore
 suppose that we have proved the statement for subcomplexes of dimension
 $m$ and the number of $m$-simplices less than $k$.

 Suppose $L$ has $k$ $m$-simplices. Choose an $m$-simplex
 $\sigma$ in $L$. Let $L'$ be the subcomplex of $L$ obtained by removing
 $\sigma$ and its faces in $\partial |L|$. Since $|L|$ is homeomorphic
 to a disk, it is collapsible to a point. Thus we may choose $\sigma$
 with which $|L'|$ is homeomorphic to an $m$-disk. Then
 $\sigma\cap |L'|$ is homeomorphic to an $(m-1)$-disk. By inductive
 assumption on dimension, 
 $p^{-1}(\sigma\cap |L'|)$ is homeomorphic to an $m$-disk. By
 inductive assumption on the number of $m$-cells, $p^{-1}(|L'|)$ is
 homeomorphic to an $(m+1)$-disk. Thus
 $p^{-1}(|L|)=p^{-1}(|L'|)\cup p^{-1}(\sigma)$ is homeomorphic to an
 $(m+1)$-disk as a union of two $(m+1)$-disks along an $m$-disk.
\end{proof}

\begin{proof}[Proof of Proposition \ref{stable_subdivision_is_cellular}]
 Let us verify the condition in Lemma
 \ref{subdivision_of_regular_cell_complex}. 
 By the diagram (\ref{pi_f_is_subdivision}), we have
 \[
  e_{\lambda} = \bigcup_{\iota(\gamma)=e_{\lambda}} e_{\gamma}
 \]
 for a cell $e_{\lambda}$ in $F(X)$, which leads to a decomposition of
 $X$
 \[
  X = \bigcup_{\lambda\in F(X)} \bigcup_{\iota(\gamma)=e_{\lambda}}
 e_{\gamma}. 
 \]
 We need to construct a characteristic map for each $e_{\gamma}$
 and prove that 
 \[
 \overline{e_{\lambda}} =
 \bigcup_{\gamma}
 \bigcup_{e_{\gamma}\cap\overline{e_{\lambda}}\neq\emptyset}
 e_{\gamma} 
 \]
 is a regular cell decomposition of $\overline{e_{\lambda}}$.

 Let us first construct a characteristic map $\varphi_{\gamma}$ for
 $e_{\gamma}$. Choose a faithful discrete Morse function $f$ with
 $\mu=\mu_{f}$. 
 We construct $\varphi_{\gamma}$ by induction
 on $f(\iota(\gamma))$.

 When $f(\iota(\gamma))$ is the minimum value of $f$, $\iota(\gamma)$ is
 critical and $\gamma$ is the constant flow path $(\iota(\gamma))$. 
 The characteristic map $\varphi_{\gamma}$ for
 $e_{\gamma}=\iota(\gamma)$ is defined
 to be the characteristic map for $\iota(\gamma)$. 

 Suppose we have constructed a characteristic map of $e_{\gamma}$
 for $\gamma$ with $f(\iota(\gamma))\le k$. Let
 $\gamma=(e_1,u_1,\ldots,e_n,u_n;c)$ be a 
 reduced flow path with $f(e_1)=k+1$. (Recall that our $f$ is
 $\Z$-valued.) Since $\gamma$ is nontrivial, $e_1$ is not critical.
 Let $\gamma'=(e_2,u_2,\ldots,e_n,u_n;c)$. 
 Let $\varphi : D^{n} \to \overline{u_1}$ be the characteristic
 map for $u_1$. By Proposition \ref{cell_structure_for_matched_pair}, we
 may assume that $\varphi$ is a homeomorphism of triples
 \[
 \varphi : (D^{n},S^{n-1}_{+},S^{n-1}_{-}) \rarrow{}
 \left(\overline{u_1},\overline{d_1},d_1^{c}\right), 
 \]
 where $d_1=\mu^{-1}(u_1)$.
 By Remark \ref{retraction_on_disk}, we have 
 a commutative diagram 
 \[
 \begin{diagram}
  \node{D^{n-1}\times[-1,1]} \arrow{e,t}{\psi} \arrow{s}
  \node{D^n} \arrow{s,r}{r} \arrow{e,t}{\varphi}
  \node{\overline{u_1}} \arrow{s,r}{R_{d_1}} \\ 
  \node{D^{n-1}} \arrow{e,b}{\psi_{-}} \node{S^{n-1}_{-}}
  \arrow{e,b}{\varphi|_{S^{n-1}_{-}}} \node{d_1^{c}.} 
 \end{diagram}
 \]
 (See Definition \ref{retraction} and Remark \ref{retraction_on_disk}
 for definitions of $r$, $\psi$, and $\psi_{-}$.)
	     
 By Remark
 \ref{description_of_flow_path}, 
 we have
 \[
 k+1=f(e_1) \ge f(u_1) > f(e_2) = f(\iota(\gamma'))
 \]
 and the inductive hypothesis applies to $\gamma'$.
 Let
 \[
 \varphi_{\gamma'} : D^{\dim e_{\gamma'}} \rarrow{}
 \overline{e_{\gamma'}} 
 \]
 be the characteristic map for $e_{\gamma'}$.
 Note that
 $e_{\gamma'}\subset\partial u_1\setminus d_1$ by construction. 
 Since $\varphi|_{S^{n-1}_{-}}$ is a homeomorphism, there
 exists an embedding
 $\alpha:D^{\dim e_{\gamma'}}\hookrightarrow D^{n-1}$ 
 making the following diagram commutative
 \[
 \begin{diagram}
  \node{D^{n-1}} \arrow{e,t}{\psi_{-}} \node{S^{n-1}_{-}}
  \arrow{e,t}{\varphi|_{S^{n-1}_{-}}} \node{d_1^{c}} \\
  \node{D^{\dim e_{\gamma'}}} \arrow{n,l}{\alpha}
  \arrow[2]{e,b}{\varphi_{\gamma'}} 
  \node{} \node{\overline{e_{\gamma'}}.} \arrow{n,J}
 \end{diagram}
 \]
 
 There are two cases; $e_1=d_1$ or $e_1=u_1$.

 \begin{description}
  \item[The case $e_1=d_1$:] 
	     We have
	     $e_{\gamma}=R_{d_1}^{-1}(e_{\gamma'})\cap d_1$
	     by Corollary \ref{inductive_e_gamma}. 
	     Let $i_0:D^{n-1}\hookrightarrow D^{n-1}\times[-1,1]$ be the
	     inclusion into $D^{n-1}\times\{1\}$. Then the composition
	     \[
	      D^{\dim e_{\gamma'}} \rarrow{\alpha} D^{n-1}
	     \rarrow{i_0} D^{n-1}\times[-1,1] \rarrow{\psi} D^n
	     \rarrow{\varphi} \overline{e^n} \hookrightarrow X
	     \]
	     is a characteristic map for $e_{\gamma}$.

  \item[The case $e_1=u_1$:] 
	     In this case
	     $e_{\gamma}=R_{u_1}^{-1}(e_{\gamma'})\cap u_1$.
	     By pulling back the regular cell decomposition of
	      $d_1^{c}$ via $\varphi$, $S^{n-1}_{-}$ has a structure of
	     regular cell complex. Let $K$ be a simplicial
	     subdivision of this regular cell decomposition. Then  
	     the map $\psi$ can be regarded as the quotient map
	     \[
	      |K|\times[0,1] \rarrow{} |K|\times[0,1]/_{\sim}
	     \]
	     in Lemma \ref{inverse_image_of_collapse}. Thus, by Lemma
	     \ref{inverse_image_of_collapse},
	     $\varphi^{-1}(R_{u_1}^{-1}(e_{\gamma'}))$ is
	     homeomorphic to a disk of dimension
	     $\dim e_{\gamma'}+1$ and we obtain a characteristic
	     map for $e_{\gamma}$. 

%
 \end{description}

 The fact that
 $\overline{e_{\gamma}}\setminus e_{\gamma}$ is covered with
 cells of dimension $<\dim e_{\gamma}$ follows from the construction of
 the characteristic map for $e_{\gamma}$, or the construction of a
 homeomorphism in Lemma \ref{inverse_image_of_collapse}. 
\end{proof}

 \begin{definition}
  \label{stable_subdivision_definition}
 The decomposition (\ref{decomposition_by_flows}) is called the
 \emph{stable subdivision} of $X$ associated to $\mu$ 
 and is denoted by $\Sd_{\mu}(X)$ or $\Sd_{f}(X)$ when $\mu=\mu_{f}$ for
 a faithful discrete Morse function $f$.

 For a critical cell $c$, the subspace
 \[
  W^{s}(c)= \bigcup_{\tau(\gamma)=c} e_{\gamma}
 \]
 is called the \emph{stable subspace} of $c$.
 \end{definition}

\begin{remark}
 Recall that a smooth Morse function on a manifold $M$ induces a
 decomposition of $M$ by stable manifolds. The stable manifold $W^{s}(c)$
 at a critical point $c$ is a submanifold of $M$ consisting of points on
 gradient flows going down to $c$. 
 The stable subdivision defined above is a discrete analogue of this
 decomposition.
\end{remark}

\begin{example}
 Consider the cubical cell decomposition of a torus $T^2$ given by
 the left picture in Figure \ref{torus}.
 Define an acyclic partial matching on this cell complex by the right
 picture in Figure \ref{torus}, in which critical cells are drawn by
 black. 

 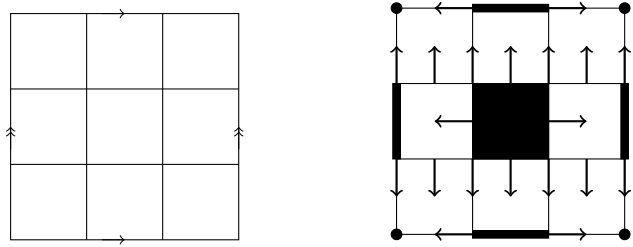
\begin{figure}[ht]
  \begin{center}
   \begin{tikzpicture}
    \draw (0,0) rectangle (3,3);

    \draw (1,0) -- (1,3);
    \draw (2,0) -- (2,3);
    \draw (0,1) -- (3,1);
    \draw (0,2) -- (3,2);
    \draw [->] (1.2,3) -- (1.5,3);
    \draw [->] (1.2,0) -- (1.5,0);
    \draw [->>] (0,1.2) -- (0,1.5);
    \draw [->>] (3,1.2) -- (3,1.5);
   \end{tikzpicture}
   \hspace{50pt}
   \begin{tikzpicture}
    \draw (0,0) rectangle (3,3);
    \draw (1,0) -- (1,3);
    \draw (2,0) -- (2,3);
    \draw (0,1) -- (3,1);
    \draw (0,2) -- (3,2);
    \draw [thick,->] (1,3) -- (0.5,3);
    \draw [thick,->] (2,3) -- (2.5,3);
    \draw [thick,->] (0,2) -- (0,2.5);
    \draw [thick,->] (0.5,2) -- (0.5,2.5);
    \draw [thick,->] (1,2) -- (1,2.5);
    \draw [thick,->] (1.5,2) -- (1.5,2.5);
    \draw [thick,->] (2,2) -- (2,2.5);
    \draw [thick,->] (2.5,2) -- (2.5,2.5);
    \draw [thick,->] (3,2) -- (3,2.5);
    \draw [thick,->] (1,1.5) -- (0.5,1.5);
    \draw [thick,->] (2,1.5) -- (2.5,1.5);
    \draw [thick,->] (0,1) -- (0,0.5);
    \draw [thick,->] (0.5,1) -- (0.5,0.5);
    \draw [thick,->] (1,1) -- (1,0.5);
    \draw [thick,->] (1.5,1) -- (1.5,0.5);
    \draw [thick,->] (2,1) -- (2,0.5);
    \draw [thick,->] (2.5,1) -- (2.5,0.5);
    \draw [thick,->] (3,1) -- (3,0.5);
    \draw [thick,->] (1,0) -- (0.5,0);
    \draw [thick,->] (2,0) -- (2.5,0);

    \draw [fill] (1,1) rectangle (2,2);

    \draw [fill] (-0.05,1) rectangle (0.05,2);
    \draw [fill] (2.95,1) rectangle (3.05,2);
    \draw [fill] (1,-0.05) rectangle (2,0.05);
    \draw [fill] (1,2.95) rectangle (2,3.05);
    
    \draw [fill] (0,0) circle (2pt);
    \draw [fill] (3,0) circle (2pt);
    \draw [fill] (0,3) circle (2pt);
    \draw [fill] (3,3) circle (2pt);
   \end{tikzpicture}
  \end{center}
  \caption{An acyclic partial matching on $T^2$}
  \label{torus}
 \end{figure}

The decomposition by stable subspaces and 
the stable subdivision are drawn in Figure
 \ref{stable_subdivision_of_torus}. 

 \begin{figure}[ht]
  \begin{center}
   \begin{tikzpicture}
    \draw [fill,lightgray] (0,2) rectangle (1,3);
    \draw [fill,lightgray] (2,2) rectangle (3,3);
    \draw [fill,lightgray] (0,0) rectangle (1,1);
    \draw [fill,lightgray] (2,0) rectangle (3,1);

    \draw [fill,lightgray] (1.3,2) -- (1,3) -- (1,2) -- (1.3,2);
    \draw [fill,lightgray] (1.7,2) -- (2,3) -- (2,2) -- (1.7,2);
    \draw [fill,lightgray] (2,1.3) -- (3,1) -- (2,1) -- (2,1.3);
    \draw [fill,lightgray] (2,1.7) -- (3,2) -- (2,2) -- (2,1.7);
    \draw [fill,lightgray] (1.3,1) -- (1,0) -- (1,1) -- (1.3,1);
    \draw [fill,lightgray] (1.7,1) -- (2,0) -- (2,1) -- (1.7,1);
    \draw [fill,lightgray] (1,1.3) -- (0,1) -- (1,1) -- (1,1.3);
    \draw [fill,lightgray] (1,1.7) -- (0,2) -- (1,2) -- (1,1.7);

    \draw (0,0) rectangle (3,3);

    \draw (1,0) -- (1,3);
    \draw (2,0) -- (2,3);
    \draw (0,1) -- (3,1);
    \draw (0,2) -- (3,2);

   \end{tikzpicture}
   \hspace{50pt}
   \begin{tikzpicture}
    \draw (0,0) rectangle (3,3);

    \draw (1,0) -- (1,3);
    \draw (2,0) -- (2,3);
    \draw (0,1) -- (3,1);
    \draw (0,2) -- (3,2);

    \draw (1.3,2) -- (1,3);
    \draw (1.7,2) -- (2,3);
    \draw (2,1.3) -- (3,1);
    \draw (2,1.7) -- (3,2);
    \draw (1.3,1) -- (1,0);
    \draw (1.7,1) -- (2,0);
    \draw (1,1.3) -- (0,1);
    \draw (1,1.7) -- (0,2);

    \draw (0.3,2) -- (0,3);
    \draw (0.7,2) -- (1,3);
    \draw (0.3,2) -- (1,1.8);
    \draw (0.7,2) -- (1,1.9);
    \draw (2.3,2) -- (2,3);
    \draw (2.7,2) -- (3,3);
    \draw (2.7,2) -- (2,1.8);
    \draw (2.3,2) -- (2,1.9);

    \draw (0.3,1) -- (0,0);
    \draw (0.7,1) -- (1,0);
    \draw (0.3,1) -- (1,1.2);
    \draw (0.7,1) -- (1,1.1);
    \draw (2.3,1) -- (2,0);
    \draw (2.7,1) -- (3,0);
    \draw (2.7,1) -- (2,1.2);
    \draw (2.3,1) -- (2,1.1);

   \end{tikzpicture}
  \end{center}
  \caption{Stable subdivision on $T^2$}
  \label{stable_subdivision_of_torus}
\end{figure}
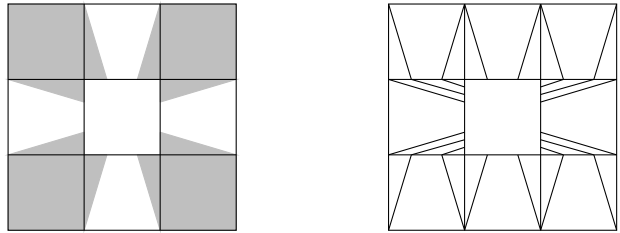

\end{example}

\section{The Flow Category}
\label{flow_category}


The aim of this section is to construct a
poset-enriched category $C(\mu)$, called the flow category, having
critical cells of $\mu$ as objects and show that the classifying space of
$C(\mu)$ is homotopy equivalent to $X$.

\subsection{Small Categories}
\label{small_category}

Let us first fix notation and terminology for categories and
functors. We regard a small category as a monoid in the category of
quivers.

\begin{definition} \hspace{\fill}
 \begin{enumerate}
  \item  A \emph{quiver} $Q$ consists of a set $Q_0$ of vertices and a
	 set $Q_1$ of arrows together with maps
	 \[
	 s,t : Q_1 \longrightarrow Q_0.
	 \]
	 For $x,y\in Q_0$ and $u\in Q_1$, when $s(u)=x$ and $t(u)=y$, we
	 write $u : x\to y$. By definition, the set of arrows from $x$
	 to $y$ equals $s^{-1}(x)\cap t^{-1}(y)$ and is denoted by
	 $Q(x,y)$. 

  \item For a quiver $Q$, define
	\[
	N_2(Q) = \set{(u,v)\in Q_1^2}{s(u)=t(v)}.
	\]

  \item A (small) \emph{category} is a quiver $C$ equipped with maps
	\begin{eqnarray*}
	 \circ & : & N_2(C) \longrightarrow C_1 \\
	 \iota & : & C_0 \longrightarrow C_1
	\end{eqnarray*}
	satisfying the associativity and the unit conditions, i.e.
	\begin{eqnarray*}
	 (u\circ v)\circ w & = & u\circ(v\circ w) \\
	 u & = & \iota(t(u))\circ u = u\circ \iota(s(u))
	\end{eqnarray*}
	for all $u,v,w\in C_1$ with $s(u)=t(v)$ and $s(v)=t(w)$.
	We also require that $s(u\circ v)=s(v)$, $t(u\circ v)=t(u)$, and 
	$s(\iota(x))=t(\iota(x))=x$. 
       
	For $x\in C_0$, $\iota(x)$ is denoted by $1_x$ and is called the
	\emph{identity morphism} at $x$.


  \item A category $C$ is said to be \emph{acyclic} if either $C(x,y)$ or
       $C(y,x)$ is empty for any $x, y\in C_0$ with $x\neq y$ and
       $C(x,x)=\{1_x\}$ for any 
       $x\in C_0$.
 \end{enumerate}
 Definitions of functors and natural transformations in this formulation
 should be obvious and are omitted.
\end{definition}

\begin{remark}
 A poset is a typical example of an acyclic category. Many of
 properties and constructions for posets can be extended to acyclic
 categories. 
\end{remark}

For a small category $C$, we have a coproduct decomposition
\[
 C_1 = \coprod_{x,y\in C_0} C(x,y).
\]
Thus a small category $C$ can be defined as a collection of a set $C_0$,
a family of sets $\{C(x,y)\}_{x,y\in C_0}$, maps 
\[
 \circ : C(y,z)\times C(x,y) \longrightarrow C(x,z),
\]
and elements $1_x\in C(x,x)$ satisfying the associativity and the unit
conditions. 

In order to study higher structures, we need categories whose sets of
morphisms $C(x,y)$ are topological spaces, posets, small categories, and
so on. More generally, there is a notion of 
categories enriched over a monoidal category. Recall that, roughly
speaking, a monoidal category is a category equipped with a ``tensor
product operation'' for pairs of objects.

\begin{definition}
 \label{enriched}
 Let $(\bm{V},\otimes, 1)$ be a strict monoidal category\footnote{It is
 well known that any monoidal category can be replaced by a strict one.
 See Mac\,Lane's book \cite{MacLaneCategory} or Kassel's book
 \cite{KasselBook}, for example.}. 
 \begin{enumerate}
  \item A \emph{$\bm{V}$-quiver} $Q$ consists of a set $Q_0$ of vertices 
	and a family $\{Q(x,y)\}_{x,y\in Q_0}$ of objects in $\bm{V}$.
  \item A \emph{category enriched over $\bm{V}$} or simply a
	\emph{$\bm{V}$-category} consists of a $\bm{V}$-quiver $C$
	together with
	\begin{itemize}
	 \item a morphism
	       \[
	       \circ : C(y,z)\otimes C(x,y) \longrightarrow C(x,z)
	       \]
	       in $\bm{V}$ for each triple $x,y,z\in C_o$ and
	 \item a morphism $1_x:1\to C(x,x)$ in $\bm{V}$ for each $x\in C_0$
	\end{itemize}
	that are subject to the standard requirements of being a small
	category. 
 \end{enumerate}
\end{definition}


Examples of monoidal categories include the category $\Posets$ of posets,
the category $\Spaces$ of topological spaces, and the category $\Cats$ of
small categories. Monoidal structures are given by products in all of
these categories. We use the following terminologies. 

\begin{definition}
 \label{popular_enriched_categories}
 Categories enriched over $\Posets$, $\Spaces$ and $\Cats$ are called
 \emph{poset-categories}, \emph{topological categories}, and 
 \emph{(strict) $2$-categories} respectively.
\end{definition}

\begin{remark}
 Note that poset-categories
 can be regarded as $2$-categories under the inclusion
 $\Posets\subset\Cats$. 
\end{remark}

Our first task in this section is to construct a poset-category $C(\mu)$
from an acyclic partial matching $\mu$. Since poset-categories are
$2$-categories, we need to recall basic definitions and properties of
$2$-categories.

\begin{definition}
 For a $2$-category $C$ and objects $x,y\in C_0$, $C(x,y)$ is a small
 category. Thus it consists of the set of objects $C(x,y)_0$
 and morphisms $C(x,y)_1$. Elements of $C(x,y)_0$ and $C(x,y)_1$ are
 called \emph{$1$-morphisms} and \emph{$2$-morphisms} in $C$,
 respectively. When $\theta\in C(x,y)(u,v)$, we 
 write $\theta : u\Rightarrow v$.

 The classes of all $1$-morphisms and
 $2$-morphisms in $C$ are denoted by $C_1$ and $C_2$, respectively. When 
 $C_2$ is a set, $C$ is called a \emph{small} $2$-category.

 There are two kinds of compositions in a $2$-category
 $C$. First of all, given a triple $x,y,z\in C_0$ of objects, we have
 the \emph{horizontal composition} functor
 \[
 \circ : C(y,z)\times C(x,y) \longrightarrow C(x,z).
 \]
 On the other hand, for a triple $u,v,w\in C(x,y)_0$, we have the
 \emph{vertical composition} denoted by
 \[
 \ast : C(x,y)(v,w)\times C(x,y)(u,v) \longrightarrow C(x,y)(u,w).
 \]
\end{definition}

There are weaker versions of $2$-categories such as bicategories and
$(\infty,2)$-categories. Although we only use strict
$2$-categories, we are often forced to use weaker notions of functors
between $2$-categories. 

\begin{definition}
 \label{colax_functor_definition}
 Let $C$ and $D$ be small $2$-categories.
 A \emph{colax functor} $f$ from $C$ to $D$
 consists of the following data:
 \begin{itemize}
  \item a map
	\[
	 f : C_0 \longrightarrow D_0,
	\]
  \item a family of functors
	\[
	 f=f_{y,x} : C(x,y) \longrightarrow D(f(x),f(y)),
	\]
  \item a family of natural transformations
	\[
	\xymatrix{%
	f(x) \ar[r]^{f(w)}
	\rruppertwocell<15>^{f(w'\circ w)}{\hspace*{30pt}f_{x'',x',x}}
	& f(x') \ar[r]^{f(w')} & f(x'')
	\lltwocell<\omit>{\omit}
	}
	\]
  \item a family of natural transformations 
	\[
	\xymatrix{%
	f(x) \rtwocell^{f(1_x)}_{1_{f(x)}}{\hspace*{5pt}f_x} & f(x).
	}
	\]
 \end{itemize}
 These maps are subject to the following additional conditions: 
 \begin{enumerate}
  \item The following diagram is commutative in $D(f(x),f(t))$
	\begin{equation}
	\xymatrixcolsep{45pt}
	 \xymatrix{
	f((w\circ v)\circ u) \ar@{=}[d]
	\ar@{=>}[r]^{f_{q,y,x}} & f(w\circ v)\circ f(u)
	\ar@{=>}[r]^{f_{q,z,y}\circ 1_{f(u)}} &
	(f(w)\circ f(v))\circ f(u) \ar@{=}[d] \\ 
	f(w\circ(v\circ u)) 
	\ar@{=>}[r]_{f_{q,z,x}} &
	f(w)\circ f(v\circ u)\ar@{=>}[r]_{1_{f(w)}\circ f_{z,y,x}} &
	f(w)\circ (f(v)\circ f(u)) 
	}
	\label{composer}
	\end{equation}
	for any composable $1$-morphisms
	$x \rarrow{u} y \rarrow{v} z \rarrow{w} q$ in $C$.

  \item The following diagrams are commutative in $D(f(x),f(y))$
	\begin{equation}
	\xymatrixcolsep{30pt}
	\xymatrix{ 
	f(u\circ 1_x) \ar@{=}[d] \ar@{=>}[r]^{f_{y,x,x}} &
	f(u)\circ f(1_x) \ar@{=>}[r]^(0.54){1_{f(u)}\circ f_x} &
	f(u)\circ 1_{f(x)} 
	\ar@{=}[d] \\ 
	f(u) \ar@{=}[rr] & & f(u)
	}
	\label{uniter1}
	\end{equation}
	\begin{equation}
	\xymatrixcolsep{30pt}
	\xymatrix{
	f(1_y\circ u) \ar@{=}[d] \ar@{=>}[r]^{f_{y,y,x}} & 
	f(1_y)\circ f_1(u) \ar@{=>}[r]^(0.54){f_y\circ 1_{f(u)}} &
	1_{f(y)}\circ f(u) 
	\ar@{=}[d] \\ 
	f(u) \ar@{=}[rr] & & f(u)
	}
	\label{uniter2}
	\end{equation}
	for any $1$-morphism $u : x\to y$ in $C$.
 \end{enumerate}

 $f$ is called a \emph{normal colax functor} when $f_x$ is the identity 
 for all objects $x \in C_0$.  

 For simplicity, we suppress the indices in $f_{z,y,x}$ and $f_x$ and
 denote them by $f$, when there is no danger of confusion.
\end{definition}

\begin{remark}
 Note that colax functors are often called \emph{oplax functors} in the
 literature. The choice of the term ``colax'' is based on the general
 principle that, when $2$-morphisms are reversed, we put a suffix
 ``co''. 
\end{remark}

\begin{remark}
 \label{lax_functor_definition}
 We may reverse the directions of natural transformations $f_{x,y,z}$
 and $f_x$ to obtain the notion of \emph{lax functors}.

 When these natural transformations are isomorphisms, $f$ is
 called a \emph{pseudofunctor}. In particular, a pseudofunctor can be
 regarded as a lax functor and a colax functor.
\end{remark}

\subsection{Categories of Combinatorial Flows}
\label{category_of_flows}

Here we construct a poset category $C(\mu)$ out of
the poset of flow paths $\FP(\mu)$. A reduced version
$\overline{C}(\mu)$ is also introduced.
 
 We first define poset-enriched quivers\footnote{Definition
 \ref{enriched}.} 
 $Q(\mu)$ and $\overline{Q}(\mu)$.

\begin{definition}
 Define $Q(\mu)_0=\Cr(\mu)$ and 
 \[
 Q(\mu)_1 = \set{(c,\gamma)\in \Cr(\mu)\times\FP(\mu)}{c\succ \iota(\gamma)}.
 \]
 Extend the target map $\tau:\FP(\mu)\to\Cr(\mu)$ to a structure of quiver
 \[
 \sigma,\tau : Q(\mu)_1 \longrightarrow Q(\mu)_0
 \]
 by $\sigma(c,\gamma)=c$. Similarly we define
 $\overline{Q}(\mu)_0=Q(\mu)_0=\Cr(\mu)$ and 
 \[
 \overline{Q}(\mu)_1 = \set{(c,\gamma)\in
 \Cr(\mu)\times\overline{\FP}(\mu)}{c\succ \iota(\gamma)}. 
 \]
 The source and the target maps
 $\sigma,\tau : \overline{Q}(\mu)_1\to \overline{Q}(\mu)_0$ are 
 defined analogously.

 For each pair $(c,c')$ of critical cells, the forgetful map
 \[
  Q(\mu)(c,c') \longrightarrow \FP(\mu)
 \]
 is injective. 
 Define a partial order on $Q(\mu)(c,c')$ as a full subposet 
 of $\FP(\mu)$ under this injection. The poset structure on
 $\overline{Q}(\mu)(c,c')$ is induced analogously from that of
 $\overline{\FP}(\mu)$. 

 The resulting poset-quivers $Q(\mu)$ and $\overline{Q}(\mu)$ are
 called the \emph{flow quiver} and the \emph{reduced flow quiver} of
 $\mu$, respectively.  When $\mu=\mu_{f}$ for a discrete Morse function
 $f$, they are also denoted by $Q(f)$ and $\overline{Q}(f)$,
 respectively. 
\end{definition}

\begin{remark}
 When $\gamma=(e_1,u_1,\ldots,e_n,u_n;c')$, we denote 
 $(c,\gamma)\in Q(\mu)(c,c')$ by $(c;e_1,u_1,\ldots,e_n,u_n;c')$ or simply
 by $\gamma$ when there is no danger of confusion.
\end{remark}

In order to make these quivers into small categories, we need to specify
how to compose morphisms. For the flow quiver $Q(\mu)$, it is simply
given by concatenations. 

\begin{definition}
 For $\gamma\in Q(\mu)(c_1,c_2)$ and
 $\gamma'\in Q(\mu)(c_2,c_3)$,
 define $\gamma * \gamma'$ by 
 \[
 \gamma * \gamma' = (c_1;e_1,u_1,\ldots,e_m,u_m,
 e'_1,u'_1,\ldots,e'_n,u'_n;c_3) 
 \]
 when $\gamma=(c_1;e_1,u_1,\ldots,e_m,u_m;c_2)$ and
 $\gamma'=(c_2;e'_1,u_1',\ldots,e_n',u_n';c_3)$. 
\end{definition}

\begin{remark}
 The fact that $\gamma\ast \gamma'$ belongs to
 $Q(\mu)(c_1,c_3)$ can be verified as follows. Suppose
 $\gamma=(c_1;e_1,u_1,\ldots,e_m,u_m;c_2)$ and 
 $\gamma'=(c_2;e_1',u_1',\ldots,e_n',u_n';c_3)$. By definition, we have
 $u_m\succ c_2$ and $c_2\succ e_1'$. Thus
 $u_m\succ e_1'$ and we obtain a flow path $\gamma *\gamma'$.
\end{remark}

\begin{lemma}
 \label{concatenation_is_poset_map}
 The concatenation of flow paths defines a poset map
 \[
 * : Q(\mu)(c_{1},c_{2}) \times Q(\mu)(c_{2},c_{3}) \longrightarrow
 Q(\mu)(c_{1},c_{3}),
 \]
 where $Q(\mu)(c_1,c_2)\times Q(\mu)(c_2,c_3)$ is equipped with the product
 poset structure.
\end{lemma}

\begin{proof}
 Suppose that $(\gamma, \gamma')$ and $(\delta, \delta')$ in
 $Q(\mu)(c_{1},c_{2}) \times Q(\mu)(c_{2},c_{3})$ satisfy  
 the order relation
 $(\gamma,\gamma') \preceq (\delta,\delta')$. In other
 words, $\gamma \preceq \delta$ and
 $\gamma' \preceq \delta'$ hold in $Q(\mu)(c_{1},c_{2})$ and
 $Q(\mu)(c_{2},c_{3})$, respectively.  

 Suppose
 \begin{eqnarray*}
  \gamma & = & (c_1;e_1,u_1,\ldots,e_{m},u_{m};c_2) \\
  \gamma' & = & (c_1;e'_{m+1},u'_{m+1},\ldots,e'_{m+m'},u'_{m+m'};c_3) \\
  \delta & = & (c_2;e''_1,u''_1,\ldots,e''_{n},u''_{n};c_2) \\
  \delta' & = & (c_2;e'''_{n+1},u'''_{n+1},\ldots,e'''_{n+n'},u'''_{n+n'};c_3).
 \end{eqnarray*}
 By Remark \ref{k=m+1}, the embedding functions for the relations
 $\gamma \preceq \delta$ and $\gamma' \preceq \delta'$ are of the
 following form
 \begin{eqnarray*}
  \varphi & : & \{0,\ldots,m+1\}\to\{0,\ldots,n+1\} \\
  \varphi' & : & \{0,\ldots,m'+1\}\to\{0,\ldots,n'+1\}
 \end{eqnarray*}
 According to our numbering scheme, the conditions for
 $\gamma\preceq\delta$ are 
 \begin{enumerate}
  \item $\varphi(0)=0$,
  \item $u_j=u''_{\varphi(j)}$ for $1\le j<m+1$,
  \item $\varphi(m+1)=n+1$,
  \item for each $1\le j\le m+1$, $e_j\preceq e''_{p}$ for all
	$\varphi(j-1)<p\le \varphi(j)$
 \end{enumerate}
 and the conditions for $\gamma'\preceq\delta'$ are
\begin{enumerate}
 \item $\varphi'(0)=0$,
 \item $u'_{j}=u'''_{\varphi'(j-m)+n}$ for
	$m+1\le j< m+m'+1$,
 \item $\varphi'(m'+1)=n'+1$,
 \item for each $m+1\le j\le m+m'+1$, $e_j\preceq e''_{p}$ for all
	$\varphi'(j-1-m)+n<p\le \varphi'(j-m)+n$.
\end{enumerate}

 Now define a map
 \[
 \varphi*\varphi' : \{0,\ldots,m+m'+1\}\to\{0,\ldots,n+n'+1\}
 \]
 by
 \[
  (\varphi*\varphi')(i) = 
 \begin{cases}
  \varphi(i), & 0\le i\le m \\
  \varphi'(i-m)+n, & m+1\le i\le m+m'+1.
 \end{cases}
 \]
 Let us verify that this map gives rise to the relation
 $\gamma * \gamma' \preceq \delta * \delta$. The conditions
 $(\varphi*\varphi')(0)=0$ and $(\varphi*\varphi')(m+m'+1)=n+n'+1$ are
 immediate from the definition.
 The remaining conditions in Definition \ref{subpath} can be split into
 the following: 
 \begin{enumerate}
  \item $u_j=u''_{\varphi*\varphi'(j)}$ for $1\le j< m+1$,
  \item $u'_{j}=u'''_{\varphi*\varphi(j)}$ for
	$m+1\le j< m+m'+1$,
  \item for each $1\le j\le m+1$, $e_j\preceq e''_{p}$ for all
	$(\varphi*\varphi')(j-1)<p\le (\varphi*\varphi')(j)$, and
  \item for each $m+1\le j\le m+m'+1$, $e'_j\preceq e'''_{p}$ for all
	$(\varphi*\varphi')(j-1)<p\le (\varphi*\varphi')(j)$.
 \end{enumerate}
 By the definition of $\varphi*\varphi'$, these conditions are 
 \begin{enumerate}
  \item $u_j=u''_{\varphi(j)}$ for $1\le j<m+1$,
  \item $u'_{j}=u'''_{\varphi'(j-m)+n}$ for
	$m+1\le j< m+m'+1$,
  \item for each $1\le j\le m+1$, $e_j\preceq e''_{p}$ for all
	$\varphi(j-1)<p\le \varphi(j)$, and
  \item for each $m+1\le j\le m+m'+1$, $e'_j\preceq e'''_{p}$ for all
	$\varphi'(j-1-m)+n<p\le \varphi'(j-m)+n$.
 \end{enumerate}
 And we obtain
 $\gamma * \gamma' \preceq \delta * \delta'$.
\end{proof}

\begin{lemma}
 \label{concatenation_is_associative}
 The concatenation operation is associative, i.e.\ 
 \[
 (\gamma_{1} * \gamma_{2}) * \gamma_{3} = \gamma_{1} * (\gamma_{2} *
 \gamma_{3}) 
 \] 
 for $\gamma_1\in Q(\mu)(c_1,c_2)$, $\gamma_2\in Q(\mu)(c_2,c_3)$, and
 $\gamma_3\in Q(\mu)(c_3,c_4)$.
\end{lemma} 

\begin{proof}
 By definition.
\end{proof}

The following relations are useful.

\begin{lemma}
 \label{order_and_composition}
 For a pair of composable flow paths
 $\gamma=(e_1,u_1,\ldots,e_m,u_m;\tau(\gamma))$ 
 and $\delta=(e_1',u_1',\ldots,e_n',u_n';\tau(\delta))$, we have  
 \begin{enumerate}
  \item $\delta \preceq \gamma\ast\delta$, 
	when either
	\begin{enumerate}
	 \item $\ell(\delta)\ge 1$, or
	 \item $\ell(\delta)=0$ and $\tau(\delta)\preceq e_p$ for all
	       $1\le p\le m$, and 
	\end{enumerate} 
  \item $\gamma\ast\delta \preceq\gamma$, when either
	\begin{enumerate}
	 \item $\ell(\gamma)\ge 1$, or
	 \item $\ell(\gamma)=0$ and $e'_1\preceq\tau(\gamma)$.
	\end{enumerate}
 \end{enumerate}
\end{lemma}

\begin{proof}
 Define $\varphi:\{0,1,\ldots,n+1\}\to \{0,1,\ldots,n+m+1\}$ by
 $\varphi(0)=0$ and $\varphi(i) = i+m$ for $i=1,\ldots,n+1$. Then the
 first three conditions for $\delta \preceq \gamma\ast\delta$ is obvious. 

 Suppose $n>0$. For each $1\le j\le n+1$, $\varphi(j-1)<p\le\varphi(j)$
 implies that $p=m+j$. Thus the fourth condition is equivalent to
 $e'_j\preceq e''_{m+j}=e'_j$ for $j=1,\ldots,n$ where
 \[
  \gamma\ast\delta =
 (e''_1,u_1'',\ldots,e''_{m+n},u''_{m+n};\tau(\delta)). 
 \]
 When $n=0$, the condition is $e'_1\preceq e''_{p}$ for
 $\varphi(0)<p\le\varphi(1)=m+1$. Or $\tau(\delta)\preceq e_{p}$ for
 $0<p\le m+1$. Since $\tau(\delta)=e_{m+1}=\tau(\gamma)$ in this case,
 the essential condition is  $\tau(\delta)\preceq e_{p}$ for
 $0<p\le m$.

 Let us show that the identity
 map $\{0,1,\ldots,m+1\}\to\{0,1,\ldots,m+1\}$ is an embedding function
 for $\gamma\ast\delta\preceq\gamma$ if $\ell(\gamma)\ge 1$. Again the
 first three conditions obviously hold. The fourth condition is
 equivalent to $e_j\preceq e_j$ 
 for $1\le j\le m$ and $e_1'\preceq \tau(\gamma)$, when
 $\ell(\gamma)\ge 1$. Since
 $\gamma\ast\delta$ is defined, $e_1'\preceq\tau(\gamma)$ holds. 
 When $\ell(\gamma)=0$, the fourth condition is equivalent to 
 $e'_1\preceq\tau(\gamma)$.
\end{proof}

\begin{example}
 \label{triangle_example2}
 Consider the function $f$ on
 the boundary $\partial [v_0,v_1,v_2]$ of a $2$-simplex defined by
 \begin{eqnarray*}
  f([v_0]) & = & 0 \\
  f([v_0,v_1]) & = & 1 \\
  f([v_0,v_2]) & = & 2 \\
  f([v_1]) & = & 4 \\
  f([v_2]) & = & 5 \\
  f([v_1,v_2]) & = & 6.
 \end{eqnarray*}
 Then this is a discrete Morse function whose partial matching is given
 by 
 \begin{eqnarray*}
  \mu_f([v_1]) & = & [v_0,v_1] \\
  \mu_f([v_2]) & = & [v_0,v_2].
 \end{eqnarray*}
 It has two critical simplices $[v_0]$ and $[v_1,v_2]$.
 Here is a list of Forman paths:
 \begin{eqnarray*}
  \gamma_1 & : & [v_1] \prec [v_0,v_1] \\
  \gamma_2 & : & [v_2] \prec [v_0,v_2].
 \end{eqnarray*}
 We regard each Forman path as a flow path by adding $[v_0]$ at the end: 
 \begin{eqnarray*}
  \gamma_{1} & : & [v_1] \prec [v_0,v_1] \succ [v_0] \\
  \gamma_{2} & : & [v_2] \prec [v_0,v_2] \succ [v_0].
 \end{eqnarray*}
 We have non-Forman flow paths as follows:
 \begin{eqnarray*}
  \gamma_{0} & : & [v_0] \\
  \gamma_{01} & : & [v_0,v_1] \preceq [v_0,v_1] \succ [v_0] \\
  \gamma_{02} & : & [v_0,v_2] \preceq [v_0,v_2] \succ [v_0] \\
  \gamma_{12} & : & [v_1,v_2]
 \end{eqnarray*}
 so that
 \[
  \FP(\mu) = \overline{\FP}(\mu) = 
 \{\gamma_0, \gamma_1, \gamma_2, \gamma_{01}, \gamma_{02}, \gamma_{12}\}.
 \]
 Note that
 \begin{eqnarray*}
  \gamma_{01} & = & \gamma_{01}\ast \gamma_0  \\
  \gamma_{02} & = & \gamma_{02}\ast \gamma_0, 
 \end{eqnarray*}
 which imply, by Lemma \ref{order_and_composition},
 \begin{eqnarray*}
  \gamma_{0} & \prec & \gamma_{01} \\
  \gamma_{0} & \prec & \gamma_{02}.
 \end{eqnarray*}
 Note also that $\gamma_{0}\prec\gamma_{1}$ and $\gamma_0\prec\gamma_2$
 do not hold even though 
 $\gamma_1=\gamma_1\ast\gamma_0$ and $\gamma_2=\gamma_2\ast\gamma_0$.

 Furthermore we also have 
 \begin{eqnarray*}
  \gamma_{1} & \prec & \gamma_{12} \\
  \gamma_{2} & \prec & \gamma_{12} \\
  \gamma_{1} & \prec & \gamma_{01} \\
  \gamma_{2} & \prec & \gamma_{02}.
 \end{eqnarray*}
 For example, an embedding function for $\gamma_{1}\prec\gamma_{01}$ is
 given by the identity map $\{0,1,2\}\to\{0,1,2\}$ and an embedding
 function for $\gamma_{1}\prec\gamma_{12}$ is given by the identity map
 $\{0,1\}\to\{0,1\}$.

 Thus the Hasse diagram is given by
 
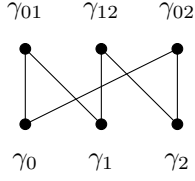
\begin{center}
  \begin{tikzpicture}
  \draw [fill] (0,0) circle (2pt);
  \draw (0,-0.5) node {$\gamma_0$};

  \draw [fill] (1,0) circle (2pt);
  \draw (1,-0.5) node {$\gamma_1$};

  \draw [fill] (2,0) circle (2pt);
  \draw (2,-0.5) node {$\gamma_2$};

  \draw [fill] (0,1) circle (2pt);
  \draw (0,1.5) node {$\gamma_{01}$};

  \draw [fill] (1,1) circle (2pt);
  \draw (1,1.5) node {$\gamma_{12}$};

  \draw [fill] (2,1) circle (2pt);
  \draw (2,1.5) node {$\gamma_{02}$};

  \draw (0,0) -- (0,1);
  \draw (1,0) -- (1,1);
  \draw (2,0) -- (2,1);
  \draw (1,0) -- (0,1);
  \draw (2,0) -- (1,1);
  \draw (0,0) -- (2,1);
  \end{tikzpicture}
 \end{center}
 which implies that $B\FP(\mu)$ is the boundary of a hexagon and can be
 identified with the barycentric subdivision of $X=\partial[v_0,v_1,v_2]$.
\end{example}

Here is a definition of the flow category $C(\mu)$.

\begin{definition}[Flow category]
 \label{flow_category_definition}
 Define a category $C(\mu)$ as follows.
 The set of objects is given by
 \[
  C(\mu)_0 = \Cr(\mu).
 \]
 The set of morphisms $C(\mu)(c,c')$ from a critical cell $c$ to another
 $c'$ is given
 \[
  C(\mu)(c,c') = \begin{cases}
		Q(\mu)(c',c), & c\neq c' \\
		\{1_{c}\}, & c=c.
	       \end{cases}
 \]
 The composition $\circ$ is defined by the concatenation $*$
 \[
 \circ : C(\mu)(c_2,c_3)\times C(\mu)(c_1,c_2) = Q(\mu)(c_3,c_2)\times
 Q(\mu)(c_2,c_1) \rarrow{\ast} Q(\mu)(c_3,c_1) = C(\mu)(c_1,c_3).
 \]
 This category $C(\mu)$ is called the \emph{flow category} of
 $\mu$. When $\mu=\mu_{f}$ for a discrete Morse function, we also use
 the notation $C(f)$.
\end{definition}

As a corollary to Lemma \ref{concatenation_is_poset_map}, we see that
$C(\mu)$ is a poset category. In particular it is a
$2$-category and we may apply any of the classifying space functors
reviewed in \S\ref{nerve}, i.e.\ $B^2$, $B^{cl}$, and $B^{ncl}$. By
Theorem \ref{classifying_space_of_2-category}, these constructions are
weakly homotopy equivalent to each other.

\begin{example}
 \label{triangle_example3}
 Consider the flow paths in Example \ref{triangle_example2}. We have
 \[
  C(\mu)([v_0],[v_1,v_2])=Q(\mu)([v_1,v_2],[v_0]) = \{\gamma_1,\gamma_2\}
 \]
 with trivial order relation. Thus $C(\mu)$ is a $1$-category and we
 have a homeomorphism
 \[
  BC(\mu) \cong \partial[v_0,v_1,v_2].
 \]
\end{example}

We would like to show that this example generalizes, i.e.\ the
classifying space of the flow category $C(\mu)$ is always homotopy 
equivalent to $X$. It is not easy, however, to 
relate these spaces directly. We need intermediate categories and a
zigzag of functors
\begin{equation}
 F(\Sd_f(X)) \larrow{} \overline{\FP}(\mu) \rarrow{\tau} \overline{C}(\mu)
 \larrow{r} C(\mu).
 \label{zigzag_of_functors}
\end{equation}
$\Sd_f(X)$ is a subdivision of $X$ which will be defined in
\S\ref{stable_subdivision}. $\overline{\FP}(\mu)$ is the poset of reduced
flow paths defined in Definition \ref{reduced_flow_path} and
\ref{subpath}. 
The rest of this section is devoted to the 
construction of the \emph{reduced flow category} $\overline{C}(\mu)$ and
a functor $r: C(\mu)\to \overline{C}(\mu)$ which induces a homotopy
equivalence of classifying spaces.

\begin{definition}
 Define a poset quiver $\overline{C}(\mu)$ by
 \begin{eqnarray*}
  \overline{C}(\mu)_0 & = & \Cr(\mu) \\
  \overline{C}(\mu)(c,c') & = & \overline{Q}(\mu)(c',c).
 \end{eqnarray*}
\end{definition}

Namely $\overline{C}(\mu)$ is a subquiver of $C(\mu)$ consisting of reduced
flow paths. Unfortunately $\overline{C}(\mu)$ is not closed under the
composition in $C(\mu)$. We need to take a reduction after the
concatenation. 

\begin{definition}
 \label{reduced_composition}
 For $\gamma \in \overline{C}(\mu)(c_{1},c_{2})$ and 
 $\delta \in \overline{C}(\mu)(c_{2},c_{3})$, 
 define $\delta\circ\gamma \in \overline{C}(\mu)(c_{1},c_{3})$ 
 by
 \[
  \delta\circ\gamma = r(\gamma*\delta).
 \]
\end{definition}

Note that 
\[
 \circ : \overline{C}(\mu)(c_{2},c_{3}) \times
 \overline{C}(\mu)(c_{1},c_{2}) \longrightarrow \overline{C}(\mu)(c_{1},c_{3}) 
\]
is a poset map as the composition of poset maps $r$ and $*$.

\begin{proposition}
\label{r_is_functor}
 The following diagram is commutative
 \[
 \begin{diagram}
  \node{C(\mu)(c_{2},c_{3})\times C(\mu)(c_{1},c_{2})} \arrow{e,t}{\circ}
  \arrow{s,l}{r\times r}
  \node{C(\mu)(c_{1},c_{3})} \arrow{s,r}{r} \\
  \node{\overline{C}(\mu)(c_{2},c_{3})\times \overline{C}(\mu)(c_{1},c_{2})}
  \arrow{e,b}{\circ} 
  \node{\overline{C}(\mu)(c_{1},c_{3}).}
 \end{diagram}
\]
\end{proposition}

\begin{proof}
 By Lemma \ref{r_is_poset_retraction}, $r(\delta)\preceq\delta$ and
 $r(\gamma)\preceq\gamma$. 
 Since $r$ and $*$ are poset maps by Lemma
 \ref{concatenation_is_poset_map} and Lemma \ref{r_is_poset_retraction},
 we have 
 \[
 r(\gamma) \circ r(\delta) = r(r(\delta) * r(\gamma)) \preceq r(\delta *
 \gamma) 
\]
 for $(\gamma,\delta) \in C(\mu)(c_{2},c_{3}) \times C(\mu)(c_{1},c_{2})$.
 On the other hand, $r(\delta * \gamma)$ is obtained from
 $\delta*\gamma$ by removing cells having indices in reducible
 intervals of $\delta*\gamma$. Note, however, it can be also obtained by
 removing cells indexed by reducible intervals in $\delta$ and $\gamma$,
 and then removing reducible intervals of the resulting path.
 In particular, $r(\delta*\gamma)$ is a subpath of
 $r(\delta) * r(\gamma)$.
 Since $r$ is a poset map and a retraction, we have 
 \[
 r(\delta * \gamma) = r(r(\delta*\gamma)) \preceq r(r(\delta)*r(\gamma))
 = r(\gamma) \circ r(\delta). 
 \]
\end{proof}

\begin{corollary}
 The following diagram is commutative
 \[
 \begin{diagram}
  \node{\overline{C}(\mu)(c_{3},c_{4})
 \times \overline{C}(\mu)(c_{2},c_{3}) \times
 \overline{C}(\mu)(c_{1},c_{2})} \arrow{e,t}{1\times\circ}
  \arrow{s,l}{\circ\times 1}
  \node{\overline{C}(\mu)(c_{3},c_{4}) \times
  \overline{C}(\mu)(c_{1},c_{3})} \arrow{s,r}{\circ} \\ 
  \node{\overline{C}(\mu)(c_{2},c_{4}) \times
  \overline{C}(\mu)(c_{1},c_{2})} \arrow{e,b}{\circ}
  \node{\overline{C}(\mu)(c_{1},c_{4})} 
 \end{diagram}
 \]
 for any critical cells $c_1,c_2,c_3$ and $c_4$.
\end{corollary}

\begin{proof}
 Since $r$ is surjective by Lemma \ref{r_is_poset_retraction} and is
 compatible with $\circ$ by Proposition \ref{r_is_functor}, the result
 follows from Lemma \ref{concatenation_is_associative}. 
\end{proof}

\begin{definition}[Reduced flow category]
 \label{reduced_flow_category_definition}
 The poset category obtained from the quiver $\overline{Q}(\mu)$ by
 using the composition in Definition \ref{reduced_composition} is called
 the \emph{reduced flow category} $\overline{C}(\mu)$.
\end{definition}

Proposition \ref{r_is_functor} implies that the reduction $r$ is
a functor of poset categories from $C(\mu)$ to $\overline{C}(\mu)$. 
The following is an immediate but important consequence of Proposition
\ref{adjoint_poset}. 

\begin{theorem}
 \label{reduction_is_homotopy_equivalence}
 The reduction $r : C(\mu) \to \overline{C}(\mu)$ induces a homotopy
 equivalence  
 \[
 B^{2}r : B^{2}C(\mu) \longrightarrow B^{2}\overline{C}(\mu).
 \]
\end{theorem}

\begin{proof}
 By Proposition \ref{adjoint_poset}, the reduction $r$ is a descending
 closure operator, which implies that it induces a
 deformation retraction on each classifying space of morphisms
 \[
  Br(c,c') : BC(\mu)(c,c') \rarrow{\simeq} B\overline{C}(\mu)(c,c')
 \]
 by Corollary \ref{DR_by_closure_operator}. 
 By applying Proposition \ref{homotopy_invariance_of_classifying_space},
 we see that $B^{2}r : B^{2}C(\mu) \to B^{2}\overline{C}(\mu)$ is a
 homotopy equivalence. 
\end{proof}

\subsection{The Collapsing Functor}
\label{collapsing_functor}

Theorem \ref{reduction_is_homotopy_equivalence} says that the right most
functor in (\ref{zigzag_of_functors}) induces a homotopy equivalence
between classifying spaces. The next step is to define a functor
$\tau:\overline{\FP}(\mu) \to \overline{C}(\mu)$ and show that it induces a
homotopy equivalence between classifying spaces.

Notice that $\overline{C}(\mu)$ is a poset-category, hence a
$2$-category. On the other hand, $\overline{\FP}(\mu)$ is a poset regarded
as a small category. We should regard $\overline{\FP}(\mu)$ as a
$2$-category whose $2$-morphisms are identities and try to construct a
$2$-functor $\tau : \overline{\FP}(\mu)\to\overline{C}(\mu)$.
On objects it is given by the target map
 \[
  \tau : \FP(\mu) \rarrow{} \Cr(\mu) = C(\mu)_0.
 \]
The restriction to $\overline{\FP}(\mu)$ is denoted by $\overline{\tau}$.

We would like to extend these maps to $2$-functors by regarding
the posets $\FP(\mu)$ and $\overline{\FP}(\mu)$ as small categories.
Unfortunately this is too much to expect.
The best we can do is the following.

\begin{proposition}
 \label{collapsing_functor_for_morphisms}
 The map $\tau$ can be extended to a normal colax
 functor\footnote{See Definition \ref{colax_functor_definition}.} 
 $\tau : \FP(\mu)\to C(\mu)$.
 Furthermore it induces a normal colax functor   
 $\overline{\tau} : \overline{\FP}(\mu)\to\overline{C}(\mu)$. 
\end{proposition}

\begin{proof}
 For a morphism $\gamma \preceq \gamma'$ in $\FP(\mu)$, 
 let $\varphi:\{0,\ldots,k\}\to\{0,\ldots,n+1\}$ be the embedding
 function. 
 Define a flow path $\tau(\gamma \preceq \gamma')$ to be the subsequence
 of $\gamma=(e_1,u_1,\ldots,e_n,u_n;\tau(\gamma))$ starting at $e_{k}$.  
 Note that we have $e_{k}\preceq \tau(\gamma')$ by Remark
 \ref{e_k_is_face_of_target} and we have 
 \[
 \tau(\gamma \preceq \gamma') = (e_{k},u_{k},\ldots, e_{n}, u_{n};\tau(\gamma))
 \in Q(\mu)(\tau(\gamma'),\tau(\gamma))=C(\mu)(\tau(\gamma),
 \tau(\gamma')).  
 \]
 Then $\tau$ sends the identity
 $\gamma \preceq \gamma$ to  $1_{\tau(\gamma)}$ by definition.

 For a sequence $\gamma_{1} \preceq \gamma_{2} \preceq \gamma_{3}$ in
 $\FP(\mu)$, let us show that
 \[
 \tau(\gamma_1\preceq\gamma_3) \preceq
 \tau(\gamma_2\preceq\gamma_3)\circ\tau(\gamma_1\preceq\gamma_2)
 \]
 in $C(\mu)(\tau(\gamma_1),\tau(\gamma_3))$.
 Let
 \begin{eqnarray*}
  \varphi_1 & : & \{0,\ldots,k_1\} \rarrow{} \{0,\ldots,n_2+1\} \\
  \varphi_2 & : & \{0,\ldots,k_2\} \rarrow{} \{0,\ldots,n_3+1\}
 \end{eqnarray*}
 be embedding functions for $\gamma_{1} \preceq \gamma_{2}$ and
 $\gamma_{2} \preceq \gamma_{3}$, respectively.
 By definition, the flow
 path $\tau(\gamma_{1} \preceq \gamma_{2})$ is the subsequence  
 of $\gamma_{1}=(e_1,u_1,\ldots,e_{n_1},u_{n_1};\tau(\gamma_1))$
 starting at the cell $e_{k_1}$ and $\tau(\gamma_2\preceq\gamma_{3})$
 is the subsequence of
 $\gamma_2=(e'_1,u'_1,\ldots,e'_{n_2},u'_{n_2};\tau(\gamma_2))$
 starting at the cell $e'_{k_2}$. 
 According to the proof of Proposition \ref{partial_order_on_FP}, on the
 other hand, the embedding function
 $\psi:\{0,\ldots,k_3\}\to\{0,\ldots,n_3+1\}$ for
 $\gamma_1\preceq\gamma_3$ is given by 
 \[
  \psi(i) = 
 \begin{cases}
  (\varphi_2\circ\varphi_1)(i), & i<k_3 \\
  \varphi_2(k_2)=n_3+1, & i=k_3
 \end{cases}
 \]
 where $k_3$ is the number with $\varphi_1(k_3-1)<k_2\le\varphi_1(k_3)$.
 Thus we have
 \begin{eqnarray*}
  \tau(\gamma_2\preceq\gamma_3)\ast  \tau(\gamma_1\preceq\gamma_2)& = &
   (e'_{k_2},u'_{k_2},\ldots,
   e'_{n_2},u'_{n_2},e_{k_1},u_{k_1}, \ldots, e_{n_1},u_{n_1},;
   \tau(\gamma_1)) \\  
  \tau(\gamma_1\preceq\gamma_3) & = & (e_{k_3},u_{k_3}, \ldots,
   e_{n_1},u_{n_1};\tau(\gamma_1)). 
 \end{eqnarray*}
 Let us rename cells in these flow paths as
 \begin{eqnarray*}
  \tilde{e}_j & = & e_{j+k_3-1} \\
  \tilde{u}_j & = & u_{j+k_3-1}
 \end{eqnarray*}
 in $\tau(\gamma_1\preceq\gamma_3)$
 and 
 \begin{eqnarray*}
  \tilde{e}'_j & = & 
   \begin{cases}
    e'_{j+k_2-1}, & 1\le j \le n_2-k_2+1 \\
    e_{j+k_1+k_2-n_2-2}, & n_2-k_2+2\le j\le n_1+n_2-k_1-k_2+2
   \end{cases} \\
  \tilde{u}'_j & = & 
   \begin{cases}
    u'_{j+k_2-1}, & 1\le j \le n_2-k_2+1 \\
    u_{j+k_1+k_2-n_2-2}, & n_2-k_2+2\le j\le n_1+n_2-k_1-k_2+2
   \end{cases}
 \end{eqnarray*}
 in $\tau(\gamma_2\preceq\gamma_3)\ast \tau(\gamma_1\preceq\gamma_2)$
 so that we have
 \begin{eqnarray*}
  \tau(\gamma_1\preceq\gamma_3) & = & (\tilde{e}_{1},\tilde{u}_{1},
   \ldots, \tilde{e}_{n_1-k_3+1},\tilde{u}_{n_1-k_3+1};\tau(\gamma_1)) \\ 
  \tau(\gamma_2\preceq\gamma_3)\ast \tau(\gamma_1\preceq\gamma_2) & = &
   (\tilde{e}'_{1},\tilde{u}'_{1}, \ldots,
   \tilde{e}'_{n_1+n_2-k_1-k_2+2},\tilde{u}'_{n_1+n_2-k_1-k_2+2};
   \tau(\gamma_1)).
 \end{eqnarray*}
 With this notation, what we need is a map
 \[
 \zeta : \{0,1,\ldots,n_1-k_3+2\} \rarrow{}
 \{0,1,\ldots,n_1+n_2-k_1-k_2+3\} 
 \]
 satisfying the following conditions:
 \begin{enumerate}
  \item $\zeta(0)=0$,
  \item $\tilde{u}_{j}=\tilde{u}'_{\zeta(j)}$ for $1\le j< n_1-k_3+2$,
  \item $\zeta(n_1-k_3+2)=n_1+n_2-k_1-k_2+3$,
  \item $\tilde{e}_j\preceq \tilde{e}'_p$ for $\zeta(j-1)<p\le\zeta(j)$ and
	$1\le j\le n_1-k_3+2$.
 \end{enumerate}
 Define $\zeta$ by
 \[
  \zeta(j) =
 \begin{cases}
  0, & j=0 \\
  \varphi_1(j+k_3-1)-k_2+1, & 1\le j\le k_1-k_3 \\
  j+n_2-k_1-k_2+k_3+1, & k_1-k_3+1\le j\le n_1-k_3+2.
 \end{cases}
 \]
 Let us verify that $\zeta$ is an embedding function for 
 \begin{equation}
 \tau(\gamma_{1} \preceq \gamma_{3})  \preceq \tau(\gamma_{2} \preceq
 \gamma_{3}) * \tau(\gamma_{1} \preceq \gamma_{2}). 
 \label{colax_condition}
 \end{equation}
 The first and the third conditions are obvious. 

 For the second condition, suppose $1\le j\le k_1-k_3$. Then
 \begin{eqnarray*}
  \tilde{u}_j & = & u_{j+k_3-1} \\
  \tilde{u}'_{\zeta(j)} & = & u''_{\zeta(j)+k_2-1} \\
  & = & u'_{\varphi_1(j+k_3-1)}.
 \end{eqnarray*}
 By the second condition for the embedding function $\varphi_1$, these
 cells coincide. 
 When $k_1-k_3+1\le j\le n_1-k_3+2$, we have
 \[
  n_2-k_2+2\le \zeta(j)\le n_1+n_2-k_1-k_2+3.
 \]
 and thus
 \begin{eqnarray*}
  \tilde{u}'_{\zeta(j)} & = & u_{\zeta(j)+k_1+k_2-n_2-2} \\
  & = & u_{j+n_2-k_1-k_2+k_3+1+k_1+k_2-n_2-2} \\
  & = & u_{j+k_3-1}.
 \end{eqnarray*}
 On the other hand, $\tilde{u}_j=u_{j+k_3-1}$ by definition and we have
 $\tilde{u}'_{\zeta(j)}=\tilde{u}_j$. 

 For the fourth condition, suppose $1\le j\le k_1-k_3$ and
 $\zeta(j-1)<p\le\zeta(j)$. In this case
 $\zeta(j-1)=\varphi_1(j+k_3-2)-k_2+1$ and
 $\zeta(j)=\varphi_1(j+k_3-1)-k_2+1$, which implies that
 \[
  \varphi_1(j+k_3-2) < p+k_2-1\le \varphi_1(j+k_3-1).
 \]
 The fourth condition for $\varphi_1$ implies that
 \[
  \tilde{e}_j = e_{j+k_3-1} \preceq e'_{p+k_2-1} = \tilde{e}'_{p}.
 \]
 When $k_1-k_3+1\le j\le n_1-k_3+2$, the condition
 $\zeta(j-1)<p\le\zeta(j)$ is equivalent to 
 $p=j+n_2-k_1-k_2+k_3+1$, in which case
 \begin{eqnarray*}
  \tilde{e}'_p & = & e_{p+k_1+k_2-n_2-2}   \\
  & = & e_{j+k_3-1}
 \end{eqnarray*}
 which coincides with $\tilde{e}_j=e_{j+k_3-1}$.
 And we obtain an embedding function $\zeta$ for (\ref{colax_condition}). 

 Suppose $\gamma\preceq\delta$ in $\overline{\FP}(\mu)$. Since
 $\tau(\gamma\preceq\delta)$ is defined as a subsequence of $\gamma$,
 $\tau(\gamma\preceq\delta)$ belongs to
 $\overline{C}(\mu)(\tau(\gamma),\tau(\delta))$. 
 For a sequence $\gamma_{1} \preceq \gamma_{2} \preceq \gamma_{3}$ in
 $\overline{\FP}(\mu)$, we have (\ref{colax_condition}) in
 $Q(\mu)(\tau(\gamma_3),\tau(\gamma_1))$. 
 Since the reduction $r$ is a poset map and
 $\tau(\gamma_1\preceq\gamma_3)$ is reduced, we obtain
 \[
 \tau(\gamma_1\preceq\gamma_3) = r(\tau(\gamma_{1} \preceq \gamma_{3}))
 \preceq r(\tau(\gamma_{2}\preceq 
 \gamma_{3}) \circ \tau(\gamma_{1} \preceq \gamma_{2})) =   
 \tau(\gamma_{2} \preceq \gamma_{3}) \circ \tau(\gamma_{1} \preceq
 \gamma_{2})).
 \]
 in $\overline{C}(\mu)(\tau(\gamma_{1}),\tau(\gamma_{3}))$. And we obtain
 a normal colax functor
 \[
  \overline{\tau}: \overline{\FP}(\mu) \rarrow{} \overline{C}(\mu).
 \]
\end{proof}

\begin{example}
 \label{triangle_example4}
 Consider the discrete Morse function in Example
 \ref{triangle_example2}. Its flow category is described in Example
 \ref{triangle_example3}. 
 
 The collapsing functor $\tau$ is given, on objects, by
 \begin{eqnarray*}
  \tau(\gamma_0) & = & [v_0] \\
  \tau(\gamma_1) & = & [v_0] \\
  \tau(\gamma_2) & = & [v_0] \\
  \tau(\gamma_{01}) & = & [v_0] \\
  \tau(\gamma_{02}) & = & [v_0] \\
  \tau(\gamma_{12}) & = & [v_1,v_2].
 \end{eqnarray*}
 We have six nontrivial order relations in $\FP(\mu)$, for which the
 collapsing functor is defined by
 \begin{eqnarray*}
  \tau(\gamma_0\prec\gamma_{01}) & = & ([v_0]) = \gamma_0 = 1_{[v_0]} \\
  \tau(\gamma_0\prec\gamma_{02}) & = & ([v_0]) = \gamma_0 = 1_{[v_0]} \\
  \tau(\gamma_1\prec\gamma_{01}) & = & ([v_0]) = \gamma_0 = 1_{[v_0]} \\
  \tau(\gamma_2\prec\gamma_{02}) & = & ([v_0]) = \gamma_0 = 1_{[v_0]} \\
  \tau(\gamma_1\prec\gamma_{12}) & = & \gamma_1 \\
  \tau(\gamma_2\prec\gamma_{12}) & = & \gamma_2.
 \end{eqnarray*}
\end{example}

By Lemma \ref{lax_classifying_spaces}, $\tau$ and $\overline{\tau}$
induce maps between normal colax 
classifying spaces
\begin{eqnarray*}
 B^{ncl}\tau & : & B^{ncl}\FP(\mu)=B\FP(\mu) \rarrow{}
 B^{ncl}C(\mu) \\ 
 B^{ncl}\overline{\tau} & : &
  B^{ncl}\overline{\FP}(\mu)=B\overline{\FP}(\mu) \rarrow{} 
 B^{ncl}\overline{C}(\mu). 
\end{eqnarray*}
Let us show that both $B^{ncl}\tau$ and $B^{ncl}\overline{\tau}$ are
homotopy equivalences. One of the most convenient tools for showing a
functor to 
be a homotopy equivalence is Quillen's Theorem A, which says that a
functor between small categories induces a homotopy equivalence between
classifying spaces if ``(homotopy) fibers'' are contractible.
There are variations, depending on the three
choices for (homotopy) fibers; $\tau\downarrow c$, $c\downarrow\tau$, 
and $\tau^{-1}(c)$. 
For functors between $2$-categories, analogous theorems have been proved.
Here we use Corollary \ref{prefibered_2QuillenA} by 
showing that $\tau$ is prefibered and each fiber $\tau^{-1}(c)$ is
contractible.

Let us first find explicit descriptions of homotopy fibers in order to
show that $\tau$ is prefibered.
Note that both left and right homotopy fibers are posets by
Corollary \ref{homotopy_fiber_poset}.

\begin{lemma}
 \label{homotopy_fiber_of_collapsing_functor}
 For a critical cell $c$, the posets $\tau\downarrow c$ and
 $c\downarrow\tau$ are given by 
 \begin{align*}
 \tau\downarrow c & =
   \set{(\gamma,\delta)\in \FP(\mu)^2}{\sigma(\delta)=c,
   \tau(\delta)=\tau(\gamma)} \\ 
  c\downarrow\tau & = \set{(\delta,\gamma)\in \FP(\mu)^2}{\tau(\gamma)=
  \sigma(\delta), 
   \tau(\delta)=c},  
 \end{align*}
 respectively.
 The partial orders $\preceq_{c}$ and $\preceq^{c}$ on
 $\tau\downarrow c$ and  $c\downarrow\tau$ are, respectively, given by
 \begin{align*}
 (\gamma,\delta) \preceq_{c} (\gamma',\delta') & \Longleftrightarrow
 \gamma\preceq\gamma' \text{ and } \delta'\ast\tau(\gamma\preceq\gamma')
 \preceq\delta. \\
 (\delta,\gamma) \preceq^{c} (\delta',\gamma') & \Longleftrightarrow
 \gamma\preceq\gamma' \text{ and } \tau(\gamma\preceq\gamma')\ast\delta
 \preceq \delta'. 
 \end{align*}
\end{lemma}

\begin{proof}
 Let us consider the case of the left homotopy fiber. We have
 \begin{eqnarray*}
  (\tau\downarrow c)_0 & = & \set{(\gamma,\delta)\in \FP(\mu)\times
   C(\mu)_1}{\delta\in C(\mu)(\tau(\gamma),c)} \\
  & = & \set{(\gamma,\delta)\in\FP(\mu)^2}{\delta\in Q(\mu)(c,\tau(\gamma))}
   \\ 
  & = &
   \set{(\gamma,\delta)\in \FP(\mu)^2}{\sigma(\delta)=c,
   \tau(\delta)=\tau(\gamma)}.  
 \end{eqnarray*}
 By Corollary \ref{homotopy_fiber_poset}, the partial order
 $\preceq_{c}$ on $\tau\downarrow c$ is given by
 \[
 (\gamma,\delta)\preceq_{c} (\gamma',\delta') \Longleftrightarrow
 \gamma\preceq\gamma' \text{ and } \delta'\ast\tau(\gamma\preceq\gamma')
 \preceq\delta. 
 \]

 The right homotopy fiber is given by
  \begin{eqnarray*}
  (c\downarrow\tau)_0 & = & \set{(\delta,\gamma)\in
   C(\mu)_1\times\FP(\mu)}{\delta\in C(\mu)(c,\tau(\gamma))} \nonumber
   \\ 
  & = & \set{(\delta,\gamma)\in
   \FP(\mu)^2}{\delta\in Q(\mu)(\tau(\gamma),c)} \nonumber \\
  & = &
   \set{(\delta,\gamma)\in \FP(\mu)^2}{\tau(\gamma)= \sigma(\delta),
   \tau(\delta)=c} \nonumber
 \end{eqnarray*}
 The partial order on this poset set is given by
 \[
 (\delta,\gamma)\preceq^{c} (\delta',\gamma') \Longleftrightarrow
 \gamma\preceq \gamma' \text{ and }
 \tau(\gamma\preceq\gamma')\ast \delta \preceq \delta'.
 \]
\end{proof}

\begin{proposition}
 \label{DR_to_genuine_fiber}
 For each $c\in \Cr(\mu)$, define maps
 $i_{c}:\tau^{-1}(c)\to c\downarrow\tau$ and
 $s_{c}: c\downarrow\tau\to\tau^{-1}(c)$
 by $i_c(\gamma)=(1_{c},\gamma)$ and
 $s_{c}(\delta,\gamma)=\gamma\ast\delta$, 
 respectively. Then the composition 
 \[
 \rho_{c} : c\downarrow\tau \rarrow{s_{c}} \tau^{-1}(c) \rarrow{i_{c}}
 c\downarrow\tau
 \]
 is a descending closure operator.
\end{proposition}

\begin{proof}
 Obviously $\rho_{c}\circ\rho_{c}=\rho_{c}$. Since
 $\gamma\ast\delta\preceq \gamma$ by Lemma
 \ref{order_and_composition} and 
 $\tau(\gamma\ast\delta\preceq\gamma)=\delta$, we have
 \[
  \rho_{c}(\delta,\gamma) =
 (1_{\tau(\gamma\ast\delta)},\gamma\ast\delta)\preceq^{c} 
 (\delta,\gamma)
 \]
 in $c\downarrow\tau$. 
\end{proof}

This implies that $\rho$ is the counit for the adjunction $i\dashv s$
and we obtain the following corollary.

\begin{corollary}
\label{tau_is_prefibered}
 The collapsing functor $\tau$ is prefibered.
\end{corollary}

%

Now we are ready to prove the following theorem.

\begin{theorem}
\label{homotopy_eq_collapsing}
 For an acyclic partial matching $\mu$ on a finite regular CW
 complex, the collapsing functor $\tau$ induces a homotopy 
 equivalence  
 \[
 B^{ncl}\tau : B\FP(\mu) \longrightarrow B^{ncl}C(\mu)
 \]
 between classifying spaces.
\end{theorem}

\begin{proof}
 By Corollary \ref{tau_is_prefibered} and Corollary
 \ref{prefibered_2QuillenA}, it suffices to show that 
 $B\tau^{-1}(c)$ is contractible for each $c\in \Cr(\mu)$.

 Recall from Proposition \ref{faithful_discrete_Morse_function} that we
 may choose a faithful and $\Z$-valued discrete Morse function $f$ whose
 associated partial matching is $\mu$. In particular, it is injective. 
 For each nonnegative integer $\ell$, define
 \[
  \tau^{-1}(c)_{\ell} = (f\circ\iota)^{-1}(\ell)\cap \tau^{-1}(c).
 \]
 Define a filtration on $\tau^{-1}(c)$ by
 \[
  F_{\ell}\tau^{-1}(c) = \bigcup_{i=0}^{\ell} \tau^{-1}(c)_{i}.
 \]
 We are going to show that $BF_{\ell}\tau^{-1}(c)$ deformation-retracts
 onto $BF_{\ell-1}\tau^{-1}(c)$ for all $\ell\ge 1$.

 By the injectivity of $f$, $f^{-1}(\ell)$ contains at most one
 element. There are three cases;
 \begin{enumerate}
  \item $f^{-1}(\ell)=\emptyset$
  \item $f^{-1}(\ell)=\{d\}$ with $d\in D(\mu_f)$ 
  \item $f^{-1}(\ell)=\{\mu_f(d)\}$ with $d\in D(\mu_f)$. 
 \end{enumerate}

 In the first case, $F_{\ell}\tau^{-1}(c)=F_{\ell-1}\tau^{-1}(c)$ and
 there is nothing to prove.  

 Suppose $\tau^{-1}(c)_{\ell}=\{d\}$ with $d\in D(\mu_f)$.
 Define a map
 \[
  m_{\ell} : F_{\ell}\tau^{-1}(c) \to F_{\ell}\tau^{-1}(c)
 \]
 by 
 \[
 m_{\ell}(\gamma) =
 \begin{cases}
  (\mu_{f}(d),\mu_{f}(d),e_2,u_2,\ldots, e_{n},u_{n};c), & \text{ if }
  \gamma=(d,\mu_{f}(d), e_2,u_2,\ldots, e_{n},u_{n};c), \\ 
  \gamma, & \text{ otherwise.}
 \end{cases} 
 \]
 Since $f(\mu_{f}(d))\le f(d)$, 
 $m_{\ell}(\gamma)$ belongs to $F_{\ell}\tau^{-1}(c)$. 
 Let us verify that this is an ascending closure operator on
 $F_{\ell}\tau^{-1}(c)$. 
 The embedding function for $\gamma\preceq m_{\ell}(\gamma)$ is given by
 the identity map. We have $m_{\ell}\circ m_{\ell}=m_{\ell}$ by
 definition. 
 It remains to show that this is a poset map. 

 Suppose $\gamma\preceq\gamma'$ with embedding function $\varphi$.
 In particular, we have $\iota(\gamma)\preceq\iota(\gamma')$. Since $f$
 is faithful, we have $f(\iota(\gamma))\le f(\iota(\gamma'))$. 

 When $\iota(\gamma)\neq d$, we have
 \[
  m_{\ell}(\gamma)=\gamma\preceq \gamma' \preceq m_{\ell}(\gamma').
 \]
 When $\iota(\gamma)=d$ the injectivity of $f$ implies that
 $\iota(\gamma'))=\iota(\gamma)=d$, for 
 $\ell=f(\iota(\gamma))\le f(\iota(\gamma'))\le \ell$. 
 Thus the same embedding function $\varphi$
 serves as an embedding function for
 $m_{\ell}(\gamma)\preceq m_{\ell}(\gamma')$. 
 And we have an ascending closure operator $m_{\ell}$ whose image is 
 $F_{\ell-1}\tau^{-1}(c)$.
 By Corollary \ref{DR_by_closure_operator}, $BF_{\ell-1}\tau^{-1}(c)$ is
 a deformation retract of $BF_{\ell}\tau^{-1}(c)$. 
 
 Let us consider the third case $\tau^{-1}(c)_{\ell}=\{\mu_{f}(d)\}$.
 Define
 \[
  b_{\ell} : F_{\ell}\tau^{-1}(c) \rarrow{} F_{\ell}\tau^{-1}(c)
 \]
 by 
 \[
  b_{\ell}(\gamma) = 
 \begin{cases}
  (e_2,u_2,\ldots, e_{n},u_{n};c), & \text{ if }
  \iota(\gamma)=\mu_{f}(d) \\
  \gamma, & \text{ otherwise}
 \end{cases}
 \]
 for $\gamma=(e_1,u_1,e_2,u_2,\ldots,e_n,u_n;c)$.
 As we have seen in Remark \ref{description_of_flow_path}, the
 faithfulness of $f$ implies that 
 \[
 \ell=f(\mu_{f}(d)) \ge f(\mu_{f}(d)) > f(e_2)\ge f(u_2) >
 \cdots > f(e_{n}) \ge f(u_{n})>f(\tau(\gamma))
 \]
 and thus $b_{\ell}(\gamma)\in F_{\ell}\tau^{-1}(c)$.

 Let us show that this is a descending closure operator.
 For a flow path $\gamma$ with $\iota(\gamma)=\mu_{f}(d)$, consider the
 map  $d^1:\{0,\ldots,\ell\}\to\{0,\ldots,\ell+1\}$ given by
 \[
  d^1(i) = 
 \begin{cases}
  0, & i=0 \\
  i+1, & i\ge 2.
 \end{cases}
 \]
 This serves as an embedding function for
 $b_{\ell}(\gamma)\preceq\gamma$, since $e_2\preceq\mu(d_1)$.
 The map $b_{\ell}$ is idempotent by definition. It remains to show that
 $b_{\ell}$ is a poset map. Suppose $\gamma\preceq\gamma'$ with
 embedding function
 $\varphi:\{0,\ldots,\ell(\gamma)+1\}\to\{0,\ldots,\ell(\gamma')+1\}$. 
 As is the case of $\tau^{-1}(c)_{\ell}=\{d\}$, $\gamma\preceq\gamma'$
 implies that $f(\iota(\gamma))\le f(\iota(\gamma'))$ and we have the
 following three cases. 
 \begin{enumerate}
  \item $f(\iota(\gamma))\le f(\iota(\gamma'))<\ell$
  \item $f(\iota(\gamma))<f(\iota(\gamma'))=\ell$
  \item $f(\iota(\gamma))=f(\iota(\gamma'))=\ell$
 \end{enumerate}

 When $f(\iota(\gamma'))<\ell$, we have
 $b_{\ell}(\gamma)=\gamma\preceq\gamma'=b_{\ell}(\gamma)$. 
 When $f(\iota(\gamma))=f(\iota(\gamma'))=\ell$, the restriction of the
 $\varphi$ to 
 $\{0,2,\ldots,\ell(\gamma)+1\}\to\{0,2,\ldots,\ell(\gamma')+1\}$
 is an embedding function for $b_{\ell}(\gamma)\preceq b_{\ell}(\gamma')$. 
 Suppose $f(\iota(\gamma))< f(\iota(\gamma'))=\ell$. We need to show
 that $\gamma\preceq b_{\ell}(\gamma')$. Let
 $\gamma=(e_1,u_1,e_2,u_2,\ldots,e_n,u_n;c)$. Since
 $f(\iota(\gamma))<\ell$, the injectivity of $f$ implies that
 $\iota(\gamma)=e_1\neq \mu(d)$. Since $e_1$ is either $\mu^{-1}(u_1)$
 or $u_1$, $\mu^{-1}(u_1)\neq d$. This implies that the embedding function
 $\varphi$ for $\gamma\preceq\gamma'$ satisfies $\varphi(1)>1$. Thus the
 same function $\varphi$ regarded as a map
 $\{0,1,\ldots,\ell(\gamma)+1\} \to \{0,2,3,\ldots,\ell(\gamma')+1\}$ is
 an embedding function for $b_{\ell}(\gamma)=\gamma\prec b_{\ell}(\gamma')$.
 And we obtain a descending closure operator
 $b_{\ell}: F_{\ell}\tau^{-1}(c)\to F_{\ell}\tau^{-1}(c)$ with image
 $F_{\ell-1}\tau^{-1}(c)$. 

 Again, by Corollary \ref{DR_by_closure_operator},
 $BF_{\ell-1}\tau^{-1}(c)$ is a deformation retract of
 $BF_{\ell}\tau^{-1}(c)$. 
\end{proof}

\begin{example}
 \label{triangle_example5}
 We continue with Example \ref{triangle_example4}.
 Let us compute $[v_0]\downarrow\tau$ and $[v_1,v_2]\downarrow\tau$. 
 As sets, we have
 \begin{eqnarray*}
  ([v_0]\downarrow\tau)_{0} & = & 
   \{(\gamma_{0},\gamma_{0}),
   (\gamma_{0},\gamma_{1}), (\gamma_{0},\gamma_{2}), 
   (\gamma_{0},\gamma_{01}), (\gamma_{0},\gamma_{02}),
   (\gamma_{1},\gamma_{12}), (\gamma_{2},\gamma_{12})\} \\
   ([v_1,v_2]\downarrow\tau)_{0} & = & \{(\gamma_{12},\gamma_{12})\}. 
 \end{eqnarray*}
 By
 $\gamma_{0}\prec \gamma_{01},\gamma_{02}$,
 $\gamma_{1}\prec \gamma_{01}$ and $\gamma_{2}\prec\gamma_{02}$,
 the partial order on $[v_0]\downarrow\tau$ is given by 
 \begin{eqnarray*}
  (\gamma_{0},\gamma_{0}) & \prec^{[v_0]} & (\gamma_{0},\gamma_{01}) \\
  (\gamma_{0},\gamma_{0}) & \prec^{[v_0]} & (\gamma_{0},\gamma_{02}) \\
  (\gamma_{0},\gamma_{1}) & \prec^{[v_0]} & (\gamma_{0},\gamma_{01}) \\
  (\gamma_{0},\gamma_{2}) & \prec^{[v_0]} & (\gamma_{0},\gamma_{02}).
 \end{eqnarray*}
 By $\gamma_{1}\prec\gamma_{12}$ and $\gamma_{2}\prec\gamma_{12}$, we
 have
 \begin{eqnarray*}
  (\gamma_{0},\gamma_{1}) & \prec^{[v_0]} & (\gamma_{1},\gamma_{12}) \\
  (\gamma_{0},\gamma_{2}) & \prec^{[v_0]} & (\gamma_{2},\gamma_{12}).  
 \end{eqnarray*}
 The Hasse diagram of $[v_0]\downarrow\tau$ is
 given by 
 
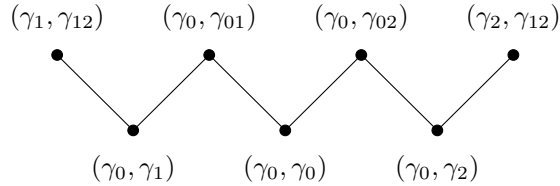
\begin{center}
  \begin{tikzpicture}
   \draw [fill] (-2,0) circle (2pt);
   \draw (-2,-0.5) node {$(\gamma_{0},\gamma_{1})$};

   \draw [fill] (0,0) circle (2pt);
   \draw (0,-0.5) node {$(\gamma_{0},\gamma_{0})$};

   \draw [fill] (2,0) circle (2pt);
   \draw (2,-0.5) node {$(\gamma_{0},\gamma_{2})$};

   \draw [fill] (-3,1) circle (2pt);
   \draw (-3,1.5) node {$(\gamma_{1},\gamma_{12})$};

   \draw [fill] (-1,1) circle (2pt);
   \draw (-1,1.5) node {$(\gamma_{0},\gamma_{01})$};

   \draw [fill] (1,1) circle (2pt);
   \draw (1,1.5) node {$(\gamma_{0},\gamma_{02})$};

   \draw [fill] (3,1) circle (2pt);
   \draw (3,1.5) node {$(\gamma_{2},\gamma_{12})$};

   \draw (0,0) -- (-1,1);
   \draw (0,0) -- (1,1);
   \draw (-2,0) -- (-3,1);
   \draw (-2,0) -- (-1,1);
   \draw (2,0) -- (3,1);
   \draw (2,0) -- (1,1);
  \end{tikzpicture}
%
%
%
%
%
%
%
 \end{center}
 Thus both $B([v_0]\downarrow\tau)$ and $B([v_0,v_1]\downarrow \tau)$
 are contractible.
 Note that $\tau^{-1}([v_0])$ is the subposet given by the zigzag
 between $(\gamma_{0},\gamma_{1})$ and $(\gamma_{0},\gamma_{2})$ and
 embedded in $[v_0]\downarrow\tau$ as a deformation retract.
\end{example}

\subsection{The Face Poset of Stable Subdivision}
\label{stable_subdivision}


We have constructed a subdivision $\Sd_{\mu}(X)$ of $X$ by using flow
paths in \S\ref{continuous_path}. It turns out that the face poset of
the stable subdivision $\Sd_{\mu}(X)$ is isomorphic to the poset of
reduced flow paths $\overline{\FP}(\mu)$. 

The aim of this section is to complete the
proof of Theorem \ref{main1} by proving this fact.
To this end, we need to understand relations between partial order on
$\overline{\FP}(\mu)$ and the deformation retraction
$R_{u}:\overline{u}\to d^{c}$ defined for each matched pair
$d\prec_{1}u=\mu(d)$. 

\begin{definition}
 \label{stable_subspace_definition}
 For a reduced flow path $\gamma=(e_1,u_1,\ldots,e_n,u_n;c)$, define a
 sequence of subpaths $\gamma^{(1)},\ldots,\gamma^{(n+1)}$ of $\gamma$ by
 \[
  \gamma^{(i)} = (e_i,u_i,\ldots,e_n,u_n;c)
 \]
 for $i\le n$ and $\gamma^{(n+1)}=(c)$.

 Let $i_1,\ldots,i_k$ be the indices of $e_i$'s with $e_i\in D(\mu)$ and
 define
 \[
  W^{s}_{\gamma} = \bigcup_{i=1}^{n+1} e_{\gamma^{(i)}} \cup
 \bigcup_{\ell=1}^{k} e_{u(\gamma^{(i_{\ell})})},
 \]
 for $\ell=1,\ldots,k$, where $u(\gamma^{(i_{\ell})})$ is the operation
 on flow paths defined in Example \ref{upgrading_flow_path}.
 This is called the \emph{stable subspace along $\gamma$}.
\end{definition}

\begin{example}
 Consider the partial matching on a $2$-simplex in Example
 \ref{subpath_example}, in which
 we have shown $\gamma\prec\delta$ 
 for the flow paths
 $\gamma = (e_1,e_2,e_5,e_6;e_7)=(d_1,u_1,d_2,u_2;c)$ and
 $\delta=(e_1,e_2,e_4,e_4,e_5,e_6;e_7)=(d'_1,u'_1,u'_2,u'_2,d'_3,u'_3;c)$. 

 \begin{figure}[ht]
  \begin{center}
   \begin{tikzpicture}
    \draw (2,0) -- (0,2) -- (0,0) -- (2,0);

    \draw [->] (1,1) -- (0.5,0.5);
    \draw [->] (-0.1,2) -- (-0.1,1);
    \draw [->] (0,-0.1) -- (1,-0.1);

    \draw [fill] (2,0) circle (2pt);
    \draw [fill] (0,2) circle (2pt);
    \draw [fill] (0,0) circle (2pt);

    \draw (1.2,1.2) node {$e_1$};
    \draw (0.5,0.8) node {$e_2$};
    \draw (-0.4,2) node {$e_3$};
    \draw (-0.4,1) node {$e_4$};
    \draw (-0.3,-0.3) node {$e_5$};
    \draw (1,-0.4) node {$e_6$};
    \draw (2,-0.4) node {$e_7$};    
   \end{tikzpicture}  
%
%
%
  \end{center}
  \caption{A partial matching on $2$-simplex}
  \label{stable_subspace_example_figure1}
 \end{figure}
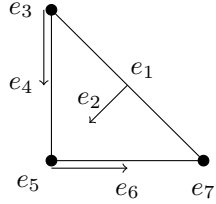
 The subpaths in Definition \ref{stable_subspace_definition} for
 $\gamma$ are
 \begin{align*}
  \gamma^{(1)} & = (e_1,e_2,e_5,e_6;c) \\
  \gamma^{(2)} & = (e_5,e_6;c) \\
  \gamma^{(3)} & = (c).
 \end{align*}
 Thus
 \[
 W^{s}_{\gamma} = e_{\gamma^{(1)}}\cup e_{\gamma^{(2)}} \cup
 e_{\gamma^{(3)}} \cup e_{u(\gamma^{(1)})}\cup e_{u(\gamma^{(2)})},
 \]
 which is the one dimensional complex drawn by dotted lines in Figure
 \ref{stable_subspace_example_figure2}.

  \begin{figure}[ht]
  \begin{center}
   \begin{tikzpicture}
    \draw (4,0) -- (0,4) -- (0,0);

    \draw [dotted] (2,2) -- (0,0);
    \draw [dotted] (0,0) -- (4,0);
    
    \draw [fill] (4,0) circle (2pt);
    \draw [fill] (0,4) circle (2pt);
    \draw [fill] (0,0) circle (2pt);
    \draw [fill] (2,2) circle (2pt);
    
    \draw (2.8,2.4) node {$e_{\gamma^{(1)}}=e_{\gamma}$};
    \draw (0.7,1.65) node {$e_{u(\gamma^{(2)})}$};
    \draw (-0.4,-0.4) node {$e_{\gamma^{(2)}}$};
    \draw (2,-0.5) node {$e_{u(\gamma^{(2)})}$};
    \draw (4,-0.5) node {$e_{\gamma^{(3)}}$};    
   \end{tikzpicture}  
%
%
%
  \end{center}
  \caption{Stable subspace along $\gamma$}
  \label{stable_subspace_example_figure2}
 \end{figure}
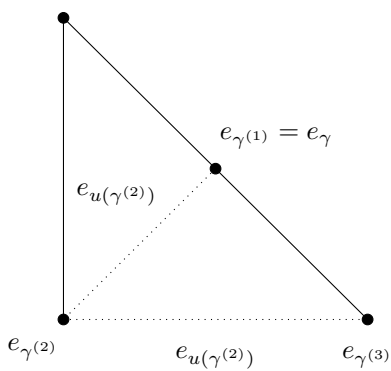
 Note that we need to add $e_{u(\gamma^{(1)})}$ and
 $e_{u(\gamma^{(2)})}$ to obtain a connected region.

 Similarly 
  \begin{align*}
  \delta^{(1)} = \delta & = (e_1,e_2,e_4,e_4,e_5,e_6;c) \\
  \delta^{(2)} & = (e_4,e_4,e_5,e_6;c) \\
  \delta^{(3)} & = (e_5,e_6;c) \\
  \delta^{(4)} & = (c).
 \end{align*}
 Thus
 \[
 W^{s}_{\delta} = e_{\delta^{(1)}}\cup e_{\delta^{(2)}} \cup
 e_{\delta^{(3)}}\cup 
 e_{\delta^{(4)}} \cup e_{u(\delta^{(1)})} \cup e_{u(\delta^{(3)})},
 \]
 which is the union of the shaded area and dotted lines in Figure
 \ref{stable_subspace_example_figure3}.

  \begin{figure}[ht]
  \begin{center}
   \begin{tikzpicture}
    \draw (4,0) -- (0,4) -- (0,0);

    \draw [fill,lightgray] (0,4) -- (0,0) -- (2,2) -- (0,4);
    \draw [dotted] (0,4) -- (2,2);
    \draw [dotted] (0,4) -- (0,0);
    \draw [dotted] (0,0) -- (4,0);
    
    \draw (2,2) -- (0,0);

    \draw [fill] (4,0) circle (2pt);
    \draw [fill] (0,4) circle (2pt);
    \draw [fill] (0,0) circle (2pt);

    \draw (2.2,3.2) node {$e_{\delta^{(1)}}=e_{\delta}$};
    \draw (0.88,2) node {$e_{u(\delta^{(1)})}$};
    \draw (-0.6,2) node {$e_{\delta^{(2)}}$};
    \draw (-0.4,-0.4) node {$e_{\delta^{(3)}}$};
    \draw (2,-0.5) node {$e_{u(\delta^{(3)})}$};
    \draw (4,-0.5) node {$e_{\delta^{(4)}}$};    
   \end{tikzpicture}  
%
%
%
%
  \end{center}
  \caption{Stable subspace along $\delta$}
  \label{stable_subspace_example_figure3}
  \end{figure}
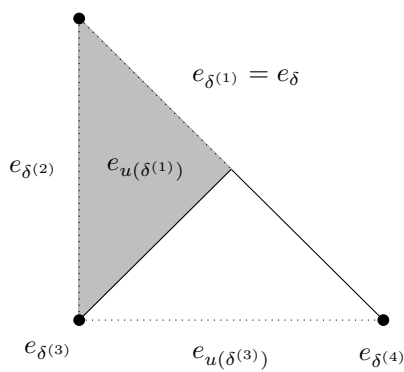

 From these figures, we see that
 $W_{\gamma}^s\subset \overline{W_{\delta}^s}$. 
\end{example}

\begin{lemma}
 For $x\in e_{\gamma}$, the continuous flow $L_{x}$ stays in
 $W_{\gamma}^{s}$. Furthermore the flows $\{L_{x}\}_{x\in\iota(\gamma)}$
 can be glued together to give rise to a continuous flow on
 $W_{\gamma}^{s}$.
 \[
  L_{\gamma}: W_{\gamma}^{s}\times [0,h]\rarrow{} W_{\gamma}^{s}.
 \]
\end{lemma}

\begin{proof}
 The first statement is obvious from the definition.
 For a reduced flow path $\gamma=(e_1,u_1,\ldots,e_n,u_n;c)$, let
 $h=h_x$ for $x\in e_{u(\gamma)}$. This is independent of $x$, since the
 initial cell of $u(\gamma)$ belongs to $D(\mu)$.

 Similarly define $h_2=h_{y}$ for $y\in e_{u(\gamma_2)}$ and define a
 map 
 \[
  L_1 : W_{\gamma}^{s}\times[h_2,h] \rarrow{} W_{\gamma}^{s}
 \]
 by
 \[
  L_1(x,t) = L_{d_1,u_1}(x,t-h_2),
 \]
 where $d_1=\mu^{-1}(u_1)$. This is a deformation retraction onto
 $W_{\gamma_2}^{s}$. By iterating this process, we obtain a sequence of
 homotopies
 $\{L_i:W_{\gamma_i}^{s}\times[h_{i+1},h_i]\to W_{\gamma_i}^{s}\}$. By
 concatenating these homotopies we obtain a continuous map
 \[
  L : W_{\gamma}^{s}\times[0,h] \rarrow{} W_{\gamma}^{s}
 \]
 whose restriction to $\{x\}\times[0,h_{x}]$ coincides with $L_{x}$.
\end{proof}

By using $L_{\gamma}$, we have the following extension of the
description of Lemma \ref{e_gamma_by_L}.

\begin{lemma}
 \label{more_inductive_e_gamma}
 For a reduced flow path $\gamma=(e_1,u_1,\ldots,e_n,u_n;c)$, we have
 \[
  e_{\gamma^{(k)}} = \set{x\in e_{k}}{\exists t \text{ s.t. }
 L_{\gamma}(x,t)\in e_{\gamma^{(k+1)}}}
 \]
 for $1\le k\le n$.
\end{lemma}

\begin{proof}
 By definition.
\end{proof}

The stable subspace $W^{s}(c)$ introduced in Definition
\ref{stable_subdivision_definition} for a critical cell $c$ decomposes
into a union
\[
 W^{s}(c) = \bigcup_{\tau(\gamma)=c} W^{s}_{\gamma}.
\]
The flows $L_{\gamma}$ for $\gamma$ with $\tau(\gamma)=c$ can be glued
together to give rise to a flow on $W^{s}(c)$
\[
 L_{c} : W^{s}(c)\times[0,h] \rarrow{} W^{s}(c). 
\]

The main result of this section is the following theorem,  
which establishes that the map $\overline{\FP}(\mu) \to F(\Sd_{\mu}(X))$
 given by $\gamma \mapsto e_{\gamma}$ is an isomorphism of posets. 

\begin{proposition}
 \label{face_poset_of_stable_subdivision}
 For $\gamma,\delta\in\overline{\FP}(\mu)$, 
we have $\gamma\preceq \delta$ if and only if
 $e_{\gamma}\subset \overline{e_{\delta}}$.
\end{proposition}

Assume the existence of a pair of reduced flow paths 
 \begin{eqnarray*}
 \gamma & = & (e_{1},u_{1},e_{2}, u_{2},\ldots, e_{n}, u_{n}; c), \text{ and} \\
 \delta & = & (e_{1}',u_{1}', e_{2}', u_{2}', \ldots, e_{m}', u_{m}'; c')  
 \end{eqnarray*}
 so that $\gamma \preceq \delta$ holds with the embedding function
 $\varphi:\{0,1,\ldots,k\}\to\{0,1,\ldots,m+1\}$. We also denote
 $d_i=\mu^{-1}(u_i)$ and $d'_j=\mu^{-1}(u'_j)$ for each $i$ and $j$, and recall
for the reader's convenience the fact that a deformation retraction 
$R_u: \overline{u} \to \partial u \setminus d$ has been chosen for each
matched cell pair $u = \mu(d)$ and used to define a continuous flow
in the sense of Definition \ref{continuous_flow_definition}.

\begin{lemma}
\label{path_range_retraction_invariance}
The cell $e_j$ is fixed pointwise by the deformation retraction 
$R_{u'_\ell}:\overline{u'}_\ell \to \partial u'_\ell \setminus d'_\ell$
for each $1 \leq j \leq k$ and $\phi(j-1) \leq \ell < \phi(j)$.
\end{lemma}
\begin{proof} 
Since $\phi$ is an embedding function, we have $e_j \preceq e'_p$
for $\phi(j-1) < p \leq \phi(j)$. But since $\delta$ is a reduced flow 
path, any such $e'_p$ is a face of the subsequent $u'_{p+1}$ different 
from $d'_{p+1}$, and all of its points must therefore be fixed by 
$R_{u'_{p+1}}$. For these values of $p$, note that $\ell = p+1$ 
satisfies $\phi(j-1) \leq \ell < \phi(j)$.
\end{proof}

\begin{proposition}
\label{reverse_induction_path}
 For every $1 \leq j \leq k$ and $\phi(j-1) \leq \ell < \phi(j)$, the 
 following statements hold.
 \begin{enumerate}
  \item If $e_{\gamma^{(j)}} \subset \overline{e_{\delta^{(\ell+1)}}}$, 
	then  
	$e_{\gamma^{({j)}}} \subset \overline{e_{\delta^{(\ell)}}}$ 
	for every $\phi(j-1) \leq \ell < \phi(j)$.
  \item If $e_{\gamma^{(j)}} \subset \overline{e_{\delta^{(\phi(j-1))}}}$, 
	then $e_{\gamma^{({j-1)}}} \subset
	\overline{e_{\delta^{(\phi(j-1)-1)}}}$. 
 \end{enumerate}
\end{proposition}

\begin{proof}
 To see that (1) holds, recall from Lemma \ref{more_inductive_e_gamma} 
 that $e_{\delta^{(\ell)}}$ consists of precisely those points of
 $e'_\ell$ which are mapped to $e_{\delta^{(\ell+1)}}$ by the
 deformation retraction $R_{u'_{\ell}}$. By assumption,
 $e_j \subset \overline{e_{\delta^{(\ell+1)}}}$ 
 and by Lemma \ref{path_range_retraction_invariance}, the cell $e_j$ is
 fixed by these retractions for all $\ell$ in the stated range. Thus, we
 have
 \[
 e_{\gamma^{(j)}} \subset e_j \subset \overline{e_{\delta^{(\ell)}}}.
 \] 
 Turning to (2), note that $u_{j-1} = u'_{\phi(j-1)}$ since $\phi$ is
 an embedding function. Therefore, the retraction $R_{u_{j-1}}$ which
 defines $e_{\gamma^{(j-1)}}$ from $e_{\gamma^{(j)}}$ coincides exactly
 with the retraction $R_{u'_{\phi(j-1)}}$ which similarly defines
 $e_{\delta^{(\phi(j-1)-1)}}$. 
 The desired inclusion
 $e_{\gamma^{({j-1)}}} \subset \overline{e_{\delta^{(\phi(j-1)-1)}}}$ 
 now follows from
 $e_{\gamma^{({j)}}} \subset \overline{e_{\delta^{(\phi(j-1))}}}$.
\end{proof}

\begin{remark}
 It is worth noting that the argument in the above proof can be used to
 prove that $L_{\gamma}$ is a part of $L_{\delta}$ when
 $\gamma\preceq\delta$. Namely, if $x\in e_{\gamma}$,
 $L_{\delta}(x,t)=L_{\gamma}(x,t')$ for some $t'$.
\end{remark}

 Recall that $e_k \prec c'$ because $\phi$ is an embedding
 function. Proposition  
 \ref{reverse_induction_path}, along with the initial condition 
 \[
 e_{\gamma^{(k)}} \subset e_k \subset \overline{c'} =
 \overline{e_{\delta^{(\phi(k))}}} 
 \]
 proves one half of Proposition \ref{face_poset_of_stable_subdivision}
 by reverse-induction on $j \in \{1, \ldots, k\}$
 (recall that $\gamma = \gamma^{(0)}$ and similarly for $\delta$). 
 Conversely, suppose that $e_{\gamma}\subset\overline{e_{\delta}}$. In
 order to prove  
 $\gamma\preceq\delta$, we need to find an embedding function $\varphi$.

 Define a sequence of nonnegative
 integers $i_0=0,i_1,i_2,\ldots$ inductively by
 \[
  i_{j} =
  \begin{cases}
  \max \set{p>i_{j-1}}{e_{\gamma^{(j)}}\subset
   \overline{e_{\delta^{(p)}}}}, & \text{ if}
   \set{p>i_{j-1}}{e_{\gamma^{(j)}}\subset
   \overline{e_{\delta^{(p)}}}} \neq\emptyset, \\
   m+1, & \text{ otherwise}
  \end{cases}
 \]
 and set $k = \min \set{j}{i_{j}=m+1}$. Then we obtain a strictly
 increasing function
 \[
  \varphi : \{0,\ldots, k\} \rarrow{} \{0,\ldots,m+1\}
 \]
 by $\varphi(j)=i_j$.

 Let us verify that this is an embedding function for
 $\gamma\preceq\delta$. By definition, it is 
 a strictly increasing function with $\varphi(0)=0$ and $\varphi(k)=m+1$.
 It remains to prove that
 \begin{enumerate}
  \item $u_j=u'_{\varphi(j)}$ for each $1\le j<k$, and
  \item for each $1\le j\le k$, $e_j\preceq e'_{p}$ for all
	$\varphi(j-1)<p\le\varphi(j)$. 
 \end{enumerate}

 The second part is immediate. By definition,
 $e_{\gamma^{(j)}}\subset \overline{e_{\delta^{(p)}}}$ for
 $\varphi(j-1)<p\le\varphi(j)$. But $e_{\gamma^{(j)}}\subset e_j$ and
 $e_{\delta^{(p)}}\subset\overline{e'_{p}}$,
 which imply that $e_j\cap \overline{e'_{p}}\neq \emptyset$. Thus we have
 $e_{j}\preceq e'_{p}$.

 For the first part, let us first prove that, for each $j$,
 $e_{\gamma^{(j)}}\not\subset (d'_{\varphi(j)})^{c}$. Suppose
 $e_{\gamma^{(j)}}\subset (d'_{\varphi(j)})^{c}$. Then
 $R_{u'_{\varphi(j)}}(e_{\gamma^{(j)}})=e_{\gamma^{(j)}}$ by the
 definition of the retraction
 $R_{u'_{\varphi(j)}}: \overline{u'_{\varphi(j)}}\to (d'_{\varphi(j)})^{c}$.
 Since $e_{\gamma^{(j)}}\subset\overline{e_{\delta^{(\varphi(j))}}}$, we
 have
 \[
 e_{\gamma^{(j)}} = R_{u'_{\varphi(j)}}(e_{\gamma^{(j)}}) \subset
 R_{u'_{\varphi(j)}}(\overline{e_{\delta^{(\varphi(j))}}}) =
 \overline{e_{\delta^{(\varphi(j)+1)}}}, 
 \]
 which contracts to our choice of $\varphi(j)$. Hence
 $e_{\gamma^{(j)}}\not\subset (d'_{\varphi(j)})^{c}$ or
 \[
 e_{\gamma^{(j)}} \subset
 \overline{u'_{\varphi(j)}} \setminus (d'_{\varphi(j)})^{c}
 = d'_{\varphi(j)}\cup u'_{\varphi(j)}.
 \]
 If $e_{\gamma^{(j)}}\subset u'_{\varphi(j)}$,
 $u'_{\varphi(j)}\cap e_{j}\neq\emptyset$ and we obtain
 $e_{j}=u_{j}=u'_{\varphi(j)}$. 
 When $e_{\gamma^{(j)}}\subset d'_{\varphi(j)}$,
 $d'_{\varphi(j)}\cap e_{j}\neq\emptyset$ and we have
 $e_{j}=d_{j}=d'_{\varphi(j)}$. Apply the matching $\mu$, and we obtain
 $u_{j}=u'_{\varphi(j)}$. 

This concludes the proof of Proposition
\ref{face_poset_of_stable_subdivision}, 
and we have the following desired consequence.

\begin{corollary}
 \label{reduced_flow_path_and_stable_subdivision}
 The map $\overline{\FP}(\mu) \to F(\Sd_{\mu}(X))$
 given by $\gamma \mapsto e_{\gamma}$
 is an isomorphism of posets.
\end{corollary}


\begin{proof}[Proof of Theorem \ref{main1}]
 The isomorphism in Corollary
 \ref{reduced_flow_path_and_stable_subdivision} induces a
 homeomorphism between classifying spaces
 $B\overline{\FP}(\mu_{f}) \cong  BF(\Sd_{\mu_{f}}(X))$.
 It is well known that the composition $BF$ is nothing but the
 barycentric subdivision\footnote{See, for example, section 12.4
 of Bj{\"o}rner's article \cite{Bjorner95} or section 10.3.5 of Kozlov's
 book \cite{KozlovCombinatorialAlgebraicTopology}.} $\Sd$ and we
 obtain a homeomorphism  
 \[
 B\overline{\FP}(\mu_{f}) \cong \Sd(\Sd_{\mu_{f}}(X)) \cong X.
 \]
 By combining the homotopy equivalences in Theorem
 \ref{homotopy_eq_collapsing}, Theorem
 \ref{classifying_space_of_2-category}, and  Theorem
 \ref{reduction_is_homotopy_equivalence} with this homeomorphism, we
 obtain a homotopy equivalence
 \[
  X \cong B\overline{\FP}(\mu_{f}) \simeq B^{ncl}\overline{C}(\mu_{f})
 \simeq B^2\overline{C}(\mu_{f}) \simeq B^2C(\mu_{f}).
 \]
\end{proof}

\appendix

\section{Homotopy Theory of Small Categories}
\label{category_theory}

Our main technical tool in this paper is homotopy theory of small
categories, including $2$-categories.
We collect definitions and important properties in
homotopy theory of small categories used in this paper for the
convenience of the reader. 

Our main references include \S11
of May's book \cite{May72}, Dugger's exposition
\cite{DuggerHomotopyColimit} on homotopy colimits, and
Segal's papers \cite{SegalBG,Segal73-1,Segal74-1}.

\subsection{Simplicial Sets and Simplicial Spaces}
\label{simplicial_space}

In homotopy theory of small categories, each small category $C$ is made
into a topological space by the classifying space construction, which is
defined as the geometric realization of a simplicial set $NC$, called
the nerve of $C$. When $C$ is a topological category, the nerve $NC$ is
a simplicial space.

In this section, we recall homotopy theoretic properties of simplicial
spaces, including simplicial sets.

\begin{definition}
 A \emph{simplicial object} in a category $\bm{C}$ consists of a
 sequence of objects $\{X_n\}_{n=0,1,\ldots}$ in $\bm{C}$ and morphisms
 \begin{eqnarray*}
  d_i & : & X_n \rarrow{} X_{n-1} \\
  s_i & : & X_n \rarrow{} X_{n+1}
 \end{eqnarray*}
 for $0\le i\le n$ satisfying the following relations:
 \begin{enumerate}
  \item $d^j\circ d^i = d^i\circ d^{j-1}$ for $i<j$
  \item $s^j\circ d^i = d^i\circ s^{j-1}$ for $i<j$
  \item $s^j\circ d^j = 1 = s^j\circ d^{j+1}$
  \item $s^j\circ d^i = d^{i-1}\circ s^j$ for $i>j+1$
  \item $s^j\circ s^i = s^i\circ s^{j+1}$ for $i\le j$.
 \end{enumerate}

 Simplicial objects in the categories of sets and
 topological spaces are called \emph{simplicial sets}
 and \emph{simplicial spaces}, respectively.
\end{definition}

It is convenient to introduce the following small category.

\begin{definition}
 The category of isomorphism classes of finite totally ordered sets and
 order-preserving maps is denoted by $\Delta$. The object of cardinality
 $n+1$ is denoted by $[n]$ and is identified with the subposet
 $0<1<\cdots<n$ of $\Z$.

 The subcategory of injective morphisms is denoted by $\Delta_{\inj}$.
\end{definition}

\begin{lemma}
 \label{simplicial_set_as_functor}
 The category of simplicial sets is isomorphic to the category
 $\Cats(\Delta^{\op},\Sets)$ of contravariant functors from $\Delta$ to
 $\Sets$.  Similarly the category of simplicial spaces is isomorphic to
 the category $\Cats(\Delta^{\op},\Spaces)$ of contravariant
 functors from $\Delta$ to $\Spaces$. 
\end{lemma}

\begin{definition}
 We denote the categories of simplicial sets and simplicial spaces
 by $\Sets^{\Delta^{\op}}$ and $\Spaces^{\Delta^{\op}}$,
 respectively. We regard $\Sets^{\Delta^{\op}}$ as a full subcategory of
 $\Spaces^{\Delta^{\op}}$ consisting of simplicial spaces with discrete
 topology. 

 For a simplicial space $X$ and a morphism $\varphi:[m]\to [n]$ in
 $\Delta$, the induced map is denoted by $\varphi^*: X_n\to X_m$.
\end{definition}

The following variation is useful when we study acyclic categories.

\begin{definition}
 A functor $X:\Delta^{\op}_{\inj}\to \category{Top}$ is called a
 \emph{$\Delta$-set}. 
\end{definition}

\begin{remark}
 A $\Delta$-set $X$ can be regarded as a ``simplicial set without
 degeneracies'', i.e.\ it consists of a sequence of sets $\{X_n\}$
 together with maps $d_i: X_n\to X_{n-1}$ satisfying the same
 relations as in the definition of simplicial sets.
 See \cite{Rourke-Sanderson1971-1} by Rourke and
 Sanderson, for more details on $\Delta$-sets.
\end{remark}

There are two popular ways to form a topological space from a simplicial
space.

\begin{definition}
 For a simplicial space $X$, define
 \begin{eqnarray*}
  |X| & = & \quotient{\coprod_{n=0}^{\infty} X_n\times\Delta^n}{\sim} \\
  \|X\| & = &
   \quotient{\coprod_{n=0}^{\infty} X_n\times\Delta^n}{\sim_{\inj}},
 \end{eqnarray*}
 where $\Delta^n$ is the standard $n$-simplex whose vertices are
 identified with elements in $[n]$ and $\sim$ is the
 equivalence relation generated by 
 $(\varphi^*(x),s) \sim (x,\varphi_*(s))$ for $x\in X_{n}$,
 $s\in\Delta^{m}$, and $\varphi\in\Delta([m],[n])$. 
 The map $\varphi_* : \Delta^m\to \Delta^n$ is the affine map induced by
 the map between vertices $\varphi:[m]\to [n]$.
 The equivalence relation $\sim_{\inj}$ is generated by 
 $(\varphi^*(x),s) \sim (x,\varphi_*(s))$ for
 $\varphi\in\Delta_{\inj}([m],[n])$. 

 $|X|$ and $\|X\|$ are called the \emph{(geometric) realization} and the
 \emph{fat realization} of $X$, respectively.
 The canonical quotient map from the fat realization to the 
 realization is denoted by 
 \[
  q_{X} : \|X\| \rarrow{} |X|.
 \]
\end{definition}

One of the advantages of the fat realization $\|X\|$ is the following
homotopy invariance.

\begin{proposition}
 For a morphism $f: X\to Y$ of simplicial spaces whose $n$-th stage
 $f_n : X_n\to Y_n$ is a homotopy equivalence for all $n$, the induced
 map on fat realizations
 \[
  \|f\| : \|X\| \rarrow{} \|Y\|
 \]
 is a homotopy equivalence.
\end{proposition}

\begin{proof}
 See Appendix A of Segal's paper \cite{Segal74-1}, for example.
\end{proof}

Thus the realization is homotopy invariant for simplicial spaces $X$
whose realization $|X|$ is homotopy equivalent to the fat realization
$\|X\|$. The following is a sufficient condition for the quotient
$q_X:\|X\|\to|X|$ to be a homotopy equivalence. 

\begin{definition}
 \label{cofibrant_simplicial_space}
 For a simplicial space $X$, define the \emph{$n$-th latching space}
 $L_nX$ by
 \[
  L_nX = \bigcup_{i=0}^{n-1} s_i(X_{n-1}).
 \]
 A simplicial space $X$ is said to be \emph{cofibrant} if
 the inclusion $i_n:L_nX\hookrightarrow X_n$, the \emph{$n$-th
 latching map}, is a closed cofibration for all $n$. 
\end{definition}

\begin{remark}
 For a pair $(X,A)$ of topological spaces, the inclusion
 $A\hookrightarrow X$ is a closed cofibration if and only if $(X,A)$ is
 an \emph{NDR pair}, i.e.\ there exist a continuous function
 $u : A \to [0,1]$ and a homotopy $h : X\times[0,1]\to X$ such that
 $A=u^{-1}(0)$, $h(x,t)=(x,t)$ for
 $(x,t)\in X\times\{0\}\cup A\times[0,1]$,  and $h(x,1)\in A$ for
 $x\in h^{-1}([0,1))$.
\end{remark}

Recall that one of the most important examples of closed cofibrations is
an inclusion of a subcomplex in a CW complex.

\begin{example}
 \label{CW_simplicial_space}
 If $X$ is a simplicial space consisting of CW complexes in which
 the latching space $L_nX$ is a subcomplex of $X_n$ for all $n$, then
 $X$ is cofibrant.

 In particular, for a $2$-category $C$ the nerve $NBC$ of the associated 
 topological category $BC$ is a cofibrant simplicial space.
\end{example}

\begin{proposition}
 When $X$ is a cofibrant simplicial space, the quotient map
 $q_{X}:\|X\|\to |X|$ is a homotopy equivalence.
\end{proposition}

\begin{proof}
 See also Appendix A of Segal's paper \cite{Segal74-1}.
\end{proof}

\begin{corollary}
 \label{homotopy_invariance_of_realization}
 Let $f:X\to Y$ be a map of simplicial spaces. Suppose $X$ and $Y$ are
 cofibrant and that $f_n : X_n\to Y_n$ is a homotopy equivalence for all
 $n$. Then the induced map $|f|: |X|\to |Y|$ is a homotopy equivalence. 
\end{corollary}

\subsection{Nerves and Classifying Spaces}
\label{nerve}

The nerve construction transforms small categories to simplicial sets,
and then to topological spaces by taking the geometric realization
functor. The nerve construction can be extended to $2$-categories and
topological categories by using simplicial spaces.

\begin{definition}
 \label{nerve_definition}
 For a small category $C$, define
 \[
  N_n(C) = \set{(u_n,\ldots,u_1)\in C_1^n}{s(u_n)=t(u_{n-1}),\ldots,
 s(u_2)=t(u_1)}. 
 \]
 Elements of $N_n(C)$ are called \emph{$n$-chains}.
\end{definition}

\begin{lemma}
 For a small category $C$, these sets $\{N_n(C)\}_{n\ge 0}$ together
 with maps
 \begin{eqnarray*}
  d_i & : & N_n(C) \longrightarrow N_{n-1}(C) \\
  s_i & : & N_n(C) \longrightarrow N_{n+1}(C)
 \end{eqnarray*}
 defined by 
 \begin{eqnarray*}
  d_i(u_n,\cdots, u_1) & = &
   \begin{cases}
    (u_n,\ldots, u_2) & i=0 \\
    (u_n,\ldots, u_{i+2},u_{i+1}\circ
   u_i,\ldots, u_1), & 1\le i\le n-1 \\
    (u_{n-1},\ldots, u_1), & i=n,
   \end{cases}
   \\
  s_i(u_n,\cdots, u_1) & = & \begin{cases}
			      (u_n,\ldots,u_1,1_{s(u_1)}), & i=0 \\
			      (u_n,\ldots, u_{i+1}, 1_{t(u_i)},
			      u_i,\ldots, u_1) & 1\le i\le n
			     \end{cases}
 \end{eqnarray*}
 form a simplicial set.
\end{lemma}

\begin{definition}
 The simplicial set $N(C)$ is called the \emph{nerve} of $C$. 
 The geometric realization of the nerve $N(C)$ is denoted by
 $BC$ and is called the \emph{classifying space} of $C$.
\end{definition}

\begin{remark}
 When $C$ is a poset, $BC$ agrees with the geometric realization of the
 order complex of $C$.
\end{remark}

The following fact is fundamental.

\begin{lemma}
 \label{nerve_as_functor}
 Regard the totally ordered set $[n]=\{0<1<\cdots<n\}$ as a small
 category. Then, for a small category $C$, we have a natural isomorphism
 of sets 
 \[
  N_n(C) \cong \Cats([n],C).
 \]
 Furthermore, under the identification between simplicial sets and
 functors $\Delta^{\op}\to\Sets$, the nerve construction
 defines a functor
\[
 N : \Cats \longrightarrow \Sets^{\Delta^{\op}},
\]
 where $\Sets^{\Delta^{\op}}$ is the category of simplicial sets.
\end{lemma}


Thus the classifying space construction is also a functor
\[
 B : \Cats \rarrow{N} \Sets^{\Delta^{\op}} \rarrow{|-|} \Spaces.
\]
where $|-|$ denotes the geometric realization functor. In particular,
any functor $f : C\to D$ induces a continuous map of topological 
spaces $Bf : BC \to BD$.

A natural transformation between functors induces a homotopy. 

\begin{lemma}
 \label{natural_transformation_is_homotopy}
 A natural transformation $\theta : f\Rightarrow g$ between
 functors $f,g : C\to D$ induces a homotopy
 \[
 B\theta : BC\times [0,1] \longrightarrow BD
 \]
 between $Bf$ and $Bg$.  
\end{lemma}

As a corollary, we obtain the following well-known but useful fact.

\begin{corollary}
 \label{BC_is_contractible}
 Suppose a functor $f: C\to D$ between small categories has a right or a
 left adjoint $g:D\to C$. Then $Bf: BC\to BD$ is a homotopy
 equivalence. In particular, when $C$ has an initial or a terminal
 object, $BC$ is contractible.
\end{corollary}

\begin{proof}
 Suppose $g$ is right adjoint to $f$. We have natural transformations 
 \begin{eqnarray*}
  1_{C} & \Longrightarrow & g\circ f \\
  f\circ g & \Longrightarrow & 1_{D}.
 \end{eqnarray*}
 By Lemma \ref{natural_transformation_is_homotopy}, they induce homotopies
 \begin{eqnarray*}
  Bg\circ Bf & = & B(g\circ f) \simeq B1_{C} = 1_{BC} \\
  Bf\circ Bg & = & B(f\circ g) \simeq B1_{D} = 1_{BD}
 \end{eqnarray*}
 and thus $Bg$ is a homotopy inverse of $Bf : BC\to BD$.

 When $C$ has an initial object $*$, the collapsing functor $p:C\to *$
 is a right adjoint to the inclusion $*\hookrightarrow C$. And we have
 $BC\simeq B*=*$. 
\end{proof}

Furthermore we obtain a deformation retraction in the following special
case. 

\begin{definition}
 \label{closure_operator_definition}
 A map of poset $f: P\to P$ is called a \emph{descending closure
 operator} if $f\circ f= f$ and $f(x)\le x$ for all $x\in P$. Dually it
 is called an \emph{ascending closure operator} if $f\circ f=f$ and
 $f(x)\ge x$ for all $x \in P$.
\end{definition}

\begin{corollary}
 \label{DR_by_closure_operator}
 If a poset $P$ has a descending or an ascending closure operator
 $f:P \to P$, $B(f(P))$ is a strong deformation retract of $BP$.
\end{corollary}

\begin{proof}
 Regard $P$ as a small category. A poset map $f:P\to P$ is an
 endofunctor.  
 Let $i: f(P)\hookrightarrow P$ be the inclusion. When $f$ is a
 descending closure operator, we have $f\circ i = 1_{f(P)}$. Furthermore
 the relation $f(x)\le x$ gives rise to a natural transformation
 $i\circ f \Rightarrow 1_{P}$, which is identity on $f(P)$. Thus we have
 a homotopy $Bi\circ Bf\simeq 1_{BP}$ which is identity on $Bf(P)$.

 When $f$ is a descending closure operator, we have a natural
 transformation $1_{P}\Rightarrow i\circ f$ and the same conclusion holds.
\end{proof}

\begin{remark}
 It can be proved that $BP$ collapses onto $B(f(P))$ under the same
 assumption. See \S13.2 of Kozlov's book
 \cite{KozlovCombinatorialAlgebraicTopology}. 
\end{remark}

When $C$ is acyclic, we only need nondegenerate chains.

\begin{definition}
 For a small category $C$, define
 \[
  \overline{N}_n(C) = N_n(C)-\bigcup_{i=0}^{n-1}s_i(N_{n-1}(C)).
 \]
 Elements of $\overline{N}_n(C)$ are called \emph{nondegenerate
 $n$-chains}. 
\end{definition} 

\begin{lemma}
 When $C$ is acyclic, the face operators $d_i$ can be restricted to
 nondegenerate chains. Thus the collection
 $\overline{N}(C)=\{\overline{N}_n(C)\}_{n\ge 0}$ forms a
 $\Delta$-set. 
\end{lemma}

\begin{definition}
 For an acyclic category $C$, elements of $\overline{N}_n(C)$ are called
 \emph{nondegenerate $n$-chains} and the $\Delta$-set $\overline{N}(C)$
 is called the \emph{nondegenerate nerve} of $C$.
\end{definition}

\begin{lemma}
 For a small acyclic category $C$, we have a natural homeomorphism
 \[
  BC = |N(C)| \cong \|\overline{N}(C)\|,
 \]
 where $\|-\|$ denotes the geometric realization of $\Delta$-sets.
\end{lemma}

It is straightforward to extend the construction of the classifying space
to topological categories. 

\begin{lemma}
 For a topological category $C$, $NC=\{N_n(C)\}_{n\ge 0}$ forms a
 simplicial space.
\end{lemma}

\begin{definition}
 For a topological category $C$, the geometric realization of the
 simplicial space $NC$ is called the \emph{classifying space} of $C$ and
 is denoted by $BC$.
\end{definition}

There are several ways to define the classifying space of a
$2$-category. A good reference is the paper \cite{0903.5058} by
Carrasco, Cegarra, and Garz{\'o}n in which they compare ten different 
nerves for bicategories (weak $2$-categories).

One of the most classical constructions is the following bisimplicial
nerve.

\begin{definition}
 For a $2$-category $C$, define a topological category $BC$ by
 $(BC)_0=C_0$ and 
 \[
  (BC)(x,y) = B(C(x,y))
 \]
 for $x,y \in C_0$. The classifying space of this topological category is
 denoted by $B^2C$ and is called the \emph{classifying space} of $C$.
\end{definition}

$B^2$ obviously defines a functor
\[
 B^2 : \Cats_2 \rarrow{} \Spaces
\]
from the category $\Cats_2$ of $2$-categories and $2$-functors to the one
of topological spaces and continuous maps.

\begin{proposition}
 \label{homotopy_invariance_of_classifying_space}
 If a strict $2$-functor $f: C\to D$ induces a homotopy equivalence
 \[
  Bf(x,y): BC(x,y) \rarrow{} BD(f(x),f(y))
 \]
 for all pairs of objects $x,y\in C_0$, then 
 \[
  B^2f : B^2C \rarrow{} B^2D
 \]
 is a homotopy equivalence.
\end{proposition}

\begin{proof}
 By assumption, $f$ induces a homotopy equivalence
 \[
  N_nBf : N_nBC \rarrow{} N_nBD
 \]
 for all $n$. Now the result follows from 
 Lemma \ref{homotopy_invariance_of_realization} and Example
 \ref{CW_simplicial_space}. 
\end{proof}

Note that the functor $B^2 : \Cats_2 \to \category{Top}$,
cannot be extended to lax or colax functors. In order to construct a
classifying space which is functorial with respect to lax or colax
functors, one of the ways is to modify the description in Lemma
\ref{nerve_as_functor}. 

\begin{definition}
 \label{lax_nerve}
 For a $2$-category $C$, define 
 \begin{eqnarray*}
  N_n^{cl}(C) & = & \set{\bm{u}:[n]\to C}{\text{colax functors}} \\  
  N_n^{ncl}(C) & = & \set{\bm{u}:[n]\to C}{\text{normal colax functors}} \\  
  N_n^{l}(C) & = & \set{\bm{u}:[n]\to C}{\text{lax functors}} \\  
  N_n^{nl}(C) & = & \set{\bm{u}:[n]\to C}{\text{normal lax functors}}.  
 \end{eqnarray*}
 \end{definition}

\begin{lemma}
 \label{lax_classifying_spaces}
 The collections $N^{cl}(C)=\{N_n^{cl}(C)\}_{n\ge 0}$,
 $N^{ncl}C=\{N_n^{ncl}(C)\}_{n\ge 0}$,
 $N^{l}(C)=\{N_n^{l}(C)\}_{n\ge 0}$, and
 $N^{nl}(C)=\{N_n^{nl}(C)\}_{n\ge 0}$ form simplicial
 sets and we obtain functors
 \begin{eqnarray*}
  B^{cl} & : & \Cats_{2,cl} \longrightarrow \Sets^{\Delta^{\op}} \\
  B^{ncl} & : & \Cats_{2,ncl} \longrightarrow \Sets^{\Delta^{\op}} \\
  B^{l} & : & \Cats_{2,l} \longrightarrow \Sets^{\Delta^{\op}} \\
  B^{nl} & : & \Cats_{2,nl} \longrightarrow \Sets^{\Delta^{\op}}.
 \end{eqnarray*}
\end{lemma}

\begin{remark}
 The above construction is originally due to Duskin \cite{Duskin02} in
 the case of lax functors. See
 also a paper \cite{Street96} by Street.
\end{remark}

It is known that all of these constructions give rise to the same
homotopy type. 

\begin{theorem}
 \label{classifying_space_of_2-category}
 We have weak homotopy equivalences
 \[
  B^{cl}C \simeq B^{ncl}C \simeq B^2C \simeq B^{nl}C \simeq B^{l}C.
 \]
\end{theorem}

\begin{proof}
 See the paper \cite{0903.5058} by Carrasco, Cegarra, and Garz{\'o}n. 
\end{proof}
\subsection{Comma Categories and Quillen's Theorem A}
\label{QuillenAB}

Given a functor $f : C\to D$ between small categories, a basic question
is when the induced map
\[
 Bf : BC \longrightarrow BD
\]
is a homotopy equivalence. In order to measure the difference between
$C$ and $D$ via $f$, one might want to look at fibers of $f$.

\begin{definition}
 \label{fiber_of_functor}
 For a functor $f:C\to D$, the \emph{fiber} $f^{-1}(y)$ over $y\in D_0$
 is defined to be the subcategory of $C$ whose sets of objects and
 morphisms are given by 
 \begin{eqnarray*}
  f^{-1}(y)_0 & = & f_0^{-1}(y) = \set{x\in
   C_0}{f_0(x)=y} \\
  f^{-1}(y)_1 & = & f_1^{-1}(1_y) = \set{u\in C_1}{f_1(u)=1_y}.
 \end{eqnarray*}
\end{definition}

Unfortunately $Bf$ fails to be a fibration in general, and genuine
fibers do not tell us differences. The standard
technique in homotopy theory suggests to take homotopy fibers.
Quillen \cite{Quillen73} found that homotopy fibers of $Bf$ can be
described in terms of comma categories.

\begin{definition}
 \label{comma_category_definition}
 For a small category $C$ and an object $x\in C_0$, define a category
 $C\downarrow x$ as follows. Objects are morphisms in $C$ of the form
 $u:y\to x$:
 \[
  (C\downarrow x)_0 = \set{u\in C_1}{t(u)=x} = t^{-1}(x).
 \]
 A morphism from $u:y\to x$ to $v: z\to x$ is a morphism
 $w:y\to z$ in $C$ satisfying $u=v\circ w$
 \[
  (C\downarrow x)(u,v) = \set{w\in C(y,z)}{u=v\circ w}.
 \]
 Compositions of morphisms are given by compositions in $C$. Dually we
 define a category 
 $x\downarrow C$ whose objects are morphisms 
 in $C$ of the form $u:x\to y$. A morphism from $u: x\to y$ to
 $v:x\to z$ is a morphism $w : y\to z$ satisfying $v=w\circ u$.

 More generally, given a functor $f : C\to D$ and an object $y\in D_0$,
 define a category $f\downarrow y$ as follows. The set of objects is
 given by 
 \[
  (f\downarrow y)_0 = \set{(x,u)\in C_0\times D_1}{u\in  D(f(x),y)}.
 \]
 The set of morphisms from $(x,u)$ to $(x',u')$ is given by
 \[
  (f\downarrow y)((x,u),(x',u')) = \set{w\in C(x,x')}{u=u'\circ f(w)}.
 \]
 Dually define a category $y\downarrow f$ by
 \begin{eqnarray*}
  (y\downarrow f)_0 & = & \set{(u,x)\in D_1\times C_0}{u\in  C(y,f(x))} \\
  (y\downarrow f)((u,x),(u',x')) & = & \set{w\in C(x,x')}{u'=f(w)\circ u}.
 \end{eqnarray*}

 The categories $f\downarrow x$ and $x\downarrow f$ are called
 \emph{left and right homotopy fibers of $f$ at $x$}, respectively.
\end{definition}

\begin{remark}
 The categories $f\downarrow x$ and $x\downarrow f$ are often called
 \emph{comma categories}. The terminology used above is based on a
 homotopy-theoretic point of view.
 The name comma category actually refers to a more general
 construction. Categories defined in Definition
 \ref{comma_category_definition} are sometimes called
 \emph{coslice/slice categories} and \emph{under/over categories}. 

 Note that different authors use different notations 
 for these categories. The notation $f\downarrow x$ and $x\downarrow f$
 can be found in Mac\,Lane's book \cite{MacLaneCategory}. 
 Quillen \cite{Quillen73} denotes them by $C/x$, $x\backslash C$, 
 $f/y$, and $y\backslash f$, respectively.
\end{remark}

%
%

The following result is due to Quillen \cite{Quillen73} and is called
Quillen's Theorem A.

\begin{theorem}[Theorem A]
 \label{TheoremA}
 For a functor $f : C \to D$ between small categories, if
 $B(f\downarrow y)$ is contractible for all $y\in D_0$, then  
 $Bf : BC\to BD$ is a homotopy equivalence.

 The same is true when $B(y\downarrow f)$ is contractible for all
 $y\in D_0$.
\end{theorem}

This fact says that the family of spaces $B(f\downarrow y)$ measures the
difference of $BC$ and $BD$ via $Bf$ in the homotopy category of
topological spaces. 

We close this section by describing relations between genuine fibers and
homotopy fibers.

\begin{definition}
 For a functor $f:C\to D$ between small categories and an object
 $y\in D_0$, define functors
 \begin{eqnarray*}
  i_{y} & : & f^{-1}(y) \rarrow{} y\downarrow f \\
  j_{y} & : & f^{-1}(y) \rarrow{} f\downarrow y
 \end{eqnarray*}
 by 
 \begin{eqnarray*}
  i_y(x) & = & (1_y,x) \\
  j_y(x) & = & (x,1_y)
 \end{eqnarray*}
 on objects.

 When $i_{y}$ has a right adjoint $s_y$ for each object $y$ in
 $D$, the functor $f$ is called \emph{prefibered} or a \emph{prefibered
 category}. The collection 
 $\{s_y\}_{y\in D_0}$ is called a \emph{prefibered structure} on $f$.  
 Dually when $j_{y}$ has a left adjoint $t_x$ for each object $y$ in
 $D$, the functor $f$ is called \emph{precofibered} or a
 \emph{precofibered category}. The
 collection $\{t_y\}_{y\in D_0}$ is called a \emph{precofibered
 structure} on $f$. 
\end{definition}

\begin{corollary}
 \label{prefibered_QuillenA}
 Suppose $f:C\to D$ is either prefibered or precofibered. If
 $Bf^{-1}(y)$ is contractible for each $y\in D_0$, then $Bf:BC\to BD$ is
 a homotopy equivalence.
\end{corollary}

\subsection{Comma Categories and Quillen's Theorem A for $2$-Categories}
\label{2-QuillenAB}

We need to extend Theorem \ref{TheoremA} to poset-categories and, more
generally, to $2$-categories for our purposes.
We first need to define homotopy fibers or comma categories for colax
functors. 
According to Cegarra \cite{0909.4229}, comma categories for 2-functors
were introduced by Gray \cite{J.Gray1980}, whose extension to colax
functors can be found in del Hoyo's paper \cite{1005.1300}\footnote{Note
that colax functors in the sense of this paper are called ``lax
functors'' in del Hoyo's paper.}. 

\begin{definition}
 Let $C$ and $D$ be $2$-categories. 
 \begin{enumerate}
  \item For a colax functor $f : C\to D$ and an object $y\in D_0$,
	define a $2$-category $f\downarrow y$ as follows.
	\begin{enumerate}
	 \item The set of objects is given by 
	       \[
		(f\downarrow y)_0 = \set{(x,u)\in C_0\times D_1}{u\in
		 D(f(x),y)_0}.
	       \]

	 \item For $(x,u), (x',u') \in (f\downarrow y)_0$, the sets of
	       objects and morphisms in the category
	       $(f\downarrow y)((x,u),(x',u'))$ are defined by
	       \begin{eqnarray*}
		(f\downarrow y)((x,u),(x',u'))_0 & = &
		 \set{(\theta,w)\in D_2\times 
		 C_1}{\theta :  u'\circ f(w) \Rightarrow u} \\
		(f\downarrow y)((x,u),(x',u'))((\theta,w),(\theta',w'))
		 & = & \set{\varphi\in
		 C(x,x')(w',w)}{\theta'=\theta\ast(u'\circ f(\varphi))}.  
	       \end{eqnarray*}
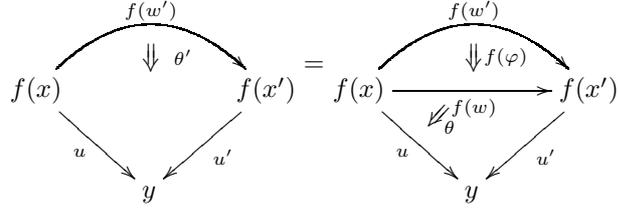
	       \[
	       \xymatrix{
	       f(x)
	       \rruppertwocell<10>^{f(w')}{\hspace*{12pt}\theta'}
	       \ar[dr]_{u} & & f(x') \ar[dl]^{u'} \\
	       & y &
	       }
	       \raisebox{10pt}{=}
		\xymatrix{
	       f(x) \ar[rr]_{f(w)}
	       \rruppertwocell<10>^{f(w')}{\hspace*{12pt}f(\varphi)}
	       \ar[dr]_{u} \druppertwocell<\omit>{<-2.5>\theta} & &
	       f(x') \ar[dl]^{u'} \\  
	       & y & 
	       }
	       \]

	 \item For an object $(x,u)\in (f\downarrow y)_0$, define 
	       $1_{(x,u)}= (u'\circ f_x, 1_{x})$.
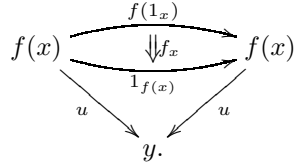
	       \[
	       \xymatrix{
	       f(x) \ar[dr]_{u} \rrtwocell^{f(1_x)}_{1_{f(x)}}{f_{x}} & &
	       f(x) \ar[dl]^{u} \\
	       & y. &
	       }
	       \]

	 \item The horizontal composition in $f\downarrow y$
	       \begin{equation}
		\circ : (f\downarrow y)((x',u'),(x'',u''))\times
		 (f\downarrow y)((x,u),(x',u')) \rarrow{}
		 (f\downarrow y)((x,u),(x'',u'')) 
		 \label{horizontal_composition_in_comma_2-category}
	       \end{equation}
	       is defined by the following diagram
	       
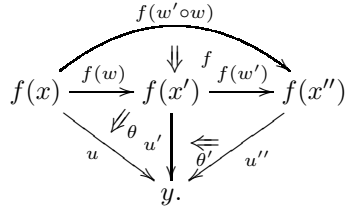
\[
	       \xymatrix{
	       f(x) \ar[r]^{f(w)} \ar[dr]_{u}
	       \rruppertwocell<10>^{f(w'\circ w)}{\hspace*{12pt}f}
	       \drtwocell<\omit>{<-2>\theta}
	       & f(x') \ar[r]^{f(w')}
	       \ar[d]_{u'} \dtwocell<\omit>{<-2.5>\theta'} & f(x'')
	       \ar[dl]^{u''} \\   
	       & y. &
	       }
	       \]
	 \item The vertical composition in $f\downarrow y$
	       is given by that of $C$. 

	\end{enumerate}

  \item  For a lax functor $f$ and an object $y\in D_0$, define  
	 a $2$-category $y\downarrow f$ as
	 follows. 
	 \begin{enumerate}
	  \item The set of objects is given by
		\[
		(y\downarrow f)_0 = \set{(u,x)\in D_1\times C_0}{u\in
		 D(y,f(x))_0}.
		\]

	  \item For $(u,x), (u',x') \in (y\downarrow f)_0$, the sets of
		objects and morphisms in the category
		$(y\downarrow f)((u,x),(u,x'))$ are defined by 
		\begin{eqnarray*}
		 (y\downarrow f)((u,x),(u',x'))_0 & = &
		  \set{(\theta,w)\in D_2\times 
		  C_1}{\theta :  u' \Rightarrow u\circ f(w)} \\
		 (y\downarrow f)((u,x),(u',x'))((\theta,w),(\theta',w')) & = &
		  \set{\varphi\in
		  C(x,x')(w,w')}{\theta=(f(\varphi)\circ u)\ast
		  \theta'}.  
		\end{eqnarray*}
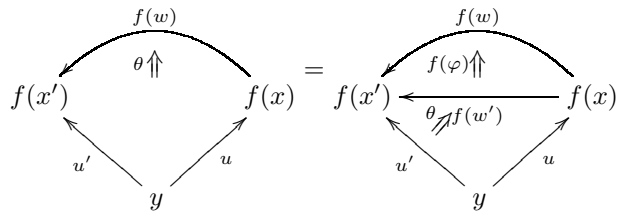
	       \[
	       \xymatrix{
	       f(x') & & f(x)
		\lllowertwocell<-10>_{f(w)}{\theta} \\
	       & y \ar[ul]^{u'} \ar[ur]_{u} &
	       }
	       \raisebox{10pt}{=}
		\xymatrix{
	       f(x')
	       & &
	       f(x) \ar[ll]^{f(w')}
	       \lllowertwocell<-10>_{f(w)}{\hspace*{-10pt}f(\varphi)}
		\\  
	       & y \ar[ul]^{u'} \uluppertwocell<\omit>{<2.5>\theta} 
	        \ar[ur]_{u} & 
	       }
	       \]

	 \item For an object $(u,x)\in (y\downarrow f)_0$, define 
	       $1_{(u,x)}=(f_x\circ u,1_{x})$.
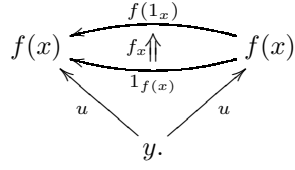
	       \[
	       \xymatrix{
	       f(x) & &
	       f(x) \lltwocell^{1_{f(x)}}_{f(1_x)}{f_{x}} \\
	       & y. \ar[ul]^{u} \ar[ur]_{u} &
	       }
	       \]

	  \item The horizontal composition in $y\downarrow f$
	       \[
		\circ : (y\downarrow f)((u',x'),(u'',x''))\times
		 (y\downarrow f)((u,x),(u',x')) \rarrow{}
		 (y\downarrow f)((u,x),(u'',x'')) 
	       \]
	       is defined by the following diagram
	       
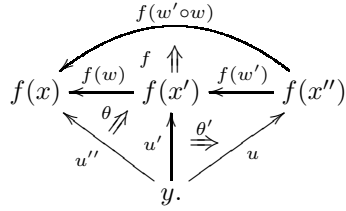
\[
	       \xymatrix{
	       f(x) & f(x') \ar[l]_{f(w)} & f(x'') \ar[l]_{f(w')}
	       \lllowertwocell<-10>_{f(w'\circ w)}{\hspace*{-10pt}f}
	       \\   
	       & y.  \ar[ul]^{u''}
	       \ultwocell<\omit>{<2>\theta}
	       \utwocell<\omit>{<2.5>\theta'} 
	       \ar[u]^{u'} \ar[ur]_{u} &
	       }
	       \]

	 \item The vertical composition in $y\downarrow f$
	       is given by that of $C$. 

	 \end{enumerate}

 \end{enumerate}
\end{definition}

\begin{remark}
 \label{homotopy_fiber_for_normal_lax_functor}
 For a colax functor $f$, $y\downarrow f$ cannot be defined, since
 the direction of $f_x: f(1_{x})\Rightarrow 1_{f(x)}$ does not allow us
 to define an identity morphism on $(u,x)$. This is the only obstruction
 to defining $y\downarrow f$. When $f$ is normal,
 therefore, $y\downarrow f$ is defined.
 For example, a $1$-morphism from $(u,x)$ to $(u',x')$ in
 $y\downarrow f$ is given by a pair $(\theta,w)\in D_2\times C_1$ with
 \[
  \theta : f(w)\circ u \Longrightarrow u'.
 \]
 Horizontal compositions are given by the following diagram
 
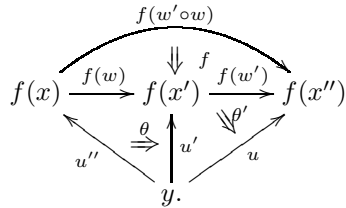
\[
 \xymatrix{
 f(x) \ar[r]^{f(w)} 
 \rruppertwocell<10>^{f(w'\circ w)}{\hspace*{10pt}f}
 & f(x') \ar[r]^{f(w')} & f(x'') \\   
 & y.  \ar[ul]^{u''}
 \uuppertwocell<\omit>{<-2>\theta}
 \uruppertwocell<\omit>{<-2>\theta'} 
 \ar[u]_{u'} \ar[ur]_{u} &
 }
 \]
 Similarly, for a normal lax
 functor, $f\downarrow y$ can be defined.
\end{remark}

\begin{lemma}
 For a colax functor, the above data define a $2$-category
 $f\downarrow y$. When $f$ is a lax functor, $y\downarrow f$ becomes a
 $2$-category. 
\end{lemma}

\begin{proof}
 Consider the case of a colax functor.
 We need to verify the following:
 \begin{enumerate}
  \item $(f\downarrow y)((x,u),(x',u'))$ is a category for each pair
	$(x,u),(x',u')\in (f\downarrow y)_0$.
  \item The horizontal composition
	(\ref{horizontal_composition_in_comma_2-category}) is a
	functor. 
  \item The horizontal composition is associative.
  \item $1_{(x,u)}$ serves as a unit for each
	$(x,u)\in (f\downarrow y)_0$. 
 \end{enumerate}

 Note that the set $(f\downarrow y)((x,u),(x',u'))_1$ is a subset of
 $C(x,x')_1\subset C_2$. It can be easily verified that
 $(f\downarrow y)((x,u),(x',u'))_1$ is closed under compositions. Thus
 $(f\downarrow y)((x,u),(x',u'))$ is a subcategory of $C(x,x')_1$.

 The fact that (\ref{horizontal_composition_in_comma_2-category}) is a
 functor follows from the naturality of
 $f_{x'',x',x'}: f(w'\circ w)\Rightarrow f(w')\circ f(w)$.
 The associativity for the horizontal composition follows from the
 commutativity of (\ref{composer}). 

 The fact that $1_{(x,u)}$ is a unit
 follows from the commutativity of (\ref{uniter1}) and (\ref{uniter2}).
\end{proof}

\begin{definition}
 The $2$-categories $f\downarrow y$ and $y\downarrow f$ are called the
 \emph{left homotopy fiber} and the \emph{right homotopy fiber} of $f$
 over $y$, respectively.
\end{definition}

When $C$ is a $1$-category regarded as a $2$-category, the homotopy
fiber $f\downarrow y$ has a simpler description, since $2$-morphisms in
$f\downarrow y$ are defined by $2$-morphisms in $C$.

\begin{lemma}
 \label{homotopy_fiber_in_1-category}
 Let $f : C \to D$ be a colax functor and suppose that $C$ is a
 $1$-category. Then the left homotopy fiber $f\downarrow y$ is a 
 $1$-category for any object $y\in D_0$.
\end{lemma}

\begin{corollary}
\label{homotopy_fiber_poset}
 Let $f : C\to D$ be a colax functor from a poset $C$ to a
 poset-category $D$. 
 Then the left homotopy fiber $f\downarrow y$ is a poset
 whose order relation $\le_y$ is defined by
 \[
 (x,u)\le_y (x',u') \Longleftrightarrow x\le x' \text{ in } C
 \text{ and } u'\circ f(x\le x')\le u \text{ in } D(f(x'),y). 
 \] 

 When $f$ is normal, the right homotopy fiber $y\downarrow f$  is
 also a poset whose order relation $\le_y$ is defined by
 \[
 (u,x)\le_y (u',x') \Longleftrightarrow x\le x' \text{ in } C
 \text{ and }
 f(x\le x')\circ u \le u' \text{ in } D(y,f(x')). 
 \] 
\end{corollary}

\begin{proof}
 A morphism from $(x,u)$ to $(x',u')$ in $f\downarrow y$ is a pair
 $(\theta,w)$ with 
 $\theta : u'\circ f(w)\Rightarrow u$. Here $w$ is a morphism
 $w:x\to x'$ in $C$ and $\theta$ is a morphism in $D(f(x),y)$. Since
 $C$ is a poset, such $w$ exists if and only if $x\le x'$. $D(f(x),y)$
 is also a poset and thus $\theta$ exists if and only if
 $u'\circ f(x\le x')\le u$. Furthermore such $w$ and $\theta$ are unique
 if exist. Thus $f\downarrow y$ is a poset.

 When $f$ is normal, $y\downarrow f$ is defined. A morphism from $(u,x)$
 to $(u',x')$ in $y\downarrow f$ is a pair $(\theta,w)$ with
 $\theta:f(w)\circ u\Rightarrow u'$. Since both $C$ and $D(y,f(x'))$ are
 posets, such morphisms are unique and thus $y\downarrow f$ is a poset.  
\end{proof}

An analogue of $f^{-1}(x)$ can be defined as follows.

\begin{definition}
 For a colax functor $f : C\to D$ between $2$-categories and an object
 $y\in D_0$, define a $2$-subcategory $f^{-1}(y)$ of $C$ by
 \begin{eqnarray*}
  f^{-1}(y)_0 & = & \set{x\in C_0}{f(x)=y} \\
  f^{-1}(y)_1 & = & \set{u\in C_1}{f(u)=1_{y}} \\
  f^{-1}(y)_2 & = & \set{\theta\in C_2}{f(\theta)=1_{1_y}}.
 \end{eqnarray*}
 Define
 \[
  j : f^{-1}(y) \rarrow{} f\downarrow y
 \]
 by $j(x) = (x,1_{y})$ on objects. For a $1$-morphism $w:x\to x'$ in
 $f^{-1}(y)$, define a $1$-morphism in $f\downarrow y$ by the pair
 $(1_{1_{y}},w)$. $j$ sends a $2$-morphism $\varphi$ in $f^{-1}(x)$ to
 $\varphi$. 
\end{definition}

A generalization of Quillen's Theorem A to colax functors was
proved by del Hoyo \cite{1005.1300} based on the work
\cite{Bullejos-Cegarra03} of Bullejos and Cegarra on Theorem A for
2-functors.

\begin{theorem}[Theorem A for colax functors]
 \label{lax_TheoremA}
 Let $f : C\to D$ be a colax functor between $2$-categories. If
 $B^{cl}(f\downarrow y)$ is contractible for every object
 $y\in D_0$, then the induced map
 \[
  B^{cl}f : B^{cl}C \longrightarrow B^{cl}D
 \]
 is a homotopy equivalence. When $f$ is normal, the induced map
 \[
  B^{ncl}f : B^{ncl}C \rarrow{} B^{ncl}D
 \]
 is a homotopy equivalence as long as $B^{ncl}(f\downarrow y)$ is
 contractible for every object $y\in D_0$.
\end{theorem}

We need the following version of a $2$-categorical analogue of Corollary
\ref{prefibered_QuillenA}. Recall from Lemma
\ref{homotopy_fiber_in_1-category} that homotopy fibers $f\downarrow y$
and $y\downarrow f$ are $1$-categories, if the domain category $C$ of
$f:C\to D$ is a $1$-category. The fiber $f^{-1}(y)$ is also a
$1$-category. 

\begin{definition}
 Let $f:C\to D$ be a colax functor from a small $1$-category to a small
 $2$-category. We say $f$ is \emph{precofibered} if the canonical
 inclusion $j_{y}: f^{-1}(y)\rarrow{} f\downarrow y$ has a left
 adjoint for each $y\in D_0$. Dually a lax functor $f:C\to D$ is said to
 be \emph{prefibered} if the inclusion $i_y: f^{-1}(y)\to y\downarrow f$
 has a right adjoint for each $y\in D_0$.
\end{definition}

\begin{corollary}
 \label{prefibered_2QuillenA}
 Let $f:C\to D$ be either a precofibered colax functor or a prefibered
 lax functor from a small $1$-category $C$ to a small $2$-category
 $D$. If $Bf^{-1}(y)$ is contractible for each $y\in D_0$, 
 $Bf:BC\to B^{cl}D$ or $Bf:BC\to B^{l}D$ is a homotopy equivalence.
 When $f$ is a prefibered normal colax functor or a precofibered normal
 lax functor, $Bf:BC\to B^{ncl}D$ or $Bf:BC\to B^{nl}D$ is a homotopy
 equivalence, respectively. 
\end{corollary}

\bibliographystyle{halpha}
\bibliography{%
mathA,%
mathB,%
mathC,%
mathD,%
mathE,%
mathF,%
mathG,%
mathH,%
mathI,%
mathJ,%
mathK,%
mathL,%
mathM,%
mathN,%
mathO,%
mathP,%
mathQ,%
mathR,%
mathS,%
mathT,%
mathU,%
mathV,%
mathW,%
mathX,%
mathY,%
mathZ,%
preprintA,%
preprintB,%
preprintC,%
preprintD,%
preprintE,%
preprintF,%
preprintG,%
preprintH,%
preprintI,%
preprintJ,%
preprintK,%
preprintL,%
preprintM,%
preprintN,%
preprintO,%
preprintP,%
preprintQ,%
preprintR,%
preprintS,%
preprintT,%
preprintU,%
preprintV,%
preprintW,%
preprintX,%
preprintY,%
preprintZ}

\end{document}